\documentclass{article}
\usepackage[T2A]{fontenc}
\usepackage[cp1251]{inputenc}
\usepackage[english,russian]{babel}
\usepackage[tbtags]{amsmath}
\usepackage{amsfonts,amssymb,mathrsfs,amscd}
\usepackage{graphics}
\usepackage{euscript,textcomp,verbatim,fancyhdr}
\usepackage{latexsym}
\usepackage[hyper]{msb-a}
\JournalName{}
\numberwithin{equation}{section}
\theoremstyle{plain}
\newtheorem{Pro}{Proposition}[section]
\newtheorem{Cor}{Corollary}[section]
\newtheorem{The}{Theorem}[section]
\newtheorem{Lem}{Lemma}[section]
\theoremstyle{definition}
  \newtheorem{scRem}{Remark}[section]
  \newtheorem{scDef}{Definition}[section]
  \newtheorem{Exa}{Example}[section]     
  \newtheorem{nnExa}{Example}[section]
\newtheorem{proof}{Proof}
\newcommand{\qed}{\hfill$\square$}

   \newcommand{\Sec}{{\S}} 
   \renewcommand{\(}{\textup(}
   \renewcommand{\)}{\textup)}
   
\newcommand{\bbZ}{{\mathbb Z}}
\newcommand{\bbR}{{\mathbb R}}
\newcommand{\bbQ}{{\mathbb Q}}
\newcommand{\bbC}{{\mathbb C}}
\newcommand{\bbN}{{\mathbb N}}
\newcommand{\bbS}{{\mathbb S}}
   \newcommand{\bA}{\mathbf A}
   \newcommand{\bB}{\mathbf B}
\newcommand \cM {{\mathcal M}}
   \DeclareMathOperator {\pr} {pr}
\def \const {{\rm const}}
\def \Tr {{\rm Tr}\,}
\def \Ker {{\rm Ker}\,}
\def \Span {{\rm span}}       
\def \grav {{g}}              
\def \m {{\mu}}               
\def \Kin {{G}}               
\def \pnt {{x}}               
\renewcommand \o {^{\circ}}   
\def \somega {{\varpi}}       
\def \fomega {{\somega}}      
\def \mod {\,{\rm mod}\,}
\def \bP {{\bf P}}            
\def \V {{\bf V}}             
\def \M {{\bf r}}             
\def \C {{\bf c}}             
\def \x {{\bf x}}             
\def \y {{\bf y}}             
\def \p {{\bf p}}             
\def \q {{\bf q}}             
\def \I {{\bf I}}             
\def \a {{\bf a}}             
\def \b {{\bf b}}             
\def \u {{\bf u}}             
\def \v {{\bf v}}             
\def \J {{\tau}}              

\newcommand \xxi {\boldsymbol{\xi}}         
\newcommand \eeta {\boldsymbol{\eta}}       
\newcommand \oomega {{\omega}}                
\newcommand \vvarphi {\boldsymbol{\varphi}} 
\newcommand \ddelta {\boldsymbol{\delta}}   

\newcommand \mxxi {{\bf X}}        
\newcommand \meeta {{\bf X}}       
\newcommand \mx {{\bf Y}}          
\newcommand \my {{\bf Y}}          

\newcommand \uxxi {{\x}}           
\newcommand \ueeta {{\x}}          
\newcommand \ux {{\y}}             
\newcommand \uy {{\y}}             

\newcommand \pxxi {\xxi}           
\newcommand \peeta {\xxi}          
\newcommand \px {\eeta}            
\newcommand \py {\eeta}            

\newcommand \heeta {\u}            
\newcommand \hx {{\bf V}}          
\newcommand \hy {\v}               

\newcommand \hyi {v}               

\begin{document}

\title{Relatively-periodic solutions of planetary systems with satellites and 
systems with slow and fast variables}
\author{E.\,A.~Kudryavtseva}
\address{Moscow State University}
\email{eakudr@mech.math.msu.su}
\date{}
\udk{}

\maketitle

\begin{fulltext}

\begin{abstract}
The partial case of the planar $N+1$ body problem, $N\ge2$, of the
type of planetary system with satellites is studied. One of the
bodies (the Sun) is assumed to be much heavier than the other bodies
(``planets'' and ``satellites''), moreover the planets are much
heavier than the satellites, and the ``years'' are much longer than
the ``months''. 
The existence of at least $2^{N-2}$ smooth 2-parameter
families of symmetric periodic solutions in a rotating coordinate system 
is proved, such that the distances between each planet and its
satellites are much shorter than the distances between the Sun and
the planets. The existence of ``gaps'' in these families of solutions is proved, 
corresponding to $k:(k+1)$ resonances of angular frequencies of planets' revolution around the Sun.
Generating symmetric periodic solutions are described.
Sufficient conditions for some periodic solutions to be orbitally stable in
linear approximation are given. 
The results are extended to a class of Hamiltonian systems with slow and fast
variables close to the systems of semidirect product type.


\end{abstract}


\markright{Periodic solutions of planetary systems with satellites}

\footnotetext[0]{The work was done in the Lomonosov Moscow State University and supported by 
Russian Science Foundation (project 17-11-01303).}

\hfill{In memory of Nikolai Nikolaevich Nekhoroshev}

\section{Introduction} \label{sec:vved}

We study the partial case of the planar $N+1$ body problem, $N\ge2$,
that can be characterized as ``the problem on the motion of a
planetary system with satellites''.

An effective estimate for the number of smooth two-parameter families
of symmetric periodic solutions of this problem in a rotating
coordinate system is proved (theorems \ref {th1}, \ref {th:ust}(A)
and corollary \ref {cor:3bodies}($\exists$) about ``solutions of the
first kind''). Sufficient conditions for orbital stability in linear
approximation for some of these solutions are given (theorem \ref
{th:ust}(B)). Generating symmetric periodic solutions are described
(theorem \ref {th1}). The necessity of a nondegeneracy condition is
proved (theorem \ref {th:degen:sym} and corollary \ref
{cor:3bodies}($\nexists$)). The periodic solutions under our
investigation are close to collections of independent ``circular''
solutions of the corresponding Kepler problems for each planet and
each satellite. 
Via properties of periodic solutions of the Hill problem in the Lunar theory (Lemma \ref {lem:mult} and Theorem \ref {th:mult}), which were proved in \cite {K:vest13} by means of the averaging method on a submanifold,
the listed results are generalized to a wide class of Hamiltonian systems with slow and fast variables
(theorems \ref {th1'}--\ref {th:nondeg}).

Theorems \ref {th1} and \ref {th:ust} of the present work include as
partial cases the result of F.R.\ Moulton \cite {32} (that in turn generalizes results of G.\ Hill \cite {H,W} and E.W.\ Brown \cite {brown1892} on families of periodic solutions of the Hill problem and the restricted three-body problem, respectively, cf.\ \cite[\S 17--19]{33}) and a result by H.\ Poincar\'e \cite {1} on the existence of periodic solutions and sufficient
conditions of their orbital stability in linear approximation for the
systems of the Sun--Earth--Moon and the Sun--two planets types
(respectively). Theorem \ref {th1} implies the known results by
G.\,A.\ Krassinsky \cite {11} and E.\,A.\ Kudryavtseva \cite
{28,29} on the number of periodic solutions of planetary systems
without and with satellites (respectively), by V.\,N.\ Tkhai \cite
{Thay} on the number and the location of symmetric periodic solutions
of the systems of the Sun--planets and Sun--planet--satellites types.

To be precise, the motions of planetary systems with satellites discovered in this paper are indeed relatively periodic (Definition \ref {def:per}) rather than periodic. This paper does not study orbital stability of these solutions, but it studies a weaker property of them, namely orbital stability in linear approximation (\S\ref{par:ust}).

Our paper studies neither motions of such planetary systems as, e.g., Sun--Jupiter--asteroid, nor motions of such planetary systems with satellites as Sun--Saturn--Mimas--a particle of Saturn's ring (recall, Mimas is Saturn's satellite). The fact is that, in the former system, the asteroid plays the role of a ``small planet'' whose mass is much smaller than the mass of the ``main planet'', the Jupiter, while in the latter system, the particle of Saturn's ring plays the role of its ``small satellite'' whose mass is much smaller than the mass of the ``main satelite'', Mimas. In contrast to these situations, our paper assumes (in order to reduce the number of small parameters) that the masses of all planets have the same order $\mu$ and masses of all satellites have the same order $\mu\nu$ (cf.\ (\ref {A})). That is, our paper does not consider systems with ``small planets'' (asteroids) or ``small satellites'' (like ring's particles). Thus, results of our work are not applicable for explaning such phenomena in the Sun systen as Kirkwood ``hatches'' in the asteroid bell or ``gaps'' in Saturn's ring. Nevertheless, we justify the presence of similar ``gaps'' corresponding to resonances $(k+1):k$ in our planetary systems.

Finally we remark that the question of interpretation of the discovered class of relatively periodic solutions of the $N+1$ body problem in terms of behaviour of planets and satellites of the real solar system is very interesting, needs an additional investigation and is not discussed in this work.

The paper has the following structure. In \Sec\ref {sec:vved'}, the statement of the problem is described, main results are formulated (Theorems \ref {th1}--\ref {th:degen:sym} and Corollary \ref {cor:3bodies}) and the method of constructing periodic solutions is described (Theorems \ref {th1} and \ref {th:ust}). In \Sec\ref {sec:quick:slow}, Theorems \ref {th1'}--\ref {th:nondeg} on periodic solutions of Hamiltonian systems with slow and fast variables are formulated.
In \S \ref {subsec:ideas}, we proof Theorems \ref {th1'} and \ref {th:sym} and describe a sketch of deriving Theorems \ref {th:stab} and \ref {th:nondeg} from \cite {K:vest13} (cf.\ Lemma \ref {lem:mult})
via the averaging method on a submanifold (cf.\ Lemma \ref {lem:mult} and Theorem \ref {th:mult})
and the method of generating function (Definition \ref {def:gen:func}). 
In \Sec\ref {par3:1}, the notions of a relatively periodic solution, a symmetric solution and a solution orbitally stable in linear approximation (Definition \ref {def5}), as well as the nondegeneracy conditions are discussed. In \Sec\ref {par3:1:5}, we construct a family of relatively periodic solutions of the unperturbed system, close to ``generating'' solutions. Along the way, we introduce coordinates in the phase space of the model (respectively, unperturbed or perturbed) problem, that bring the system to the form of the same name system in \Sec \ref {sec:quick:slow} (Lemma \ref {lem3:2:}), becides we rediscover two families of periodic solutions of the Hill problem (discovered by G.~Hill) via the method of averaging on a submanifold (\S \ref {subsec:Hill:0}). In \Sec \ref {sec:pert}, we start to derive Theorem \ref {th:ust}(A) from Theorem \ref {th1'}.
In \Sec\ref {par3:1:4}, we prove (in Lemma \ref {lem3:1}) that the transformation (\ref {eq:zamena}) brings the $N+1$ body problem to the perturbed system (\ref {eq:pert:planets}), and finish deriving Theorem \ref {th:ust}(A) from Theorem \ref {th1'}.
In \Sec \ref {par:proofs}, Theorems \ref {th1}, \ref {th:ust}(B), \ref {th:degen:sym} are derived from Theorems \ref {th:sym}, \ref {th:stab}, \ref {th:nondeg}.

The author expresses gratitude to N.\,N.\ Nekhoroshev for stating the problem, A.\,D.\ Bruno and Yu.\,M.\ Vorobiev for useful remarks that were quite helpful for making exposition clearer, to V.\,V.\ Kozlov, A.\,I.\ Neistadt, V.\,N.\ Tkhay, A.\,T.\ Fomenko, H.\ Zieschang and A.\,B.\ Kudryavtsev for their attention to the work.

\section{Statement of the problem, main results and methods of constructing solutions} \label{sec:vved'}

Let us formulate the results of the paper more precisely.

\subsection{Statement of the problem}

The planar $N+1$ body problem is described by the system of ODE's 
 \begin{equation} \label {O}
\frac{d\M_i}{dt}=\frac{\partial H(\M,\p)}{\partial\p_i}, \qquad 
\frac{d\p_i}{dt}=-\frac{\partial H(\M,\p)}{\partial\M_i}, \qquad \quad 0\le i\le N.
 \end{equation}
Here we denote
$\frac{\partial H(\M,\p)}{\partial\p_i}:=(\frac{\partial H}{\partial p_{i,1}},\frac{\partial H}{\partial p_{i,2}})$,
$\frac{\partial H(\M,\p)}{\partial\M_i}:=(\frac{\partial H}{\partial q_{i,1}},\frac{\partial H}{\partial q_{i,2}})$,
 \begin{equation} \label {pot}
H=H(\M,\p)=H(\M,\p;g,\m_0,\dots,\m_N) 
=\sum_{i=0}^N\frac{\p_i^2}{2\m_i}-\sum_{0\le i<j\le N}\frac{\grav\m_i\m_j}{r_{ij}}
=:\Kin+U
 \end{equation}
is the full energy of the system (equals the sum of kinetic and potential energies, $\Kin$ and $U$), $\M=(\M_0,\dots,\M_N)\in Q\subset(\bbR^2)^{N+1}$ are ``coordinates'', $\p=(\p_0,\dots,\p_N)\in(\bbR^2)^{N+1}$ are ``momenta'', 
$\M_i=(q_{i,1},q_{i,2})\in\bbR^2$ is the radius-vector of the $i$th body,
$\p_i=(p_{i,1},p_{i,2})\in\bbR^2$ is its momentum,
$r_{ij}=|\M_j-\M_i|$, $0\le i,j\le N$, are pairwise distances between bodies; $\grav>0$ is the gravitation constant, 
$\m_i>0$ is the mass of the $i$th body.
Here $Q\subset(\bbR^2)^{N+1}$ is the configuration manifold of the problem under consideration, consisting of all tuples of radius-vectors $\M_i\in\bbR^2$, $0\le i\le N$, such that $\M_i\ne\M_j$, $0\le i<j\le N$. The system (\ref {O}), (\ref {pot}) depends on $N+2$ parameters $g,\m_0,\dots,\m_N>0$ and it is a Hamiltonian system with $2N+2$ degrees of freedom, with the Hamilton function (\ref {pot}) and the symplectic structure 
 \begin{equation} \label {eq:oomega}
 \oomega=d\p\wedge d\M=\sum_{i=0}^Nd\p_i\wedge d\M_i:=\sum_{i=0}^N\sum_{k=1}^2
  dp_{i,k}\wedge dq_{i,k},
 \end{equation}
defined on an open subset $T^*Q=Q\times (\bbR^2)^{N+1}\subset \bbR^{4(N+1)}$ with coordinates $\M,\p$. 

In particular, the phase space of the problem has dimension $\dim(T^*Q)=4N+4$. 
Consider the submanifold $\widehat Q:=\{(\M_0,\dots,\M_N)\in Q\mid \sum_{i=0}^N \m_i\M_i=0\}$ in the configuration manifold $Q$. It consists of all configurations of $N+1$ particles in the Euclidean plane with masses $\m_0,\dots,\m_N$ and the center of masses at the origin. There exists a natural symplectomorphism between $T^*\widehat Q$ and $4N$-dimensional symplectic submanifold $M^{4N}\subset T^*Q$:
 \begin{equation} \label {eq:4N}
 M^{4N}:=\left\{(\M_0,\dots,\M_N,\p_0,\dots,\p_N)\in T^*Q \ \left| \
 \sum_{i=0}^N \m_i\M_i=\sum_{i=0}^N \p_i=0\right. \right\} .
 \end{equation}
The coordinates of the total momentum $\p_0+\dots+\p_N:T^*Q\to\bbR^2$ are first integrals of the system (\ref {O}). It is enough to study the restriction of the system (\ref {O}) to the $4N$-dimensional invariant submanifold $M^{4N}\approx T^*\widehat Q$.

\begin{scRem} \label {rem:relation}
Without loss of generality, we may choose the unities of mass,
distance and time as it will be suitable. In fact, for any constants
$a,b,c>0$, the power transformation
 $\widetilde\M_i=b^2c\M_i$, 
 $\widetilde\p_i=ab^{-1}c\p_i$, 
 $\widetilde t=a^{-1}b^3ct$, 
 $\widetilde H=a^2b^{-2}cH$, 
 $\widetilde\oomega=abc^2\oomega$,
together with the transformation of parameters
 $\widetilde g=a^2g$,
 $\widetilde\m_i=c\m_i$
brings the system 
(\ref {O}), (\ref {pot}) to the Hamiltonian system with the Hamilton function $\widetilde H=H(\widetilde\M,\widetilde\p;\widetilde g,\widetilde\m_0,\dots,\widetilde\m_N)$ and the symplectic structure $\widetilde\oomega=d\widetilde\p\wedge d\widetilde\M$.
In particular, we do not need to assume that the gravitational constant $\grav$ is arbitrary, but we may assume its value to be a distinguished number that we will choose below. (This can be achieved via scaling the time, as we described above.)
\end{scRem}

\begin{scDef} \label {def:per}
A solution $(\M(t),\p(t))$ of the planar $N+1$ body problem (\ref {O}), (\ref {pot}) will be called {\it relatively periodic} (or {\it periodic}, by abuse of language) if the locations of all bodies (and, hence, momenta too) after the time-interval $T>0$ can be obtained from their initial locations by rotating the plane by the same angle $\alpha$ around the centre of masses, for any initial time value, where
$-\pi<\alpha\le\pi$. The pair of real numbers $(T,\alpha)$ will be called {\it the relative period} of the solution, and the solution itself will be called {\it $(T,\alpha)$-periodic}.
 \end{scDef}

Any solution obtained from a $(T,\alpha)$-periodic solution via
shifting the time by a value $t$ and rotating by an angle $\varphi$
around the origin is a $(T,\alpha)$-periodic solution too. The union
of the phase trajectories of all such solutions is a {\it
two-dimensional torus} in the phase space, since it admits angular
coordinates $t\mod T$, $\varphi\mod2\pi$. All these solutions will be
regarded as a single {\it $(T,\alpha)$-periodic solution}, and the
union of their phase trajectories will be called {\it the phase
orbit} of this solution.

Many relatively periodic solutions of the planar $N+1$ body problem
happen to be {\it symmetric} (Theorem \ref {th1}). These solutions
are characterized by the following property: at some time, all the
bodies lie on the same line (i.e.\ a ``parade'' is observed) and
their velocities are perpendicular to this line.

In the present work, the following partial case of the planar $N+1$ body problem is considered, $N\ge2$. We assume that the mass of one
particle (the Sun) equals $\m_0=1$ and is much greater than the masses of
the other particles (the planets and satellites). Moreover the mass
$\mu_i$ of the $i$th planet and the mass $\mu_{ij}$ of its $j$th
satellite have the form
 \begin{equation} \label {A}
\m_i   =\mu m_i, \qquad \m_{ij}=\mu\nu m_{ij}\le\mu_i, \qquad
1\le i\le n, \quad 1\le j\le n_i,
 \end{equation}
where $0<\mu,\nu\ll 1$ are small parameters and $m_i,m_{ij}$ are
positive parameters far enough from zero (e.g.\ positive constants)
with the properties
 \begin{equation} \label {nov:**}
\sum_{i=1}^nm_i=1, \qquad
 \min_{i=1}^n \sum_{j=1}^{n_i}\frac{m_{ij}}{m_i}=1,
 \end{equation}
moreover $m_{ij}$ are bounded for $n_i\ge2$, where $n_i$ is the
number of satellites of the $i$th planet and
$1+n+\sum_{i=1}^nn_i=N+1$ is the number of all bodies. Thus, for each
``double planet'' ($n_i=1$), the mass of the satellite equals $\mu
m_i\theta_i/(1-\theta_i)$ where the parameter $\theta_i:=\nu
m_{i1}/(m_i+\nu m_{i1})\in(0,1/2]$ is not necessarily small (since
$m_{i1}$ is not necessarily bounded). 

We also assume that the distance $R_i$ between the Sun and the $i$th
planet is of order $R\gg1$, while each satellite is at the distance
$r_{ij}$ of order 1 from its planet. Finally, ``the years are much
longer than the months'', i.e.\ the angular frequency $\fomega_i$ of
the rotation of each planet around the Sun is of order $\somega$,
while the angular frequencies $\Omega_{ij}$ of the rotations of its
satellites about it have order 1 where $0<\omega\ll1$. More
precisely, let a set of non-vanishing real numbers
 \begin{equation} \label {chast}
\fomega_i=\somega\Omega_{i0}, \quad \Omega_{ij}, \qquad 1\le i\le n,
\quad 1\le j\le n_i,
 \end{equation}
called {\it the set of angular frequencies} satisfies the conditions
 \begin{equation} \label {poryadki:chastot'}
 {c}\le|\Omega_{i0}|
 \le |\Omega_{ij}|\le1, \qquad 1\le i\le n, \quad 1\le j\le n_i,
 \end{equation}
 \begin{equation} \label {poryadki:chastot''}
 \left | |{\Omega_{i0}}|-|{\Omega_{i'0}}|\right | \ge{c}, \ \
 i<i', \quad
 \left | |\Omega_{ij}|-|\Omega_{ij'}|\right | \ge{c},
 \ \ 1\le j<j'\le n_i.
 \end{equation}
Here $c$ is a suitable real number in the interval $0<c<1$.

Suppose that the parameters $\omega,\mu,g,R$ satisfy the natural
relations $\omega^2R^3=g$ and $1=g\mu$ corresponding to Kepler's second
law for the planets ($\fomega_i^2R_i^3=g\m_0$) and the
satellites ($\Omega_{ij}^2r_{ij}^3=g\m_i$), for the chosen unities of
mass, distance and time. Thus, $\grav=\omega^2R^3=1/\mu$,
 \begin{equation} \label {nov:*}
\rho^3=\frac{1}{R^3}=\omega^2\mu
 \end{equation}
and the problem has $N+3$ independent parameters: $N$ parameters $m_i,m_{ij}>0$ and three small parameters $\mu$, $\nu$ and $\somega$ (here $\rho:=1/R$). 
We emphasize that the initial $N+1$ body problem does not include the parameter $\somega$ (and, hence, $\rho$ related with it), but it was introduced by us as an additional (``scaling'') small parameter. Namely, by means of this small parameter, we construct a $N+3$-parameter family (with parameters $m_i,m_{ij},\mu,\nu,\somega>0$) of ``scaling'' transformations of phase variables (\ref {eq:rel:coord}), (\ref {eq:rel:impuls}), (\ref {eq:zamena}). These transformations bring solutions $(\M(t),\p(t))$, discovered by us, of the planetary system with satellites (with ratio of ``months'' to ``years'' $\fomega_i/\Omega_{ij}$ of order $\somega\ll1$) to motions $(\px(t),\pxxi(t))$, that are close to ``generating'' (see below) circular motions $(\mx^\beta(t),\mxxi^\beta(t))$ along circles whose radii have order $1$. (The indicated ``scaling'' is necessary, since the radii of the corresponding circles for ``non-scaled planets'' and ``satellites'' equal $R_i,r_{ij}$ and have orders $R=\frac1\rho\gg1$ and $1$.) Therefore the only imposed restrictions to the parameters of the problem are as follows: the mass of the Sun is $\m_0=1$, the distances from the satellites to their planets are of order 1, and the gravitational constant is $g=1/\mu$. This does not cause any loss of generality due to remark~\ref {rem:relation}.

\subsection{Main results}

Let us describe the ``generating'' relatively-periodic solutions of the $N+1$
body problem under consideration. These are the relatively-periodic solutions of
the ``model'' problem (described below), whose phase trajectories form an $N$-dimensional invariant torus $\Lambda\o$ in a ``model'' phase space. In order to construct solutions of the initial $N+1$ body problem, that are close to ``generating'' ones, we will introduce three dynamical systems with $2N$ degrees of freedom, called {\it model}, {\it unperturbed} and {\it perturbed} problems, respectively. 
The first and the second systems (the model and unperturbed ones) depend on $N+1$ parameters $m_i,m_{ij}>0$ and $\somega\in\bbR$, and the second system is obtained from the first one by a perturbation of order $O(\somega^2)$. The third system (the perturbed one) depends on $N+5$ parameters $m_i,m_{ij}>0$ and $\somega,\varepsilon,\mu,\nu,\rho\in\bbR$, $\varepsilon\ne0$, and it is obtained from the second system by a regular 4-parameter perturbation with small parameters $\varepsilon,\mu,\nu,\rho$ (which are assumed to be 0 for the second system).

1) The {\it model system} is a Hamiltonian system with $2N$ degrees of freedom and $N+1$ parameters,
with the following Hamilton function and the symplectic structure:
 \begin{equation} \label {eq:model:H:omega}
 \somega H_0(\mx_*,\mxxi_*;m_*)+H_1(\mx_{**},\mxxi_{**};m_*,m_{**}), \quad 
 d\mxxi\wedge d\mx = \oomega_0+\oomega_1 ,
 \end{equation}
where the functions $H_0(\mx_*,\mxxi_*)=H_0(\mx_*,\mxxi_*;m_*)$, $H_1(\my_{**},\meeta_{**})=H_1(\my_{**},\meeta_{**};m_*,m_{**})$ and the closed 2-forms $\oomega_0,\oomega_1$ are defined by the formulae
 \begin{equation} \label {eq:model:H}
 H_0 (\mx_*,\mxxi_*)
 :=\sum^n_{i=1}\left(\frac{\mxxi^2_i}{2m_i}-\frac{m_i}{|\mx_i|}\right), \
 H_1 (\my_{**},\meeta_{**})
 :=
\sum^n_{i=1}\sum^{n_i}_{j=1}\left(\frac{\meeta^2_{ij}}{2m_{ij}}-\frac{m_im_{ij}}{|{\my}_{ij}|} \right) ,
 \end{equation}
 \begin{equation} \label {eq:model:omega}
 \oomega_0 
 := d\mxxi_*\wedge d\mx_*
 =\sum_{i=1}^nd\mxxi_i\wedge d\mx_i, \quad
 \oomega_1 
 :=d\meeta_{**}\wedge d\my_{**}
 =\sum_{i=1}^n\sum_{j=1}^{n_i}d\meeta_{ij}\wedge d\my_{ij} .
 \end{equation}
Here $\mx_i,\my_{ij}\in\bbR^2\setminus\{(0,0)\}$ are ``coordinates'',  $\mxxi_i,\meeta_{ij}\in\bbR^2$ are ``momenta'', $m_i,m_{ij}>0$ and $\somega\in\bbR$ are parameters. 
From now on, $\mx_*\in(\bbR^2)^n$ denotes the tuple of vectors $\mx_i$, and $\my_{**}\in(\bbR^2)^{N-n}$ denotes the tuple of vectors $\my_{ij}$ ($1\le i\le n$, $1\le j\le n_i$). For $\somega\ne0$, the model problem (\ref{eq:model:H:omega}), (\ref {eq:model:H}), (\ref{eq:model:omega}) splits into $N$ independent Kepler's problems (for ``scaled planets'' and ``satellites'').

Let us assume that $\somega>0$ and the collection (\ref {chast}) is maximally relatively-resonant, i.e.\ has the form
 \begin{equation} \label {maxrez}
  \begin{array} {rcll}
  \fomega_i   &=& \fomega_1+k_i\frac{2\pi}{T},    \qquad & 1\le i\le n, \\
  \Omega_{ij}&=& \fomega_1+K_{ij}\frac{2\pi}{T}, \qquad & 1\le j\le n_i,
  \end{array}
 \end{equation}
where $k_i,K_{ij}\in\bbZ$, $T>0$. A solution $(\mx^0(t),\mxxi^0(t))=(\mx_*^0(t),\my_{**}^0(t),\mxxi_*^0(t),\meeta_{**}^0(t))$ of the model problem (\ref{eq:model:H:omega}), (\ref {eq:model:H}), (\ref{eq:model:omega}) will be called a {\em main generating solution} if it is a tuple of circular solutions
 \begin{equation} \label {eq:gen:sol}
 (\mx_\ell^0(t),\mxxi_\ell^0(t)) :=  \frac{e^{i\fomega_\ell t}}{\Omega_{\ell 0}^{2/3}}
 (1,i\Omega_{\ell 0} m_\ell) ,  \
 (\my_{\ell j}^0(t),\meeta_{\ell j}^0(t)) :=   
 \frac{e^{i\Omega_{\ell j}t}}{\Omega_{\ell j}^{2/3}}
 m_\ell^{1/3} (1,i\Omega_{\ell j}m_{\ell j}) 
 \end{equation}
of the corresponding Kepler problems (for the ``scaled planets'' around the Sun and the ``satellites'' around the planets) with angular frequencies (\ref {chast}) of the form (\ref {poryadki:chastot'}), (\ref {poryadki:chastot''}), (\ref {maxrez}), where $\fomega_\ell=\somega\Omega_{\ell 0}$ and the plane $\bbR^2$ is identified with $\bbC$. 
Any solution $(\mx^\beta(t),\mxxi^\beta(t))$, that is a tuple of circular solutions 
 $(\mx_\ell^\beta(t),\mxxi_\ell^\beta(t)):=e^{i\beta_\ell} (\mx_\ell^0(t),\mxxi_\ell^0(t))$, 
 $(\my_{\ell j}^\beta(t),\meeta_{\ell j}^\beta(t)):=e^{i\beta_{\ell j}} 
 (\my_{\ell j}^0(t),\meeta_{\ell j}^0(t))$, $1\le \ell\le n$, $1\le j\le n_\ell$, 
with a given collection of frequencies, will be called a {\it generating solution}. Here $\beta_\ell,\beta_{\ell j}\in\bbR/2\pi\bbZ$ are arbitrary constants, $\beta:=(\beta_*,\beta_{**})\in(\bbR/2\pi\bbZ)^N$.
The union of phase trajectories of all generating solutions (with $\somega\ne0$) is a $N$-dimensional torus $\Lambda\o=\Lambda\o(\somega,\fomega_*,\Omega_{**};m_*,m_{**})$. Indeed:
the polar angles of $N$ radius vectors $\mx_\ell$, $\my_{\ell j}$ (drawn from the Sun to the planets and from the planets to their satellites) can be used as coordinates on it. 
Any generating solution can be obtained from the solution $(\mx^0(t),\mxxi^0(t))$ by shifts along the flows of $N$ pairwise commuting Hamiltonian vector fields with Hamiltonian functions $I_i:=[\mx_i,\mxxi_i]$, $I_{ij}:=[\my_{ij},\meeta_{ij}]$. Due to the condition (\ref {maxrez}), generating solutions are $(T,\alpha)$-periodic with relative period 
 \begin{equation} \label {parametry}
T>0, \qquad \alpha=\fomega_1T+2\pi k\in(-\pi,\pi],
 \end{equation}
where $k\in\bbZ$ is a suitable integer. 

In dependence on whether the satellites are absent ($N=n$) or present ($N>n$), the following condition will be called {\it the nondegeneracy condition}:
 \begin{equation} \label {alpha}
\alpha\ne0 \quad \mbox{при }N=n, \qquad \quad
|\alpha|>\somega^2T \quad \mbox{при }N>n.
 \end{equation}

2) The {\it unperturbed problem} is the dynamical system (which is non-Hamiltonian, but it is ``$0$-Hamiltonian'', i.e.\ a semidirect product of Hamiltonian systems, cf.\ (\ref {eq:lam:Ham})) with $2N$ degrees of freedom and $N+1$ parameters, consisting of two parts. The first part of the system is the autonomous Hamiltonian system with $2n$ degrees of freedom, which is given in the space with coordinates $\ux_i\in\bbR^2\setminus\{(0,0)\}$, $\uxxi_i\in\bbR^2$ ($1\le i\le n$), with the Hamilton function and the symplectic structure 
 \begin{equation} \label {eq:unpert0}
 \somega H_0(\ux_*,\uxxi_*;m_*), \qquad 
 \oomega_0 = d\uxxi_*\wedge d\ux_*.
 \end{equation}
The second part of the system is the non-autonomous Hamiltonian system with $2(N-n)$ degrees of freedom, which is given in the space with coordinates $\uy_{ij}\in\bbR^2\setminus\{(0,0)\}$, $\ueeta_{ij}\in\bbR^2$ ($1\le i\le n$, $1\le j\le n_i$), with the Hamilton function and the symplectic structure 
 \begin{equation} \label {eq:unpert1} 
 H_1(\ux_{**},\uxxi_{**};m_*,m_{**}) + \somega^2 \Phi(\ux_*(t),\uy_{**};m_{**}), \qquad 
 \oomega_1= d\ueeta_{**}\wedge d\uy_{**} ,
 \end{equation}
cf.\ (\ref {eq:model:H}) and (\ref {eq:model:omega}). Here $\ux_*(t)=(\ux_1(t),\dots,\ux_n(t))$ denote the corresponding components of the solution $(\ux_*(t),\uxxi_*(t))=(\ux_1(t),\dots,\ux_n(t),\uxxi_1(t),\dots,\uxxi_n(t))$ of the first part (\ref {eq:unpert0}) of the system, and the functions $\Phi(\ux_*,\uy_{**})=\Phi(\ux_*,\uy_{**};m_{**})$ is defined by the formulae
 \begin{equation} \label {F:xy}
\Phi(\ux_*,\uy_{**};m_{**}):=\sum_{i=1}^n \sum_{j=1}^{n_i} m_{ij} F(\ux_i,\uy_{ij}), \ 
 \ F(\hx,\hy):=\frac{\hx^2\hy^2-3\langle\hx,\hy\rangle^2}{2|\hx|^5}.
 \end{equation}
Here $\ux_i,\uy_{ij}\in\bbR^2\setminus\{(0,0)\}$ are ``coordinates'', $\uxxi_i,\ueeta_{ij}\in\bbR^2$ are ``momenta'', $m_i,m_{ij}>0$ and $\somega\in\bbR$ are parameters of the unperturbed system. The first part (\ref {eq:unpert0}) of the unperturbed system coincides with the corresponding part of the model system (\ref {eq:model:H:omega}), and for $\somega\ne0$ it is a collection of $n$ independent Kepler's problems for ``scaled planets''. The second part (\ref {eq:unpert1}) of the unperturbed system is $O(\somega^2)$--close to the second part of the model system (\ref {eq:model:H:omega}), 
 and it is the collection of $N-n$ ``sidereal Hill problems'' for the ``satellites'' (cf.\ \S \ref {subsec:Hill:0}), 
provided that $\somega\ne0$ and the ``scaled planets'' perform circular motions $\ux_i(t)=\mx_i^\beta(t)$, $\uxxi_i(t)=\mxxi_i^\beta(t)$ by virture of the first part (\ref {eq:unpert0}) of the unperturbed problem.

For any $|\somega|>0$ small enough, denote by $(\ux_*^0(t),\uy_{**}^0(t),\uxxi_*^0(t),\ueeta_{**}^0(t))=(\uy^0(t),\ueeta^0(t))$ such a $(T,\alpha)$-periodic solution of the unperturbed problem (\ref {eq:unpert0}), (\ref {eq:unpert1}), that is $O(\somega)$--close to the generating solution $(\mx^0(t),\mxxi^0(t))$ and satisfies the following conditions: $\ux_i^0(t)=\mx_i^0(t)$, $\uxxi_i^0(t)=\mxxi_i^0(t)$ and the vectors $\uy_{ij}^0(0)$ are collinear with the abscissa axis. Let 
 \begin{equation} \label {eq:unpert:sol}
 (\ux^\beta(t),\uxxi^\beta(t)) = (\ux_*^\beta(t),\uy_{**}^\beta(t),\uxxi_*^\beta(t),\ueeta_{**}^\beta(t)), \quad \beta=(\beta_*,\beta_{**})\in(\bbR/2\pi\bbZ)^N
 \end{equation}
be a $(T,\alpha)$-periodic solution of the unperturbed system, obtained from the indicated ``main'' solution $(\ux^0(t),\uxxi^0(t))$ by a composition of some shifts (in time periods $\beta_i,\beta_{ij}$) along the flows of $N$ pairwise commuting Hamiltonian vector fields with the Hamilton functions $I_i:=[\ux_i,\uxxi_i]$, $I_{ij}:=[\uy_{ij},\ueeta_{ij}]$ and the symplectic structure $d\uxxi\wedge d\ux$, where $\beta_\ell,\beta_{\ell j}\in\bbR/2\pi\bbZ$ are arbitrary constants, $\beta:=(\beta_*,\beta_{**})$.
Let $\Lambda=\Lambda(\somega;\fomega_*,\Omega_{**};m_*,m_{**})$ be the union of phase trajectories of $(T,\alpha)$-periodic solutions $(\ux^\beta(t),\uxxi^\beta(t))$, $\beta\in(\bbR/2\pi\bbZ)^N$. 
We prove (\S \ref {par3:1:5}) by means of the averaging method that such a solution $(\ux^0(t),\uxxi^0(t))$ exists and it is $O(\somega^2)$--close to the main generating solution $(\mx^0(t),\mxxi^0(t))$ (hence any solution $(\ux^\beta(t),\uxxi^\beta(t))$ is $O(\somega^2)$--close to the corresponding generating solution $(\mx^\beta(t),\mxxi^\beta(t))$, and the $N$-dimensional torus $\Lambda$ is $O(\somega^2)$--close to the torus $\Lambda\o$ when $0<|\somega|\ll1$).

3) The {\it perturbed problem} is the Hamiltonian system with $2N$ degrees of freedom and $N+5$ parameters, with the Hamiltin functions and the symplectic structure 
 \begin{equation} \label {eq:pert:planets}
 \widetilde H= \somega \widetilde H_0(\px_*,\pxxi_*)
 + \varepsilon \left( \widetilde H_1(\py_{**},\peeta_{**})
 + \somega^2 \widetilde\Phi (\px_*,\py_{**}) \right), \ \
 \widetilde\oomega = 
 d\pxxi_*\wedge d\px_*+\varepsilon d\peeta_{**}\wedge d\py_{**},
 \end{equation}
where the functions $\widetilde H_0(\px_*,\pxxi_*)=\widetilde H_0(\px_*,\pxxi_*;\bar m_*,\mu)$, $\widetilde H_1(\px_{**},\pxxi_{**})=\widetilde H_1(\py_{**},\peeta_{**};\bar m_*,m_{**},\nu)$, 
$\widetilde\Phi=\widetilde\Phi(\px_*,\py_{**};\bar m_*,m_{**},\mu,\nu,\rho)$ are analytic, are in involution with the function 
$$
\widetilde I=\sum_{i=1}^n\left(
[\px_i,\pxxi_i]+\varepsilon\sum_{j=1}^{n_i}[\px_{ij},\pxxi_{ij}] \right)
$$ 
for any values of the parameters, and have the property $\widetilde H_0|_{\mu=0}=H_0$, $\widetilde H_1|_{\nu=0}=H_1$, $\widetilde\Phi|_{\mu=\nu=\rho=0}=\Phi$ (the functions $\widetilde H_0,\widetilde H_1,\widetilde\Phi$ and parameters $\bar m_i$ are defined in more details in (\ref {H0}), (\ref {G}) and (\ref {m})), cf.\ (\ref {eq:model:omega}).
Here $\px=(\px_*,\py_{**})\in\widetilde Q=\widetilde Q(m_*,m_{**},\mu,\nu,\rho)\subset(\bbR^2)^N$ are ``coordinates'', 
$\pxxi=(\pxxi_*,\peeta_{**})\in(\bbR^2)^N$ are ``momenta'', $m_i,m_{ij}>0$ and $\somega,\varepsilon,\mu,\nu,\rho\in\bbR$ are parameters of the system, $\varepsilon\ne0$. Here the domain 
 $$
\bigcup\limits_{(m_*,m_{**},\mu,\nu,\rho)\in\bbR^{N+3}}
\widetilde Q(m_*,m_{**},\mu,\nu,\rho)\times(\bbR^2)^N\times\{(m_*,m_{**},\mu,\nu,\rho)\} \subset(\bbR^2)^{2N}\times\bbR^{N+3}
 $$ 
contains the torus $\Lambda\o(\somega,\fomega_*,\Omega_{**};m_*,m_{**})\times\{(m_*,m_{**},0,0,0)\}$, coincides with the domain of (smooth and analytical) functions $H_0,H_1,\widetilde R_0,\widetilde R_1,\widetilde\Phi$ (as functions in variables $\px,\pxxi,m_*,m_{**},\mu,\nu,\rho$), moreover the complement of this domain in $(\bbR^2)^{2N}\times\bbR^{N+3}$ is a real-algebraic subset. Here the functions $\widetilde R_0$ and $\widetilde R_1$ are defined by the conditions $\widetilde H_0=H_0+\mu\widetilde R_0$ and $\widetilde H_1=H_1+\nu\widetilde R_1$.  
It turns out (\S \ref {sec:pert}) that there exists a smooth family of tori $\widetilde\Lambda=\widetilde\Lambda(\somega,\varepsilon,\mu,\nu,\rho;\fomega_*,\Omega_{**};m_*,m_{**})$, containing the phase trajectories of all $(T,\alpha)$-periodic solutions near $\Lambda$, such that $\widetilde\Lambda|_{\varepsilon=\mu=\nu=\rho=0}=\Lambda$.

A key (and apparently new) observation of this work is the following. Consider the planar $N+1$ body problem with the center of masses at the origin and with the parameters $m_i,m_{ij},\mu,\nu>0$, and 
the following canonical (by Lemma \ref {lem3:1}) coordinates in its $4N$-dimensional phase space $M^{4N}\approx T^*\widehat Q$ (cf.\ (\ref {eq:4N})):
 \begin {equation} \label {eq:rel:coord} 
\widehat\M_i:=\C_i-\M_0, \quad
\widehat\M_{ij}:=\M_{ij}-\M_i, \qquad \mbox{где} \quad \C_i
:=\frac{m_i\M_i+\nu\sum_{\ell=1}^{n_i}m_{i\ell}\M_{i\ell}}{m_i+\nu\sum_{\ell=1}^{n_i}m_{i\ell}},
 \end {equation}
 \begin {equation} \label {eq:rel:impuls} 
 \widehat\p_i := \p_i+\sum_{\ell=1}^{n_i}\p_{i\ell}, \quad
 \widehat\p_{ij}:=\p_{ij}
 -\frac{\nu m_{ij}}{m_i+\nu\sum_{\ell=1}^{n_i}m_{i\ell}} \left(\p_i+\sum_{\ell=1}^{n_i}\p_{i\ell}\right),
 \end {equation}
$1\le i\le n$, $1\le j\le n_i$.
It turns out (Lemma \ref {lem3:1}) that the power transformation 
 \begin {equation} \label {eq:zamena}
\px_i:=\rho\widehat\M_i, \
\py_{ij}:=\widehat\M_{ij}, \quad
\pxxi_i:=\frac{\widehat\p_i}{\sqrt{\mu\rho}}, \
\peeta_{ij}:=\frac{\widehat\p_{ij}}{\mu\nu}, \quad 
\widetilde H:=
 \sqrt{\frac{\rho}{\mu}} H, \
\widetilde\oomega:=
 \sqrt{\frac{\rho}{\mu}} \oomega
 \end {equation}
brings the $N+1$ body problem on $M^{4N}\approx T^*\widehat Q$ under consideration to the $N+3$-parameter subfamily of the $N+5$-parameter (real-analytic) family of perturbed systems (\ref {eq:pert:planets}), in which the parameters $m_i,m_{ij},\somega,\varepsilon,\mu,\nu,\rho$ are positive and related by fractional-power relations
 \begin{equation} \label {eq:rho:varepsilon}
\rho=\somega^{2/3}\mu^{1/3}, \qquad \varepsilon=\somega^{1/3}\mu^{2/3}\nu.
 \end{equation}

When we talk about closeness of some solution (or of some submanifold of the phase space) of one system to a solution (resp.\ a submanifold) of another system, we assume that the phase spaces of the model, unperturbed and perturbed systems are identified with the corresponding open subsets of the vector space $(\bbR^2)^N\times(\bbR^2)^N$ with coordinates $(\px,\pxxi)$. Besides, we identify the phase space $T^*\widehat Q\approx M^{4N}$ (depending on the parameters $m_i,m_{ij},\mu,\nu>0$) of the $N+1$ body problem under consideration with the corresponding subset (with the corresponding values of $m_i,m_{ij},\mu,\nu$ and any $\rho>0$) of the phase space $\widetilde Q\times(\bbR^2)^N$ of the perturbed system (\ref {eq:pert:planets}) via the transformation (\ref {eq:zamena}).

{\it Symmetric} periodic solutions of the three-body problem were
studied already by Poincar\'e \cite {1}. Recall the definition of a
symmetric solution of the planar $N+1$ body problem.

\begin {scDef} \label {def:sym}
Consider a problem describing the motion of $N+1$ particles in a
Euclidean plane. A solution of this problem will be called {\it
symmetric} if there exists a line $l$ in the plane, called the axis
of symmetry, and a time $t=t_0$ satisfying one of the following
(equivalent) conditions called a ``parade'' of the particles:

1) at the time $t=t_0$, all points are placed on the line $l$ (i.e.\
a ``parade'' of the particles is observed) and their velocities are
orthogonal to the line $l$;

2) the locations (and, hence, also the velocities) of all particles
at any time $t\in\bbR$ can be obtained from their locations at the
time $2t_0-t$ by reflecting with respect to the axis $l$.

All particles of the system are assumed to be numbered. 
Any solution of the $N+1$ body problem obtained from a
symmetric solution by shifting the time and by rotating the plane is
also symmetric. Similarly to the case of $(T,\alpha)$-periodic
solutions, we will not distinguish such solutions and will regard
them as a {\it single symmetric solution}.
 \end {scDef}

Exactly $2^{N-2}$ of the generating solutions $(\mx^\beta(t),\mxxi^\beta(t))$ are symmetric. They are determined by the corresponding tuples of values $\beta_\ell,\beta_{\ell j}\in\{0,\pi\}$ with $\beta_1=0$. Here the tuples $\beta=(\beta_*,\beta_{**})\in\{0,\pi\}^N$ and $(\beta_*+\pi k_*,\beta_{**}+\pi K_{**})$ correspond to the same relatively-periodic solution (we assume without loss of generality that the tuple $(k_*,K_{**})\in\bbZ^N$ has no common factors). 
Let
 \begin{equation} \label {eq:B}
B\subset\{0,\pi\}^N\subset(\bbR/2\pi\bbZ)^N
 \end{equation} 
be a subset of cardinality $2^{N-2}$ containing exactly one element of each such a pair of tuples. 
The symmetric $(T,\alpha)$-periodic solutions are characterized by the condition that ``parades'' of the planets and satellites are observed, moreover they repeat each half of the
period, $T/2$. This means that all the particles of the system are
posed on a line, which turns by the angle $\alpha/2$ after the time-interval $T/2$.

\begin {The} [(on the number of families of relatively-periodic solutions)] \label {th1}
There exist constants $\somega_0,C>0$ and continuous positive functions 
$\mu_0=\mu_0(m_*,m_{**},\somega,c)$, $\nu_0=\nu_0(m_*,m_{**},\somega,c)$,
$0<\somega\le\somega_0$, $0<c<1$, such that, for any parameter values
$\somega,c,\mu,\nu$ with the properties $0<\somega\le\somega_0$,
$0<c<1$, $0<\mu\le\mu_0(m_*,m_{**},\somega,c)$, $0<\nu\le\nu_0(m_*,m_{**},\somega,c)$, for any tuple of angle frequencies {\rm(\ref {chast})} of the form {\rm(\ref {poryadki:chastot'}), (\ref {poryadki:chastot''}), (\ref {maxrez})} the following properties hold. Suppose that the parameters
{\rm(\ref {parametry})} satisfy either the nondegeneracy condition
{\rm(\ref {alpha})} or the following conditions {\rm(}which are more delicate when satellites are present{\rm)}:
 \begin{equation} \label {eq:fine}
\alpha\ne0, \qquad
 \alpha-\frac{\fomega_i^2}{4\Omega_{ij}}T\not\in\left[-C\somega^3T,C\somega^3T\right]+2\pi\bbZ
 \end{equation}
for $1\le i\le n$, $1\le j\le n_i$. Then the $N+1$ body problem of
the type of planetary system with satellites, $N\ge2$, has exactly
$2^{N-2}$ symmetric $(T,\alpha)$-periodic solutions $(\px^\beta(t),\pxxi^\beta(t))$, $\beta\in B\subset\{0,\pi\}^N$ {\rm(см.\ (\ref {eq:B}))}, that in coordinates {\rm(\ref {eq:zamena})} are
$O(\omega^2)$-close to the symmetric generating solutions $(\mx^\beta(t),\mxxi^\beta(t))$ {\rm(}corresponding to independent circular rotations of the ``scaled planets'' around the Sun and the ``satellites'' around the planets{\rm)} with the angular frequencies {\rm(\ref {chast})}. 
Each of these $2^{N-2}$ solutions depends smoothly on the pair $(T,\alpha)$, provided that $\beta$, the parameters of the problem and the integers $k_i,K_{ij}$ in {\rm(\ref {maxrez})} are fixed. 
All such solutions {\rm(}for all possible values of the parameters of the problem as above, and for fixed $\somega\in(0,\somega_0]$, $k_i,K_{ij}$ and $\beta${\rm)} form a $2+(N+2)$-parameter subfamily of the $2+(N+4)$-parameter smooth family of solutions of the perturbed problem {\rm(\ref {eq:pert:planets})} {\rm(}with fixed $\somega${\rm)} with parameters $T,\alpha$ and $m_i,m_{ij}>0$, $|\mu|\le\mu_0(m_*,m_{**},\somega,c)$, $|\nu|\le\nu_0(m_*,m_{**},\somega,c)$, $|\rho|\le\mu_0(m_*,m_{**},\somega,c)$, $|\varepsilon|\le\mu_0(m_*,m_{**},\somega,c)$, where the parameters of the subfamily are related by the relations $\mu,\nu>0$ and {\rm(\ref {eq:rho:varepsilon})}.
For each of these solutions, parades are observed, which repeat each
time-interval $\frac T2$: all of the particles of the system are
posed on a line \(which turns by the angle $\alpha/2$ after the time-interval $T/2$\).
 \end {The}

Let $\Lambda\o\subset (\bbR^2)^{2N}$ be the $N$-dimensional torus 
formed by the phase trajectories of the generating solutions (see (\ref {parametry})). 
Let $\Sigma\subset (\bbR^2)^{2N}$ be a ``transversal surface'' (called a cross section) of codimension 2 in the phase space, that transversally intersects invariant two-dimensional tori lying on $\Lambda\o$ and corresponding to the $(T,\alpha)$-periodic solutions:
 $$
\Sigma:=\left\{\sum_{i=1}^n(\varphi_i+\sum_{j=1}^{n_i}\varphi_{ij})=
 \frac{T_{\min}}{T}
 \sum_{i=1}^n(k_i\varphi_i+\sum_{j=1}^{n_i}K_{ij}\varphi_{ij})=0(\mod2\pi) \right\}.
 $$
Here $\varphi_i$ and $\varphi_{ij}$ are polar angles of the radius-vectors $\px_i$ and $\py_{ij}$,
$k_i,K_{ij}$ are integers in (\ref {maxrez}), $T_{\min}$ is the
minimal positive period, hence the integer $T/T_{\min}$ is the
greater common divisor of the collection of integers $k_i,K_{ij}$.
Let
$\widetilde\Lambda =\widetilde\Lambda(\somega,\varepsilon,\mu,\nu,\rho;\fomega_*,\Omega_{**};m_*,m_{**})$ be the $N$-dimensional torus of the indicated family of tori, containing the torus $\Lambda =\widetilde\Lambda(\somega,0,0,0,0;\fomega_*,\Omega_{**};m_*,m_{**})$.
Let $\Psi$ be the generating function of the ``succession
map'' $g_{H-\fomega_1I}^T$ of the $N+1$ body problem under
consideration (see (\ref {M}) and Definition \ref {def:gen:func}). Consider the
smooth function $\bar S=\Psi|_{\widetilde\Lambda\cap\Sigma}$ on the
$(N-2)$-dimensional torus $\widetilde\Lambda\cap\Sigma$. Since the
function $\bar S$ is defined on a $(N-2)$-dimensional torus, it has
at least $N-1$ critical points, moreover at least $2^{N-2}$ points
counted with multiplicities \cite {35}. We will prove (see Theorem
\ref {th:ust}) the same lower bound for the number of
$(T,\alpha)$-periodic solutions of the problem under consideration.
Observe that the function $\bar S$ has at least one critical point,
since it is defined on a closed manifold. We will also prove that each
critical point of the function $\bar S$ corresponds to a
$(T,\alpha)$-periodic solution of the $N+1$ body problem. Moreover,
we will offer sufficient conditions that guarantee the {\it orbital
\(structural\) stability in linear approximation} (see definition
\ref {def5}) of such a solution.

The following condition will be called {\it the strong nondegeneracy
condition}:
 \begin{equation} \label {alpha'}
 0<|\alpha|<\pi \ \mbox{ при }N=n, \qquad
\somega^2T<|\alpha|<\pi-\somega^2T \ \mbox{ при }N>n.
 \end{equation}

The following conditions will be called the {\it property of having
fixed sign}:

1) all planets rotate ``to the same side'', i.e.\ the angular
frequencies of the rotations of the planets around the Sun have the
same sign:
 \begin{equation} \label {polozh}
 \fomega_i\fomega_{i'}>0,
 \qquad 1\le i<i'\le n;
 \end{equation}

2) the function $\bar S$ on the $(N-2)$-dimensional torus is either a Morse function or has at least one nondegenerate critical point of a local minimum (this condition
is assumed to be always true if $N=2$).

\begin {The} [(on stability of a relatively-periodic solution)] \label {th:ust}
Under the hypothesis of Theorem {\rm\ref {th1}}, there exists a
smooth $N$-dimen\-si\-onal torus $\widetilde\Lambda$ in the phase space of the problem with 
coordinates {\rm(\ref {eq:zamena})} 
that is $O(\somega^2)$--close to the torus $\Lambda\o$, smoothly depends
on the pair $(T,\alpha)$ and has the following properties.

{\rm(A)} The phase orbits of all $(T,\alpha)$-periodic solutions of
the $N+1$ body problem that are $O(\somega)$--close to the torus
$\Lambda\o$ are contained in the torus $\widetilde\Lambda$. Moreover
their intersection points with the cross section $\Sigma$
coincide with critical points of the function $\bar
S:=\Psi|_{\widetilde\Lambda\cap\Sigma}$ defined on the $(N-2)$-dimensional
torus $\widetilde\Lambda\cap\Sigma$. Here $\Psi$ is the generating
function of the ``succession map'' $g_{H-\fomega_1I}^T$ of the problem
under consideration {\rm(see (\ref {M}) and Definition \ref {def:gen:func})}. The
function $\Psi|_{\widetilde\Lambda}$ is an even function in the
collection $\vvarphi$ of angle variables 
$\varphi_i|_{\widetilde\Lambda}$,
$\varphi_{ij}|_{\widetilde\Lambda}$. The phase orbits of the
symmetric $(T,\alpha)$-periodic solutions contain the points
$\vvarphi$ of the torus $\widetilde\Lambda$ having the property
$\vvarphi=-\vvarphi$.

{\rm(B)} Suppose that the property of having fixed sign holds,
moreover either the strong nondegeneracy condition {\rm(\ref
{alpha'})} or the following conditions hold {\rm(}which are more delicate when satellites are present{\rm)}:
 \begin{equation} \label {eq:fine'a}
 \alpha\not\in\{0,\pi\}, \qquad  
 \frac{{\rm sgn\,}\fomega_1+{\rm sgn\,}\Omega_{ij}}2\alpha - \frac{\fomega_i^2}{8|\Omega_{ij}|}T \not\in
 \left[-\frac{C}2\somega^3T,\frac{C}2\somega^3T\right]+\pi\bbZ,
 \end{equation}
 \begin{equation} \label {eq:fine'b}
  \frac{{\rm sgn\,}\Omega_{ij}+{\rm sgn\,}\Omega_{i'j'}}2\alpha-
 \left(
 \frac{\fomega_i^2}{|\Omega_{ij}|}+\frac{\fomega_{i'}^2}{|\Omega_{i'j'}|}
 \right)
 \frac T8\not\in\left[-C\somega^3T,C\somega^3T\right]+\pi\bbZ
 \end{equation}
for all $1\le i,i'\le n$, $1\le j\le n_i$ and $1\le j'\le n_{i'}$.
Then the $(T,\alpha)$-periodic solution corresponding to any
nondegenerate critical point of a local minimum of the function $\bar
S$ is orbitally structurally stable in linear approximation.
 \end {The}

Theorem \ref {th:ust}(B) implies that, for $N=3$, in the ``generic
case'', a half of the $(T,\alpha)$-periodic solutions that are close
to the torus $\Lambda\o$ are orbitally stable in linear approximation
(since the function $\bar S$ is defined on a circle and, hence, has
only critical points of local minima and maxima, which alternate on
the circle).

The natural question arises: is the nondegeneracy condition (\ref
{alpha}) necessary for the validity of theorem \ref {th1}? An answer
happens to be affirmative in many cases.

In the following theorem, by {\it ``almost any''} collection of masses 
$\mu_i>0$, $1\le i\le n$, we mean any collection belonging to
the complement in $\bbR^n_{>0}$ to the union of a finite set of
linear subspaces of $\bbR^n$. Moreover each of these subspaces
depends on the collection of integers $k_i$ in (\ref {maxrez}), has
codimension at least 2, and the number of these subspaces does not
exceed $2^{n-2}$. The set ${\cM}^{\rm sym}$ of ``almost all''
collections of masses is described in more detail in \Sec\ref {par3:1:3'}. By
the phase space of satellites, we regard the direct product of big
balls in the phase spaces of the corresponding Kepler problems,
except for a small neighbourhood of ``the set of possible
collisions''.

\begin {The} [(on ``gaps'' in families of relatively-periodic solutions)] \label {th:degen:sym}
Suppose that, under the hypothesis of theorem {\rm\ref {th1}}, the
number $n$ of planets is at least $2$ and there exist two planets
whose angular frequencies satisfy the following resonance relation:
 \begin{equation} \label {degen:3bod'}
 \frac{\fomega_i}{\fomega_{i'}}\in
 \left\{
 {\left.\frac{k}{k+1}\ \right|\ k\in\bbZ\setminus\{-1,0\}}
 \right\}.
 \end{equation}
In this case, $\alpha=0$ automatically. Then there exist an open
dense subset ${\cM}^{\rm sym}\subset\bbR^n_{>0}$ and a nonempty open
subset ${\cM}\subset{\cM}^{\rm sym}$ {\rm(see definition \ref
{def:m:adm} and remark \ref {rem:sym})}, both invariant under
multiplication by any positive real number and having the following
properties. For ``almost any'' collection of planets' masses $\mu
m_i>0$, namely for any collection of planets' masses 
$\mu(m_1,\dots,m_n)\in\cM^{\rm sym}$ \(respectively for any
collection of planets' masses $\mu(m_1,\dots,m_n)\in\cM$, for example
satisfying the inequality
$|\kappa_{ii'}|c_{\kappa_{ii'}}m_{i'}>\sum_{l\ne
i,i'}|\kappa_{il}|c_{\kappa_{il}}m_l$, {\rm see (\ref {cs}), (\ref
{sii'}), (\ref {eq:cs0})}\), there exist numbers $\mu_0,\nu_0>0$ and
an open subset $U_0$ in the phase space of the planets containing the
phase orbits of all symmetric ``circular'' solutions \(respectively
all circular solutions\) of the collection of the Kepler problems for
planets with angular frequencies {\rm(\ref {chast})}, such that the
following condition holds. For any values $\mu,\nu,\widetilde
T,\widetilde\alpha\in\bbR$ of the form
 $$
 0<\left(\frac{\nu}{\nu_0}\right)^3\le\mu\le\mu_0, \qquad |\widetilde T-T|+|\widetilde\alpha|\le D\mu,
 $$
there exists no $(\widetilde T,\widetilde\alpha)$-periodic solution of
the $N+1$ body problem under consideration whose phase orbit has a
nonempty intersection with the direct product $U$ of $U_0$ and the
phase space of the satellites.
 \end {The}

In particular, the region $U$ does not contain the phase orbit of any
{\it symmetric $T$-periodic solution} (respectively $T$-periodic
solution).

Consider the planetary system with two planets, a partial case of the
three-body problem. In this case, the minimal positive period
$T_{\min}$ equals $\frac{2\pi}{|\fomega_2-\fomega_1|}$, thus the
condition (\ref {degen:3bod'}) means that the corresponding angle
$\alpha_{\min}=\frac{2\pi\fomega_1}{\fomega_2-\fomega_1}-2\pi k$
vanishes. We also observe that, in this case, the region $U$ in
theorem \ref {th:degen:sym}\ contains the whole two-dimensional torus
$\Lambda\o$ and $\cM^{\rm sym}=\cM=\bbR^2_{>0}$ (i.e.\ ``almost any''
means ``any''). Thus, theorems \ref {th1} and \ref {th:degen:sym} for
$N=n=2$ imply the following.

\begin {Cor} \label {cor:3bodies}
Consider the three-body problem of the type of planetary system with
two planets. Fix angular frequencies of planets $\fomega_1\ne0$,
$\fomega_2\ne0$, $|\fomega_1|\ne|\fomega_2|$, and consider the
two-dimensional torus $\Lambda\o$ corresponding to the
circular motions of planets with frequencies $\fomega_1,\fomega_2$. Put
$T:=\frac{2\pi}{|\fomega_2-\fomega_1|}$,
$\alpha=2\pi\frac{\fomega_1}{\fomega_2-\fomega_1}+2\pi k\in(-\pi,\pi]$
for a suitable $k\in\bbZ$. In dependence on the ratio of these
frequencies, one of the following statements holds.

{\rm($\exists$)} Suppose that the angular frequencies
$\fomega_1,\fomega_2$ does not satisfy the special resonance condition
{\rm(\ref {degen:3bod'})}. Then $\alpha\ne0$ and there exists a
number $\mu_0=\mu_0(m_1,m_2,\fomega_1,\fomega_2)>0$ such that, for any
values $\mu,\widetilde\fomega_1,\widetilde\fomega_2$,
$|\mu|+|\widetilde\fomega_1-\fomega_1|+|\widetilde\fomega_2-\fomega_2|\le\mu_0$,
there exists a two-dimensional torus
$\widetilde\Lambda=\widetilde\Lambda(\mu,\widetilde\fomega_1,\widetilde\fomega_2)$
that smoothly depends on the triple $(\mu,\widetilde\fomega_1,\widetilde\fomega_2)$,
coincides with the torus $\Lambda\o$ if
$(\mu,\widetilde\fomega_1,\widetilde\fomega_2)=(0,\fomega_1,\fomega_2)$
and has the following property. If $0<\mu\le\mu_0$ then the torus
$\widetilde\Lambda(\mu,\widetilde\fomega_1,\widetilde\fomega_2)$
is the phase
orbit of a symmetric $(\widetilde T,\widetilde\alpha)$-periodic
solution of the problem under consideration with parameters
$\widetilde T=\frac{2\pi}{|\widetilde\fomega_2-\widetilde\fomega_1|}$,
$\widetilde\alpha=2\pi\frac{\widetilde\fomega_1}{\widetilde\fomega_2-\widetilde\fomega_1}+2\pi\widetilde k$, for a suitable $\widetilde k\in\bbZ$.

{\rm($\nexists$)} Suppose that the angular frequencies
$\fomega_1,\fomega_2$ are in a special resonance {\rm(\ref
{degen:3bod'})}. Then $\alpha=0$ and, for any numbers $T,D>0$ with
$\frac{\fomega_2-\fomega_1}{2\pi}T\in\bbZ$, there exist a number
$\mu_0=\mu_0(m_1,m_2,\fomega_1,\fomega_2,T,D)>0$ and a neighbourhood
$U=U_{m_1,m_2,\fomega_1,\fomega_2,T,D}$ of the torus $\Lambda\o$ in the
phase space such that, for any parameter value $\mu\in(0,\mu_0]$ of
the three-body problem under consideration, $U$ does not contain any
$(\widetilde T,\widetilde\alpha)$-periodic orbit with parameters
$\widetilde T$, $\widetilde\alpha$ of the form
 $$
|\widetilde T-T|+|\widetilde\alpha|\le D\mu.
 $$
 \end {Cor}

The figure shows (for a fixed $\fomega_1\ne0$) regions in the plane
$\bbR^2$ consisting of pairs $(\kappa,\mu)\in\bbR\times\bbR_{>0}$,
$\kappa:=\frac{\fomega_1}{\fomega_2-\fomega_1}$,
$0<\mu\le\mu_0(m_1,m_2,\fomega_1,\fomega_2)$, such that there exists
(respectively does not exist) a relatively-periodic solution of the
three-body problem with parameters
$T=\frac{2\pi}{|\fomega_2-\fomega_1|}$, $\alpha=2\pi\kappa+2\pi k$, for
a suitable $k\in\bbZ$.

\begin{figure}
\unitlength = 3.8mm
\begin{center}
\begin{picture}(28,4.5)(-14,-.5)

\put(-14.5,0){\vector(1,0){29}} \put(14.5,-.8){$\kappa$}
\put(0,0){\vector(0,1){3.5}} \put(.3,3.5){$\mu$}

\put(-16,0){ \put(1.5,.5){$\dots$} }

\put(-12,0){
\qbezier[200](0,0)(.4,0.1)(.7,.9) \qbezier[200](0,0)(-.4,0.1)(-.7,.9)
\qbezier[200](.7,.9)(.8,1.2)(.7,1.5)
\qbezier[200](-.7,.9)(-.8,1.2)(-.7,1.5)
\qbezier[200](.7,1.5)(0,2.5)(-.7,1.5) \put(-.3,.7){$\nexists$}
\put(-.8,-.9){\small$-3$}
\qbezier[200](0,0)(.5,0)(1,.8) \qbezier[200](4,0)(3.5,0)(3,.8)
\qbezier[200](1,.8)(2,2.2)(3,.8) \put(1.7,.3){$\exists$} }

\put(-8,0){
\qbezier[200](0,0)(.4,0.1)(.7,.9) \qbezier[200](0,0)(-.4,0.1)(-.7,.9)
\qbezier[200](.7,.9)(.8,1.2)(.7,1.5)
\qbezier[200](-.7,.9)(-.8,1.2)(-.7,1.5)
\qbezier[200](.7,1.5)(0,2.5)(-.7,1.5) \put(-.3,.7){$\nexists$}
\put(-.8,-.9){\small$-2$}
\qbezier[200](0,0)(.5,0)(1,.8) \qbezier[200](4,0)(3.5,0)(3,.8)
\qbezier[200](1,.8)(2,2.2)(3,.8) \put(1.7,.3){$\exists$} }

\put(-4,0){ \put(-.8,-.9){\small$-1$}
\qbezier[200](0,0)(.5,0)(1,.8) \qbezier[200](4,0)(3.5,0)(3,.8)
\qbezier[200](1,.8)(2,2.2)(3,.8) \put(1.7,.3){$\exists$} }

\put(0,0){ \put(-.2,-.9){\small$0$}
\qbezier[200](0,0)(.5,0)(1,.8) \qbezier[200](4,0)(3.5,0)(3,.8)
\qbezier[200](1,.8)(2,2.2)(3,.8) \put(1.7,.3){$\exists$} }

\put(4,0){
\qbezier[200](0,0)(.4,0.1)(.7,.9) \qbezier[200](0,0)(-.4,0.1)(-.7,.9)
\qbezier[200](.7,.9)(.8,1.2)(.7,1.5)
\qbezier[200](-.7,.9)(-.8,1.2)(-.7,1.5)
\qbezier[200](.7,1.5)(0,2.5)(-.7,1.5) \put(-.3,.7){$\nexists$}
\put(-.2,-.9){\small$1$}
\qbezier[200](0,0)(.5,0)(1,.8) \qbezier[200](4,0)(3.5,0)(3,.8)
\qbezier[200](1,.8)(2,2.2)(3,.8) \put(1.7,.3){$\exists$} }

\put(8,0){
\qbezier[200](0,0)(.4,0.1)(.7,.9) \qbezier[200](0,0)(-.4,0.1)(-.7,.9)
\qbezier[200](.7,.9)(.8,1.2)(.7,1.5)
\qbezier[200](-.7,.9)(-.8,1.2)(-.7,1.5)
\qbezier[200](.7,1.5)(0,2.5)(-.7,1.5) \put(-.3,.7){$\nexists$}
\put(-.2,-.9){\small$2$}
\qbezier[200](0,0)(.5,0)(1,.8) \qbezier[200](4,0)(3.5,0)(3,.8)
\qbezier[200](1,.8)(2,2.2)(3,.8) \put(1.7,.3){$\exists$} }

\put(12,0){
\qbezier[200](0,0)(.4,0.1)(.7,.9) \qbezier[200](0,0)(-.4,0.1)(-.7,.9)
\qbezier[200](.7,.9)(.8,1.2)(.7,1.5)
\qbezier[200](-.7,.9)(-.8,1.2)(-.7,1.5)
\qbezier[200](.7,1.5)(0,2.5)(-.7,1.5) \put(-.3,.7){$\nexists$}
\put(-.2,-.9){\small$3$} \put(1.3,.5){$\dots$} }

\end{picture}
\end{center}

\vskip 0.5truecm
\begin{center}
\begin{tabular}{c} \small
 Fig.~1. A set of pairs $(\kappa=\frac{\fomega_1}{\fomega_2-\fomega_1},\mu)$ such that there exists ($\exists$) or \\ does not exist ($\nexists$) a relatively-periodic solution of the 3-body problem
\end{tabular}
\end{center}
\end{figure}

Due to corollary \ref {cor:3bodies}, for the planetary system with
two planets, the condition (\ref {degen:3bod'}) is false if and only
if there exists
a $(T,\alpha)$-periodic solution close to a ``circular''
$(T,\alpha)$-periodic motion. These $(T,\alpha)$-periodic solutions
were discovered already by H.\ Poincar\'e \cite {1} who called them
{\it solutions of the first kind}. In the degenerate case (\ref
{degen:3bod'}), Poincar\'e discovered periodic solutions
corresponding to elliptic motions of the planets, {\it solutions of
the second kind}.

\subsection{Method of finding relatively-periodic solutions of the $N+1$ body problem}

Let us describe our idea of constructing $(T,\alpha)$-periodic solutions that appear in Theorems \ref {th1} and \ref {th:ust}.
Consider the model system (\ref{eq:model:H:omega}), (\ref {eq:model:H}), (\ref{eq:model:omega}) and the corresponding torus $\Lambda\o\subset (\bbR^2)^{2N}$ formed by trajectories of ``generating'' solutions with angular frequencies (\ref {chast}) of the form (\ref {poryadki:chastot'}), (\ref {poryadki:chastot''}).
Consider also the unperturbed and perturbed systems, cf.\ (\ref {eq:unpert0}), (\ref {eq:unpert1}) and (\ref {eq:pert:planets}).
We construct solutions in three stages (\Sec \ref {par3:1:5}, \ref {sec:pert}, \ref {par3:1:4}). On the first stage (\Sec \ref {par3:1:5}), by using our construction (via the averaging method on a submanifold \cite {9,29}) of periodic solutions of the Hill problem \cite{K:vest13}, we construct an $N$-dimensional torus $\Lambda\subset (\bbR^2)^{2N}$, which is $O(\somega^2)$--close to the torus $\Lambda\o$, is formed by phase trajectories of the unperturbed system with angular frequencies (\ref {chast}) and depends smoothly on the small parameter $\somega\in[-\somega_0,\somega_0]$ and the tuple of real numbers $\Omega_{ij}$ ($1\le i\le n$, $0\le j\le n_i$). 
On the second stage (\Sec \ref {sec:pert} and \S\ref {subsec:ideas}), we fix a value $\somega\in(0,\somega_0]$ and a ``relatively resonance'' tuple of frequencies of the form (\ref {maxrez}).
Then, $|\mu|,|\nu|,|\rho|,|\varepsilon|\ll1$, we construct a torus $\widetilde\Lambda\subset (\bbR^2)^{2N}$, which depends smoothly on the parameters $T,\alpha,\mu,\nu,\rho,\varepsilon$ and contains phase orbits of all $(T,\alpha)$-periodic solutions of the perturbed problem near the torus $\Lambda$, moreover $\widetilde\Lambda=\Lambda$ if $\mu=\nu=\rho=\varepsilon=0$. We prove that exactly $2^{N-2}$ of these solutions are symmetric, depend smoothly on the indicated parameters, and if $\mu=\nu=\rho=\varepsilon=0$ then they coincide with the symmetric solutions of the unperturbed system whose phase orbits are contained in the torus $\Lambda$.
On the third stage (\Sec \ref {par3:1:4}), we show (in Lemma \ref {lem3:1}) that the transformation (\ref {eq:zamena}) brings the $N+1$ body problem to a $N+3$-parameter subfamily of the $N+5$-parameter family of perturbed systems (\ref {eq:pert:planets}), where the parameters in the subfamily satisfy the relation (\ref {eq:rho:varepsilon}).

It would be interesting to prove analogues of Theorems \ref {th1}--\ref {th:degen:sym} 
for the $N+1$ body problems on the sphere and the Lobachevskiy plane. We beleave that our methods should work in this case, due to periodicity of solutions of the Kepler problems on these surfaces (cf.\ the work \cite{zkf} and references therein).

\section {The averaging method for a class of systems with slow and fast variables} \label {sec:quick:slow}

In \Sec\Sec\ref {par3:1:5}--\ref {par:proofs},
we will derive theorems \ref {th1}--\ref {th:degen:sym} from the next theorems \ref
{th1'}--\ref {th:nondeg} on periodic solutions of dynamical systems
having the following special form.

Let $p:M\to M_0$ be a surjective submersion (e.g., a locally trivial fibre bundle) of smooth
manifolds. Suppose, we are given a symplectic structure $\widehat\oomega_0$ on $M_0$ and a closed 2-form $\oomega_1$ on $M$ such that 
 $$
T_xM=\left(\Ker(dp)|_x\right)\oplus\left(\Ker\oomega_1|_x\right)
 $$ 
at every point $x\in M$. Thus, the 2-form $\oomega_1$ defines a symplectic structure on each fibre, moreover its field of kernels is transvercal to the fibre at every point of $M$ and defines a symplectic flat connection of the fibre bundle $p$. The question of the existence of such a connection on a given symplectic fibre bundle (whose base is not necessarily a symplectic manifold) was studied in \cite {GLSW} (cf.\ also \cite {GLS}). 
The pair $(
p^*\widehat
\oomega_0,\oomega_1)$ is called a {\it splitted symplectic strucure} on the fibred manifold $(M,M_0,p)$.

Suppose that functions $\widehat H_0\in C^\infty(M_0)$, $H_1\in C^\infty(M)$ and a constant $\lambda\in\bbR$ are given. Put 
$$
\oomega_0:=p^*\widehat\oomega_0, \qquad 
H_0:=\widehat H_0\circ p, \qquad 
H:=H_0+\lambda H_1.
$$
A vector field $v=v_{H,H_1}$ on $M$ will be called {\em $\lambda$-Hamiltonian} if 
 $$
 \left(
\oomega_0(\cdot,v)-dH\right)|_{\Ker\oomega_1}=0, \qquad
 \left(\oomega_1(\cdot,v)-dH_1\right)|_{\Ker dp}=0.
 $$
In this case, the dynamical system $\dot x(t)=v(x(t))$ on $M$ will be called {\em $\lambda$-Hamiltonian} and denoted by
 \begin{equation} \label {eq:lam:Ham}
 (M,M_0,p;
\oomega_0,\oomega_1;H,H_1)^\lambda,
 \end{equation}
moreover the function $H$ will be called {\em the Hamilton function}, and $H_1$ will be called {\em the $\lambda$-Hamilton function} 
of the system (\ref {eq:lam:Ham}). If $\lambda\ne0$
then the system (\ref {eq:lam:Ham}) is equivalent to the Hamiltonian
system $(M,\oomega:=
\oomega_0+\lambda\oomega_1,H)$:
 $$
 (M,M_0,p;
\oomega_0,\oomega_1;H,H_1)^\lambda \cong (M,\oomega:=
\oomega_0+\lambda\oomega_1,H).
 $$
Here the symbol $\cong$ denotes the equivalence of
($\lambda$-)Hamiltonian systems, i.e.\ the coincidence of the
corresponding ($\lambda$-)Hamiltonian vector fields. If $\lambda=0$
then the system (\ref {eq:lam:Ham}) is a {\em semidirect product} of Hamiltonian systems
(such systems are studied by Yu.\,M.\ Vorobiev~\cite {DV}; they
include e.g.\ the restricted three-body problem, i.e.\ the three-body problem with $N=2$, $n=1$, $\nu=0$).

Denote by $g^t_{H,H_1}$ the flow of the $\lambda$-Hamiltonian vector
field $v_{H,H_1}$. Similarly to the case of Hamiltonian systems, the
flow of the field $v_{H,H_1}$ always (even for $\lambda=0$) preserves the $2$-form $\oomega$
and the Hamilton function $H$.

\begin{nnExa} 
Suppose that $M=M_0\times M_1$ is a direct product, moreover (similarly to $\oomega_0$ and $H_0$) a 2-form $\oomega_1$ and a function $H_1$ are ``lifted'' from $M_1$ (i.e., they have the form $\oomega_1=p_1^*\widehat\oomega_1$ and 
$H_1=\widehat H_1\circ p_1$, for some 2-form $\widehat\oomega_1$ on
$M_1$ and a function $\widehat H_1\in C^\infty(M_1)$, where $p_1:M\to M_1$ is the projection).
Then the $\lambda$-Hamiltonian system (\ref {eq:lam:Ham}) for any $\lambda\in\bbR$ is equivalent to the Hamiltonian system $(M,
\oomega_0+\oomega_1,\widehat H_0\circ p+H_1)$, i.e.\ to the {\em direct product} of the Hamiltonian systems
$(M_0,\widehat\oomega_0,\widehat H_0)$ and
$(M_1,\widehat\oomega_1,\widehat H_1)$:
 \begin{equation} \label {eq:direct}
 (M,M_0,p;
\oomega_0,\oomega_1;H,H_1)^\lambda
 \cong(M,
\oomega_0+\oomega_1,
H_0
+H_1).
 \end{equation}
\end{nnExa}

Let us describe a class of dynamical systems {\em with slow and fast
variables}. Consider a two-parameter family of $\varepsilon$-Hamiltonian systems on $M$ with a splitted symplectic structure $(
\oomega_0,\oomega_1)$, the
Hamiltonian $H=\somega
H_0
+\varepsilon H_1$ and $\varepsilon$-Hamiltonian $H_1$
where $|\varepsilon|,|\somega|\ll1$.
Here the 2-forms $\widehat\oomega_0,\oomega_1$ and the functions $\widehat H_0$ and $H_1$
depend, in general, on the small
parameters $\varepsilon,\somega$ and possibly on some other parameters
of the system, moreover some relations between parameters may be
posed. The local coordinates of a point $x_0\in M_0$ are ``slow
variables'', while the local coordinates of a point ``on a fibre''
are ``fast variables'' of the system.


\begin{nnExa} The following mechanical problems can be transformed (by changing phase variables) to systems with slow and fast variables.

(i) The problem on slow motions of a charged particle in a
magnetic field on a symplectic Riemannian manifold $(M_0,\widehat\oomega_0,g)$.
Let a magnetic field be given by a closed 2-form $
\widehat\oomega_0$ on $M_0$, and an electric field be given by a smooth function $U(\q)$ on $M_0$ (called the electric potential). Motions of a charged particle in an electro-magnetic field are described by the Hamiltonian system 
$$
(M=T^*M_0,\ \oomega=d\p\wedge d\q+\oomega_0,\ H=T(\q,\p)+U(\q))
$$
on $M=T^*M_0$, where $\oomega_0:=p^*\widehat\oomega_0$, $p:T^M\to M$ is the projection.
Here $\q=(q^i)_{i=1}^n$ are local coordinates on $M_0$, $\p=(p_j)_{j=1}^n$ are conjugated momenta, $T(\q,\p)=\frac12 g^{ij}(\q)dp_i dp_j$ is the kinetic energy.
Suppose that the magnetic field $\widehat\oomega_0$ is nondegenerate, and the local coordinates $\q$ are canonical for it, i.e.\ $\widehat\oomega_0=\sum_{j=1}^{n/2}dq^{2j-1}\wedge dq^{2j}$.
Let us change the coordinates $\q\to\q+J\p=:x_0$ (called the {\em guiding-centre transformation}, cf.\ \cite [\S3 and 6]{Neishtadt} or \cite[\S3]{kud:pod}, where $J$ denotes the matrix of $\oomega$) and rescale the momenta $\p\to\p/\somega=:x_1$. Then the symplectic structure and the Hamilton function take the form 
$$
\oomega = \oomega_0 - \varepsilon\oomega_1, \qquad 
H = U(x_0-\somega Jx_1) + \varepsilon T(x_0-\somega Jx_1,x_1)
$$
where $\omega_1=\sum_{j=1}^{n/2}dx_1^{2j-1}\wedge dx_1^{2j}$ and the small parameters $\omega,\varepsilon$ are related by the condition $\varepsilon=\somega^2$. 
If $U\equiv0$, i.e.\ there is no electric field, 
then the obtained system is a system with slow and fast variables. Here the local coordinates $x_0$ on $M_0$ are ``slow'', while the local coordinates $x_1$ on fibres are ``fast''.

(ii) The planar $N+1$ body problem of the type of a planetary system with satellites reduces (due to Lemma \ref {lem3:1}) to the ``perturbed'' system (\ref {eq:pert:planets}), which is a system with slow and fast variables. In this system, $M=M_0\times M_1$ is the direct product of the phase spaces of ``scaled planets'' and ``satellites'', 
while the small parameters are related by (\ref {eq:rho:varepsilon}).
See Example \ref {exa:slow:fast:planets} for details.
\end{nnExa}

From now on, we assume that $M=M_0\times M_1$ is the direct product
and $\oomega_1=p_1^*\widehat\oomega_1$. 

Suppose that each symplectic manifold $(M_i,\widehat\oomega_i)$ is
equipped with the Hamiltonian action of a circle
$SO(2)=S^1=\bbR/2\pi\bbZ$ with the Hamiltonian function $I_i$,
$i=0,1$. The system (\ref {eq:lam:Ham}) will be called {\em
$S^1$-symmetric} (or {\em $SO(2)$-symmetric}) if the functions
$H,H_1$ are invariant under the diagonal action of the circle on $M$
(i.e.\ invariant under the flow of the $\lambda$-Hamiltonian field
$v_{I,I_1}$ on $M$ where $I=I_0+\lambda I_1$). All solutions of this
system that differ by shifts along the commuting vector fields
$v_{H,H_1}$ and $v_{I,I_1}$ will be regarded as a single solution,
and the union of their phase trajectories will be called {\it the
phase orbit} of this solution. Let $T,\alpha\in\bbR$, $T\ne0$. A
solution $\gamma(t)$ of a $S^1$-symmetric system will be called {\em
$(T,\alpha)$-periodic} if it is defined on the whole time-axis, and
$\gamma(t)=g^T_{H,H_1}g^{-\alpha}_{I,I_1}(\gamma(t))$ for some (and,
hence, for any) $t\in\bbR$.

\begin{scDef} \label {def:reversible}
A $S^1$-symmetric system (\ref {eq:lam:Ham}) will be called {\em
reversible} (or {\em $O(2)$-symmetric}) if each $M_i$ is equipped
with an anti-canonical involution $\J_i:M_i\to M_i$ preserving the
function $I_i$ (i.e.\ $\J_i^*\widehat\oomega_i=-\widehat\oomega_i$ and
$I_i\circ \J_i=I_i$), $i=0,1$, moreover the functions $H,H_1$ are
invariant under the (component-wise anti-canonical) involution
$\J:=\J_0\times \J_1:M\to M$. A solution $\gamma(t)$ of the reversible
system will be called {\em symmetric} if it is defined on a
time-interval $(t_0-\varepsilon,t_0+\varepsilon)\subset\bbR$ and
 \begin{equation} \label {eq:def:sym}
 \gamma(t_0)=g^{\varphi_0}_{I,I_1}\J g^{-\varphi_0}_{I,I_1}(\gamma(t_0)) \quad \mbox{for some $\varphi_0\in\bbR\mod2\pi$} 
 \end{equation}
(and, hence, $\gamma(2t_0-t)=g^{\varphi_0}_{I,I_1}\J g^{-\varphi_0}_{I,I_1}(\gamma(t))$ for any $t\in(t_0-\varepsilon,t_0+\varepsilon)$).
\end{scDef}

Let us describe the {\it model system} on $M=M_0\times M_1$:
 \begin{equation} \label {eq:model}
 (M,M_0,p;\oomega_0,\oomega_1;\somega H_0,H_1)^0 \cong (M,\oomega_0+\oomega_1,\somega H_0+H_1),
 \end{equation}
$\somega\in\bbR$, where $H_1=\widehat H_1\circ p_1$ for some function $\widehat H_1\in C^\infty(M_1)$, cf.\ (\ref {eq:lam:Ham}) and (\ref {eq:direct}). 
We will assume that each of the Hamiltonian systems
$(M_0,\widehat\oomega_0,\widehat H_0)$ and
$(M_1,\widehat\oomega_1,\widehat H_1)$ in the system (\ref
{eq:model}) is the direct product of $S^1$-symmetric Hamiltonian
systems:
$$
 (M_0,\widehat\oomega_0,\widehat H_0)=\prod_{i=1}^n(M_{i0},\widehat\oomega_{i0},H_{i0}),
 \qquad
 (M_1,\widehat\oomega_1,\widehat H_1)=\prod_{i=1}^n\prod_{j=1}^{n_i}
 (M_{ij},\widehat\oomega_{ij},H_{ij}).
$$
Moreover each factor $M_{ij}=S^1\times(a_{ij},b_{ij})\times\bbR^2$ is equipped with
coordinates $\varphi_{ij}\mod2\pi,I_{ij},q_{ij},p_{ij}$ such that
$\widehat\oomega_{ij}=dI_{ij}\wedge d\varphi_{ij}+dp_{ij}\wedge
dq_{ij}$, the action of the circle on $(M_{ij},\widehat\oomega_{ij})$
is given by the Hamiltonian $I_{ij}$, and the involution $\J$ acts
component-wise in the form
$(\varphi_{ij},I_{ij},q_{ij},p_{ij})\mapsto(-\varphi_{ij},I_{ij},q_{ij},-p_{ij})$.
In particular,
$$
\oomega_0=\sum_{i=1}^n \oomega_{i0}, \quad 
\oomega_1=\sum_{i=1}^n\sum_{j=1}^{n_i} \oomega_{ij},
$$
$$
H_0=\sum_{i=1}^nH_{i0}, \quad H_1=\sum_{i=1}^n\sum_{j=1}^{n_i}H_{ij},
\qquad  H_{ij}=H_{ij}(I_{ij},q_{ij},p_{ij})
$$
where, from now on, we use the same notation for a function (or a 2-form) and its
lift, by abuse of notation. We will assume that each coordinate cylinder
$S^1\times(a_{ij},b_{ij})\times\{0\}\times\{0\}\subset M_{ij}$ consists of
{\it relative equilibrium points}, i.e.\ the co-vector $dH_{ij}$ is
proportional to $dI_{ij}$ at any its point (with coefficient
depending on the point), $1\le i\le n$, $0\le j\le n_i$. Therefore
the $2N$-dimensional symplectic submanifold
 \begin{equation} \label {eq:subman}
\prod_{i=1}^n\prod_{j=0}^{n_i}S^1\times(a_{ij},b_{ij})\times\{0\}\times\{0\}\subset
M_0\times M_1
 \end{equation}
is invariant under the flow of the model system (\ref {eq:model}) and
it is fibred by invariant $N$-dimensional tori
$\prod_{i=1}^n\prod_{j=0}^{n_i}S^1\times\{(I_{ij},0,0)\}$ where
$N:=\sum_{i=1}^n(1+n_i)$. Those solutions of the model system whose
phase orbits are contained in the invariant submanifold (\ref
{eq:subman}) will be called the {\em generating solutions}. Consider
one of these $N$-dimensional tori, $\Lambda\o$, and a
($4N-2$)-dimensional ``cross section'' $\Sigma$ in the
$4N$-dimensional phase space $M_0\times M_1$, which is transversal to
the two-dimensional phase orbits of the generating solutions
contained in the $N$-torus $\Lambda\o$.

Let us describe the {\em unperturbed system}. Suppose that a
$S^1$-invariant function
$F_{ij}=F_{ij}(I_{i0},q_{i0},p_{i0},\varphi_{ij}-\varphi_{i0},I_{ij},q_{ij},p_{ij})$
is given on each direct product $M_{i0}\times M_{ij}$, $1\le j\le
n_i$. Put
 \begin{equation} \label {eq:Phi}
\Phi:=\sum_{i=1}^n\sum_{j=1}^{n_i}F_{ij}.
 \end{equation}
As {\em the unperturbed system}, we will regard the 0-Hamiltonian
system
 \begin{equation} \label {eq:unpert}
 (M_0\times M_1,M_0,p;\oomega_0,\oomega_1;\somega H_0,H_1+\somega^2\Phi)^0
 \end{equation}
with parameter $0<\somega\ll1$. Then the $S^1$-action is given via the
Hamiltonian function $I_0=\sum_{i=1}^nI_{i0}$ and the 0-Hamiltonian
function $I_1=\sum_{i=1}^n\sum_{j=1}^{n_i}I_{ij}$.

Both systems described above: the model one (\ref {eq:model}) and the
unperturbed one (\ref {eq:unpert}), are systems with slow and fast
variables, because of a small factor $\omega$ in their Hamiltonian
function $\omega H_0$.

In the following theorems, we suppose that each Hamiltonian system $(M_{ij},\oomega_{ij},H_{ij})$ possesses the following properties of periodicity and nondegeneracy:

1) all solutions of the Hamiltonian system $(M_{ij},\oomega_{ij},H_{ij})$ are periodic with periods $T_{ij}\circ H_{ij}$, for some functions $T_{ij}=T_{ij}(h)\ne0$;

2) the functions $H_{ij}=H_{ij}(I_{ij},q_{ij},p_{ij})$ satisfy the following nondegeneracy conditions at the points $(I_{ij},0,0)\in(a_{ij},b_{ij})\times\{0\}\times\{0\}$:
 \begin{equation} \label {eq:isoen}
 \frac{\partial H_{ij}}{\partial I_{ij}}=\Omega_{ij}(I_{ij}), \ \
 \frac{\partial^2H_{ij}}{\partial I_{ij}^2}\ne0
 = \frac{\partial H_{ij}}{\partial q_{ij}}=\frac{\partial H_{ij}}{\partial p_{ij}}, \ \
 \det \frac{\partial^2H_{ij}}{\partial(q_{ij},p_{ij})^2}=\Omega_{ij}^2(I_{ij}),
 \end{equation}
where $\Omega_{ij}(I_{ij}):=2\pi/T_{ij}(H_{ij}(I_{ij},0,0))$, $1\le
i\le n$, $0\le j\le n_i$ (in particular, the circles $\{I_{ij}\}\times S^1\times\{0\}\times\{0\}$ are elliptic nondegenerate relative equilibria of this $S^1$-symmetric system).

Suppose also that the collection of ``angular frequencies'' $\Omega_{ij}=\Omega_{ij}(I_{ij}\o)$ (for some frequencies $I_{ij}\o\in[a_{ij}\o,b_{ij}\o]$) and the real number $\somega>0$ satisfy the following conditions:

3) the ``relative resonance'' condition 
 \begin{equation} \label {maxrez'}
 \somega^{\overline{1-j}}\Omega_{ij}=\fomega_1+k_{ij}\frac{2\pi}{T},
 \qquad 1\le i\le n, \quad 0\le j\le n_i,
 \end{equation}
where the integers $k_{ij}\in\bbZ$ are not $0$ simultaneously, $\overline\ell:=\max\{0,\ell\}$,
$\fomega_1:=\somega\Omega_{10}$ and $T>0$;

4) the nondegeneracy condition (for $\alpha:=\fomega_1T$ and $C_1>|\Delta_{ij}|$, cf.\ (\ref {eq:Delta:ij}) below)
 \begin{equation} \label {eq:nondeg}
 \alpha \not\in 2\pi\bbZ \ \mbox{ for }N=n, \quad
 \alpha \not\in\left[-C_1\somega^2T,C_1\somega^2T\right]+2\pi\bbZ \ \mbox{ for }N>n;
 \end{equation}

4') the more delicate nondegeneracy condition (with $C_2>0$ in Theorem \ref {th:mult} below),
 \begin{equation} \label {eq:nondeg'}
 \alpha:=\fomega_1T \not\in2\pi\bbZ, \qquad
 \alpha+\Delta_{ij}\somega^2T\not\in\left[-C_2\somega^3T,C_2\somega^3T\right]+2\pi\bbZ
 \end{equation}
for $1\le i\le n$, $1\le j\le n_i$. Here $\Delta_{ij}$ is expressed in terms of the first and second partial derivatives of the function $\langle\widehat
F_{ij}\o\rangle=\langle\widehat F_{ij}\o\rangle(q_{ij},p_{ij})$ at the point $(0,0)$ and, in particular, 
 \begin{equation} \label {eq:Delta:ij}
\Delta_{ij}:=\frac{\Omega_{ij}}2
\Tr\left(\left(\frac{\partial^2H_{ij}(I_{ij}\o,0,0)}{\partial(q_{ij},p_{ij})^2}\right)^{-1}
\frac{\partial^2\langle\widehat F_{ij}\o\rangle(0,0)}{\partial(q_{ij},p_{ij})^2}\right)
 \ \mbox{при}\ d\langle\widehat F_{ij}\o\rangle(0,0)=0,
 \end{equation}
where the function $\langle\widehat F_{ij}\o\rangle$ is obtained by averaging the function 
$F_{ij}|_{\{(I_{i0}\o,0,0)\}\times H_{ij}^{-1}(H_{ij}(I_{ij}\o,0,0))}=:\widehat F_{ij}\o$ along 
$\frac{2\pi}{\Omega_{ij}}$-periodic solutions of the system 
$(M_{ij},\oomega_{ij},H_{ij})$.

\begin{The} [(on the number of relatively-periodic solutions)] \label {th1'}
\ Suppose that each Hamiltonian system $(M_{ij},\oomega_{ij},H_{ij})$ possesses the properties 1) and 2) of periodicity and nondegeneracy, and that the number of the systems is $N\ge2$. Then, for any collection of segments 
$[a_{ij}\o,b_{ij}\o]\subset(a_{ij},b_{ij})$, there exist real numbers
$\omega_0,C_1,C_2>0$ such that the following conditions hold for any
$\somega\in(0,\omega_0]$. Suppose that, for some numbers $I_{ij}\o\in[a_{ij}\o,b_{ij}\o]$, the real number $\somega$ and the collection of ``angular frequencies'' $\Omega_{ij}=\Omega_{ij}(I_{ij}\o)$ satisfy the ``relative resonance'' condition 3), as well as either the nondegeneracy condition 4) or the more delicate nondegeneracy condition 4'). 
Then there exists $\varepsilon_0>0$ such that, for $0<\varepsilon\le\varepsilon_0$, any $S^1$-symmetric \(``perturbed''\) Hamiltonian system 
 \begin{equation} \label {eq:pert}
 (M,\oomega_0+\varepsilon\oomega_1,\widetilde H) \cong
 (M,M_0,p;\oomega_0,\oomega_1;\widetilde H,\widetilde H_1+\somega^2\widetilde\Phi)^\varepsilon
 \end{equation}
on $M=M_0\times M_1$ has at least $N-1$ $(T,\alpha)$-periodic
solutions close to generating solutions with angular frequencies
$\somega\Omega_{i0},\Omega_{ij}$, provided that $\widetilde
H:=\somega\widetilde H_0+\varepsilon\widetilde
H_1+\somega^2\varepsilon\widetilde\Phi$, the function $\widetilde H_0$
``projects'' to the factor $M_0$, and $\|\widetilde
H_0-H_0\|_{C^2}+\|\widetilde
H_1-H_1\|_{C^2}+\|\widetilde\Phi-\Phi\|_{C^2}\le\varepsilon_0$.
Moreover there exist at least $2^{N-2}$ such solutions counted with
multiplicities. The phase orbits of all such solutions are contained in
some $N$-dimensional torus $\widetilde\Lambda$ that is
$O(\somega)$-close \(and even $O(\somega^2)$-close in the case
$d\langle F_{ij}\o\rangle(0,0)=0$\) to the torus
$\Lambda\o:=\prod_{i=1}^n\prod_{j=0}^{n_i}S^1\times\{(I_{ij}\o,0,0)\}$
with respect to a $C^1$-norm. The intersection points of these phase
orbits with a cross section $\Sigma$ \(i.e.\ with a 
transversal surface to the two-dimensional phase orbits of generating
solutions on $\Lambda\o$\) coincide with critical points of the
function $\Psi|_{\widetilde\Lambda\cap\Sigma}$ where $\Psi$ is the
generating function of the perturbed succession map $g_{\widetilde
H,\widetilde H_1+\somega^2\widetilde\Phi}^T g_{\widetilde
I,I_1}^{-\alpha}:M\to M$, $\widetilde I:=I_0+\varepsilon I_1$ {\rm
(see Definition \ref {def:gen:func} below)}.

Suppose that the $0$-Hamiltonian system {\rm(\ref {eq:unpert})} is contained in a $r$-parameter family of $\varepsilon$-Hamiltonian systems of the form {\rm(\ref {eq:pert})} such that the functions $\widetilde H_0,\widetilde H_1,\widetilde\Phi$ depend smoothly {\rm(}or analytically{\rm)} on $r\ge2$ parameters $(\varepsilon,\somega,\dots)$, moreover $\widetilde H_0=H_0$, $\widetilde H_1=H_1$, $\widetilde\Phi=\Phi$ for the zero values of parameters $(0,0,\dots)$. Then, for $\varepsilon=0$, there exists a $(N+r-1)$-parameter family of $N$-dimensional tori $\Lambda$ that are invariant w.r.t.\ the corresponding $0$-Hamiltonian systems and depend smoothly {\rm(}respectively, analytically{\rm)} on the tuple $\Omega_{**}$ of ``angular frequencies'' $\Omega_{ij}=\Omega_{ij}(I_{ij})$ and the indicated tuple of parameters apart from $\varepsilon$ {\rm(}in the domain $|(\somega,\dots)|<\somega_0(\Omega_{**})${\rm)}, moreover $\Lambda=\Lambda\o$ if $\varepsilon=\somega=\ldots=0$.
Every torus $\Lambda$ of this family corresponding to a ``relatively resonance'' tuple of ``angular frequencies'' $\Omega_{ij}$ and a real number $\somega$ {\rm(}with $\somega\ne0$ if $n>1${\rm)}, is contained in a $(2+r)$-parameter family of $N$-dimensional tori $\widetilde\Lambda$, depending smoothly {\rm(}respectively, analytically{\rm)} on $T,\alpha$ and the indicated tuple of parameters {\rm(}where $|\varepsilon|<\varepsilon_0(T,\alpha,k_{**},\somega,\dots)${\rm)}.
If $\varepsilon=0$ then the torus $\widetilde\Lambda$ coincides with the torus $\Lambda$. If $\varepsilon,\somega>0$ then the torus $\widetilde\Lambda$ possesses the above properties w.r.t.\ the $(T,\alpha)$-periodic solutions of the ``perturbed system''.
\end{The}

Similarly to \Sec\ref {subsec:period} (B1), one can show that, for any
collection $(\a\o,\b\o)$ of real numbers $a_{ij}\o,b_{ij}\o$ under
consideration and for small enough $0<\somega\ll1$, the period $T$ can
take an arbitrary value of the form $T\ge2\pi a_0(\a\o,\b\o)/\omega$,
hence the quantity $\somega^2T$ can be arbitrarily small. Thus any of
the nondegeneracy conditions (\ref {eq:nondeg}) and (\ref
{eq:nondeg'}) can always be fulfilled.

\begin{The} [(on symmetric relatively-periodic solutions)] \label {th:sym}
Suppose that, under the hypothesis of Theorem {\rm\ref {th1'}}, each
of three systems: the model system {\rm(\ref {eq:model})}, the
unperturbed system {\rm(\ref {eq:unpert})} and the perturbed system
{\rm(\ref {eq:pert})} is reversible {\rm(}Definition {\rm\ref {def:reversible})}. 
Then the perturbed system
{\rm(\ref {eq:pert})} admits exactly $2^{N-2}$ symmetric
$(T,\alpha)$-periodic solutions that are $O(\somega)$-close \(and
even $O(\somega^2)$-close in the case $d\langle
F_{ij}\o\rangle(0,0)=0$\) to the generating symmetric solutions with
the angular frequencies under consideration. Each of these $2^{N-2}$
solutions smoothly depends on the pair of parameters $(T,\alpha)$.
Moreover the function $\Psi|_{\widetilde\Lambda}$ is an even function
in the collection $\vvarphi$ of angular frequencies
$\varphi_{ij}|_{\widetilde\Lambda}$, and the phase orbits of the
symmetric $(T,\alpha)$-periodic solutions pass through the points
$\vvarphi$ of the torus $\widetilde\Lambda$ with the property
$\vvarphi=-\vvarphi$.
\end{The}

\begin{The} [(on stability of a relatively-periodic solution)] \label {th:stab}
Under the hypothesis of Theorem {\rm\ref {th1'}}, suppose that all
the numbers $\frac{\partial^2H_{ij}}{\partial
I_{ij}^2}(I_{ij}\o,0,0)$ have the same sign, e.g.\ negative
\(respectively positive\). Suppose also that either the 
strong nondegeneracy condition holds:
 $$
 \alpha\not\in\pi\bbZ \ \mbox{ for }N=n, \quad
 \alpha\not\in\left[-C_1\somega^2T,C_1\somega^2T\right]+\pi\bbZ \ \mbox{ for }N>n
 $$
and the following condition of having the same sign holds: all the signs
 \begin{equation} \label {eq:eta:ij}
\eta_{ij}:={\rm sgn\,}\left(\Omega_{ij}\Tr\frac{\partial^2H_{ij}(I_{ij}\o,0,0)}{\partial(q_{ij},p_{ij})^2}\right),
 \quad 1\le i\le n, \ 0\le j\le n_i,
 \end{equation}
are equal, or the following 
more delicate condition holds: for any set of real numbers $\alpha_{ij}\in\bbR$, $1\le i\le n$,
$0\le j\le n_i$, such that
 $$
 \alpha_{i0}=\eta_{i0}\alpha, \
 |\alpha_{ij}-\eta_{ij}(\alpha+\Delta_{ij}\somega^2T)|\le C_2\somega^3T,
 \quad 1\le i\le n, \ 1\le j\le n_i,
 $$
the sum of any two, possibly coinciding, numbers of the set does not
belong to the set $2\pi\bbZ$. Then the \($(T,\alpha)$-periodic due to
Theorem {\rm\ref {th1'}}\) phase orbit of the perturbed system
{\rm(\ref {eq:pert})} passing through any nondegenerate critical
point of local minimum \(respectively maximum\) of the function
$\Psi|_{\widetilde\Lambda\cap\Sigma}$ is orbitally structurally
stable in linear approximation {\rm(see Definition \ref {def5})}.
\end{The}

In 
all the next statements of the present section (cf.\ Theorems \ref {th:nondeg}, \ref {th:mult} and Lemma \ref {lem:mult}), the nondegeneracy condition (\ref {eq:nondeg}) is not assumed to be fulfilled.

\begin{The} [(on ``gaps'' in families of relatively-periodic solutions)]
\label {th:nondeg}
Suppose that, under the hypothesis of Theorem {\rm\ref {th1'}}, the
number $\somega\in(0,\somega_0]$ and the collection of angular
frequencies $\Omega_{ij}=\Omega_{ij}(I_{ij}\o,0,0)$ satisfy 
the conditions {\rm 1), 2), 3)} of periodicity, nondegeneracy and ``relative resonance'' \(but not necessary the nondegeneracy condition {\rm 4)}\).
Suppose that $\alpha\in2\pi\bbZ$ and that a smooth function $R_0$ on $M$
``projects'' to the factor $M_0$. Let us fix the two-dimensional
torus $\gamma\subset\Lambda\o$ corresponding to a $T$-periodic
solution of the model system {\rm(\ref {eq:model})}. Suppose that
some \(and, hence, any\) point of $\gamma$ is not a critical point of
the function $\langle R_0\o\rangle$ that is obtained by averaging 
the function $R_0\o:=R_0|_{\cap_iH_{i0}^{-1}(H_{i0}(I_{i0}\o,0,0))}$
along $T$-periodic solutions of the model system. Then, for any real
number $D>0$ and a neighbourhood $U_1\subset M_1$ of the projection $p_1(\Lambda\o)$ of the torus $\Lambda\o$ in $M_1$ with a compact closure $\overline U_1$, 
there exist a real number $\mu_0>0$ and a neighbourhood
$U_0$ of the projection $p(\gamma)$ of the two-dimensional torus
$\gamma$ to $M_0$ such that, for any $\varepsilon>0$,
$\somega\varepsilon/\mu_0\le\mu\le\mu_0$ and $|\widetilde T-T|+|\widetilde\alpha|\le D\mu$, the following holds.
The neighbourhood $U:=U_0\times U_1$ of the torus $\gamma$ does not contain any
$(\widetilde T,\widetilde\alpha)$-periodic solution of the perturbed system
{\rm(\ref {eq:pert})}, provided that $\widetilde H_0=H_0+\mu\widetilde R_0$,
the function $\widetilde R_0$ ``projects'' to the factor $M_0$, the function $\widetilde H_1$ ``projects'' to the factor $M_1$, $\|\widetilde R_0-R_0\|_{C^2}\le\mu_0$ and $\|\widetilde\Phi\|_{C^2}\le D$.
\end{The}

 \begin{nnExa} \label {exa:slow:fast:planets}
An example of ``perturbed'' systems of the form (\ref {eq:pert}) is the Hamiltonian system (\ref {eq:pert:planets}), (\ref {eq:rho:varepsilon}), which is the main system in this paper. Due to Lemma \ref {lem3:1}, the planar $N+1$ body problem reduces to this system. In this system, $M=M_0\times M_1$ is the direct product of phase spaces of ``scaled planets'' and ``satellites'', $\oomega_0=p^*\widehat\oomega_0$, $\oomega_1=p_1^*\widehat\oomega_1$, $\oomega=\oomega_0+\varepsilon\oomega_1$, and
 $$
\widetilde H=\somega\widetilde H_0(x_0;\mu)+\varepsilon\left(\widetilde
H_1(x_1;\nu)+\somega^2\widetilde\Phi(x_0,x_1;\mu,\nu,\rho)\right)
 $$
is the Hamilton function, where $\somega,\varepsilon,\mu,\nu,\rho\in\bbR$ are small parameters of the system (and $m_i,m_{ij}$ are regarded as constants for simplicity), $(x_0,x_1)\in M_0\times M_1$. The family of indicated systems (with all possible values of the parameters) contains the system (\ref {eq:pert:planets}), (\ref {eq:rho:varepsilon}) as a subfamily, where the small parameters in the subfamily are positive and related by (\ref {eq:rho:varepsilon}). The corresponding model and unperturbed systems coincide with the systems (\ref {eq:model:H:omega}) and (\ref {eq:unpert0}), (\ref {eq:unpert1}).
The planar $N+1$ body problem is $S^1$-symmetric and reversible. We will derive theorems \ref {th1}--\ref {th:degen:sym} from theorems \ref {th1'}--\ref {th:nondeg}.
 \end{nnExa}

\subsection{Proof of theorems \ref {th1'} and \ref {th:sym} and a scheme of proof of Theorems \ref {th:stab} and \ref {th:nondeg}} \label {subsec:ideas}
Let us describe two main stages of the proof of theorems \ref {th1'}--\ref {th:nondeg}.

{\it Stage one} is based on 
the averaging method on a submanifold 
\cite {8,9,29} (a similar assertion for 
systems with slow and fast variables see in \cite{Vorob11}. 
Using this method, we study the (0-Hamiltonian) unperturbed system (\ref {eq:unpert}) taking into account that it is close to the ``super-integrable'' model system (\ref {eq:model}). Namely, at first, we describe $(T,\alpha)$-periodic solutions of the
unperturbed system (\ref {eq:unpert}) that are close to the
generating solutions of the model system (\ref {eq:model}). At
second, we study the linearization of the ``succession map'' for
these solutions (see Theorem \ref {th:mult}).

Let us first consider the unperturbed system (\ref {eq:unpert}) in the simple partial case $(N,n)=(2,1)$ and $T=2\pi/|\Omega_{11}-\fomega_1|$ (an analogue of the Hill problem from the Lunar theory). 
By assumption, the circle $S^1\times\{(I_{10},0,0)\}$ is a relative equilibrium of the system $(M_{10},\oomega_{10},\somega H_{10})$, and the motion on it by virture of this system is homogeneous with angular frequency $\somega\Omega_{10}=\fomega_1$ (cf.\ (\ref {maxrez'})). Hence, this circle consists of equilibria of the Hamiltonian system 
$(M_{10},\oomega_{10},\somega H_{10}-\fomega_1 I_{10})$ (which is an analogue of the ``synodic Kepler problem'' (\ref {eq:sin:Kepler:1}) for the ``scaled planet'').
Recall that the $S^1$-action is given by the Hamiltonian $I_{10}$ and $0$-Hamiltonian $I_{11}$. Let us study the $0$-Hamiltonian system
 \begin{equation} \label {eq:unpert:}
(M_{10}\times M_{11},M_{10},p;\oomega_{10},\oomega_{11};\somega H_{10}-\fomega_1 I_{10},H_{11}-\fomega_1 I_{11}+\somega^2 F_{11})^0,
 \end{equation}
which is obtained from a ``linear combination'' of two commuting $0$-Hamiltonian systems: the unperturbed one and the system defining the $S^1$-action.
Clearly, the problem on finding its $T$-periodic solutions of the form $(\varphi_{10},I_{10},0,0,\dots)$ is equivalent to finding $(T,\alpha)$-periodic solutions of the unperturbed system, where $\alpha=2\pi\fomega_1/|\Omega_{11}-\fomega_1|$.
Put 
$$
\widehat F(\varphi,I,q,p;I_{10}):=F_{11}(I_{10},0,0,\varphi,I,q,p)/\Omega_{10}^2(I_{10}),
$$
where $\Omega_{10}(I_{10}):=2\pi/T_{10}(H_{10}(I_{10},0,0))$, cf.\ (\ref {eq:Phi}). 
Since $\fomega_1=\somega\Omega_{10}$, we have $
\fomega_1^2 \widehat F = \somega^2 F_{11}(I_{10},0,0,\varphi,I,q,p)$.
Therefore, the problem of finding solutions of our system (\ref {eq:unpert:}) of the form $(\varphi_{10},I_{10},0,0,\dots)$ is equivalent to finding solutions of the Hamiltonian system 
$$
(M_{11},\oomega_{11},H_{11}-\fomega_1 I_{11}+\fomega_1^2 \widehat F)
$$
depending on two parameters
$\Omega_{10},\fomega_1$. This system is called a {\em generalized Hill problem} \cite{K:vest13}.

\begin{Lem} [(on periodic solutions of the generalized Hill problem, cf.\ {\rm\cite[Theorem 2]{K:vest13}})]  \label {lem:mult}
Consider a {\rm$1$}-parameter family of Hamiltonian systems with two degrees of freedom, with the Hamilton function and the symplectic structure 
 \begin{equation} \label {eq:H11}
H_{11}(I,q,p) - \fomega_1 I + \fomega_1^2 \widehat F(\varphi,I,q,p), \qquad \oomega_{11}=d I\wedge d\varphi+dp\wedge dq
 \end{equation}
on $M_{11}=S^1\times(a,b)\times\bbR^2$, with parameter $\fomega_1\in\bbR$. Suppose that some neighbourhood of the cylinder $S^1\times(a,b)\times\{0\}\times\{0\}\subset M_{11}$ is filled by periodic trajectories of the Hamiltonian system $(M_{11},\oomega_{11},H_{11})$ with periods $T_{11}\circ H_{11}$, for some function $T_{11}=T_{11}(h)\ne0$. Suppose that the function $H_{11}=H_{11}(I,q,p)$ satisfies the nondegeneracy conditions {\rm(\ref {eq:isoen})} at the points $(I,0,0)\in(a,b)\times\{0\}\times\{0\}$, where $\Omega_{11}(I):=2\pi/T_{11}(H_{11}(I,0,0))$ {\rm(so, the system (\ref {eq:H11}) is a generalization of the systems (\ref {eq:unpert:Hill}) and (\ref {eq:my:unpert:Hill}), which are equivalent to the Hill problem (\ref {eq:unpert:Hill'}) if $\fomega_1\ne0$, cf.\ \S \ref {subsec:Hill:0})}. 
Then:

{\rm(A)} there exist a continuous function $\somega_0=\somega_0(\Omega)>0$ in $\Omega\in\Omega_{11}((a,b))$ and a {\rm2}-parameter family of $\frac{2\pi}{|\Omega-\fomega_1|}$-periodic solutions $\gamma_{\Omega,\fomega_1}(t)=(\varphi_{\Omega,\fomega_1}(t),I_{\Omega,\fomega_1}(t),q_{\Omega,\fomega_1}(t),p_{\Omega,\fomega_1}(t))$ of the system {\rm(\ref {eq:H11})} with parameters $\Omega\in\Omega_{11}((a,b))$ and $\fomega_1\in(-\somega_0(\Omega),\somega_0(\Omega))$ such that $\varphi_{\Omega,\fomega_1}(0)=0$ and $\gamma_{\Omega_{11}(I),0}(t)=(\Omega_{11}(I)t,I,0,0)$ for $I\in(a,b)$.
If the functions $H_{11},\widehat F$ depend smoothly on the parameters $\lambda_1,\dots$, then the function $\somega_0>0$ depends smoothly on $\Omega,\lambda_1,\dots$, moreover the solution $\gamma_{\Omega,\fomega_1}(t)$ depends smoothly on the parameters $\fomega_1,\Omega,\lambda_1,\dots$;

{\rm(B)} if $d\langle \widehat F\o\rangle(0,0)=0$ for some $I\o\in(a,b)$, then 
$\gamma_{\Omega\o,\fomega_1}(t)=((\Omega\o-\fomega_1)t,I\o,0,0)+O(\fomega_1^2)$. Here $\Omega\o:=\Omega_{11}(I\o)$, and the function $\langle \widehat F\o\rangle$ is obtained by averaging the function 
$\widehat F\o:=\widehat F|_{H_{11}^{-1}(H_{11}(I\o,0,0))}$ along the 
$\frac{2\pi}{\Omega\o}$-periodic solutions of the system 
$(M_{11},\oomega_{11},H_{11})$;

{\rm(C)} there exists a family of canonicas frames $e_{\Omega,\fomega_1,1},e_{\Omega,\fomega_1,2},e_{\Omega,\fomega_1,3},e_{\Omega,\fomega_1,4}$ in the tangent spaces $T_{x}M_{11}$, in which the linearization $dg_{H_{11}-\fomega_1I+\fomega_1^2\widehat F}^{2\pi/|\Omega-\fomega_1|}(x)$ of the ``succession map'' at the point $x:=\gamma_{\Omega,\fomega_1}(0)$ is given by the matrix
 $$
\left(
 \begin{array} {cccc}
   1& \frac{\partial^2H_{11}}{\partial I^2}(\Omega_{11}^{-1}(\Omega),0,0) \frac{2\pi}{|\Omega-\fomega_1|} & 0&  0\\
   0& 1&                                                 0&               0\\
   0& 0&                                   \cos\alpha_{11}& \sin\alpha_{11}\\
   0& 0&                                  -\sin\alpha_{11}& \cos\alpha_{11}
 \end{array}
\right),
 $$
for some smooth function $\alpha_{11}=\alpha_{11}(\Omega,\fomega_1)$ whose Taylor expansion in the variable $\fomega_1$ at $0$ has the form 
$\alpha_{11}=\eta(\fomega_1+\fomega_1^2\Delta)\frac{2\pi}{|\Omega-\fomega_1|}+O(\fomega_1^3)$ for $\Omega\in\Omega_{11}((a,b))$, where the function $\Delta=\Delta(\Omega)$ and the sign $\eta\in\{1,-1\}$ are the same as in {\rm(\ref {eq:Delta:ij})} and {\rm(\ref {eq:eta:ij})}:
$$
\Delta(\Omega\o):=\frac{\Omega\o}2
\Tr\left(\left(\frac{\partial^2H_{11}(I\o,0,0)}{\partial(q,p)^2}\right)^{-1}
\frac{\partial^2\langle \widehat F\o\rangle(0,0)}{\partial(q,p)^2}\right)
 \quad\mbox{при}\ \ d\langle \widehat F\o\rangle(0,0)=0,
$$
$$
\eta:={\rm sgn\,}\left(\Omega\o\Tr\frac{\partial^2H_{11}(I\o,0,0)}{\partial(q,p)^2}\right).
$$
Moreover, the vectors $e_{\Omega,\fomega_1,k}$ are bounded \($k=1,2,3,4$\), moreover the relations $e_{\Omega,\fomega_1,1}=\partial/\partial\varphi$, $e_{\Omega,\fomega_1,2}=\partial/\partial I$, $e_{\Omega,\fomega_1,3}\in\bbR_{>0}\,\partial/\partial q$ hold for $\widehat F\equiv0$ and hold up to $O(\fomega_1)$ in the general case. \qed
\end{Lem}

From Lemma \ref {lem:mult}, we easily obtain the following its multidimensional generalization for any $N\ge n\ge1$ and for any relatively-periodic solutions (including long periodic ones, i.e.\ those with arbitrary large period $T$).


\begin{The} [(on Jordan-Kronecker blocks of linearization of the unperturbed succession map)] \label {th:mult}
Suppose that, under the hypothesis of theorem {\rm\ref {th1'}}, the
number $\somega\in(0,\somega_0]$ and the set of angular frequencies
$\Omega_{ij}=\Omega_{ij}(I_{ij}\o,0,0)$ satisfy the conditions {\rm 1), 2), 3)} of periodicity, nondegeneracy and ``relative resonance'' \(but not necessary the nondegeneracy condition {\rm 4)}\). Then there exists a $N$-dimensional torus $\Lambda$ that is $O(\somega)$-close \(and even
$O(\somega^2)$-close in the case of $d\langle
F_{ij}\o\rangle(0,0)=0$\) to the torus $\Lambda\o$ with respect to a $C^1$-norm 
and is formed by the phase orbits of $(T,\alpha)$-periodic solutions
of the unperturbed system {\rm(\ref {eq:unpert})}. Moreover, $p(\Lambda)=p(\Lambda\o)$ and, for any point $x\in\Lambda$, there exist canonical frames
$e_{ij1},e_{ij2},e_{ij3},e_{ij4}$ in the tangent spaces
$T_{x_{ij}}M_{ij}$ \(where $x_{ij}:=\pr_{ij}(x)$, $\pr_{ij}:M\to
M_{ij}$ is the projection\) such that the linear part $d(g_{\somega
H_0,H_1+\somega^2\Phi}^T g_{I_0,I_1}^{-\alpha})(x)$ of the unperturbed
``succession map'' at the point $x$ with respect to this frame is
given by a blockwise lower-triangular matrix with the diagonal
blocks
 $$
\left(
 \begin{array} {cccc}
   1& \omega^{\overline{1-j}}T \frac{\partial^2H_{ij}}{\partial I_{ij}^2}(I_{ij}\o,0,0)& 0&               0\\
   0& 1&                                                 0&               0\\
   0& 0&                                   \cos\alpha_{ij}& \sin\alpha_{ij}\\
   0& 0&                                  -\sin\alpha_{ij}& \cos\alpha_{ij}
 \end{array}
\right),
 \qquad \begin{array}{l}1\le i\le n, \smallskip\\ 0\le j\le n_i.\end{array}
 $$
Here $\alpha_{ij}$ are some real numbers such that
$\alpha_{i0}=\eta_{i0}\alpha$,
$|\alpha_{ij}-\eta_{ij}(\alpha+\Delta_{ij}\somega^2T)|\le
C_2\somega^3T$ for $1\le j\le n_i$, where the numbers $\Delta_{ij}$
and the signs $\eta_{ij}\in\{1,-1\}$ are the same as in {\rm(\ref
{eq:Delta:ij})} and {\rm(\ref {eq:eta:ij})}. Furthermore all
non-diagonal blocks vanish, apart from those blocks whose column and row correspond to the factors $M_{i0}$ and $M_{ij}$ \(respectively\), $1\le j\le n_i$, in the direct product
$M=\prod_{i=1}^n\prod_{j=0}^{n_i}M_{ij}$. The vectors $e_{ijk}$ are
bounded \($k=1,2,3,4$\), and the relations
$e_{ij1}=\partial/\partial\varphi_{ij}$, $e_{ij2}=\partial/\partial
I_{ij}$, $e_{ij3}\in\bbR_{>0}\,\partial/\partial q_{ij}$ hold either 
exactly if $j=0$ or up to $O(\somega)$ if $1\le j\le n_i$.

In particular, if the nondegeneracy condition {\rm4)} holds then $\alpha_{ij}\not\equiv0\pmod{2\pi}$ for any $i,j$; if the strong nondegeneracy condition holds then $\alpha_{ij}\not\equiv0\pmod{\pi}$.
\end{The}

Let us explain why the matrix in Theorem \ref {th:mult} is 
blockwise lower-triangular (instead of blockwise diagonal).
Unlike to the model system (\ref {eq:model}), which is a direct product, the unperturbed system (\ref {eq:unpert}) is only a semi-direct product. Namely: it ``projects'' to each factor $M_{i0}$ and to each factor $M_{i0}\times M_{ij}$, but (unlike to the model system) does not ``projects'' to the factors of the form $M_{ij}$, $1\le i\le n$, $1\le j\le n_i$. 
This observation has the following unexpected consequence: although each diagonal block of the indicated matrix is symplectic, the whole matrix is not symplectic. Indeed: the linear operator under consideration leaves invariant each subspace 
$V_i:=\bigoplus\limits_{j=0}^{n_i}T_{x_{ij}}M_{ij}$, $1\le i\le n$, as well as each its subspace $T_{x_{ij}}M_{ij}$, $1\le j\le n_i$, moreover all the subspaces $T_{x_{ij}}M_{ij}$, $0\le j\le n_i$, are symplectic and pairwise skew-orthogonal. This implies that, if the operator whould be symplectic, then it should leave invariant the subspace $T_{x_{i0}}M_{i0}$ too (since this subspace coincides with the skew-orthogonal complement in $V_i$ of the invariant subspace $\bigoplus\limits_{j=1}^{n_i}T_{x_{ij}}M_{ij}$). But the latter would mean that the matrix of the operator would be blockwise diagonal, which is false, since it is blockwise lower-triangular only, as explained above.

Notice that the diagonal blocks in Theorem \ref {th:mult} have the form $\exp(\omega^{\overline{1-j}}T V_{ij})$ where 
$$
V_{ij}=
 \left(
 \begin{array} {cccc}
   0& \frac{\partial^2H_{ij}}{\partial I_{ij}^2}(I_{ij}\o,0,0)& 0&              0\\
   0& 0&                                            0&              0\\
   0& 0&                                            0& \widehat\Omega_{ij}\\
   0& 0&                             -\widehat\Omega_{ij}&              0
 \end{array}
\right),
 \qquad \begin{array}{l}1\le i\le n, \smallskip\\ 0\le j\le n_i,\end{array}
$$
$\widehat\Omega_{i0}=\eta_{i0}\Omega_{i0}$,
$|\widehat\Omega_{ij}-\eta_{ij}(\Omega_{ij}+\Delta_{ij}\somega^2)|\le
C_2\somega^3$ и $\alpha_{ij}=\widehat\Omega_{ij}T\pmod{2\pi}$, $1\le
j\le n_i$.

%

{\it Stage two} of the proof of Theorems \ref {th1'} and \ref {th:stab} 
is based on generalizing the ``method of generating function'' (which was initially \cite {1,9,16,29,KN} introduced for Hamiltonian systems) to the case of an $\varepsilon$-Hamiltonian
(``perturbed'') system that is $C^2$-close to a 0-Hamiltonian
(``unperturbed'') system. In this method, one studies $T$-periodic trajectories of the
perturbed system in a neighbourhood of a ``nondegenerate'' compact
submanifold $\Lambda$ formed by the phase trajectories of
$(T,\alpha)$-periodic solutions of the unperturbed system. Namely, 
one proves that the intersections points of $(T,\alpha)$-periodic trajectories of
the perturbed system with the ``cross section'' $\Sigma$ (see
the formulation of theorem \ref {th1'}) coincide with critical points
of the function $\Psi|_{\Sigma\cap\widetilde\Lambda}$. 
Since this function is defined on a $N-2$-dimensional torus $\Sigma\cap\widetilde\Lambda$, the number of its critical points is at least $N-1$, morever it is at least $2^{N-2}$ when counting with multiplicities \cite {35}. This gives us the required estimate (formulated in Theorem \ref {th1'}) for the number of $(T,\alpha)$-periodic trajectories of the perturbed system.
Here $\Psi$ is the generating function (cf.\ Definition \ref {def:gen:func} below) of the perturbed succession map $\widetilde A=g_{\widetilde H,\widetilde H_1+\somega^2\widetilde \Phi}^{T} g_{\widetilde I,I_1}^{-\alpha}:M\to M$, moreover $\widetilde\Lambda\subset M$ is a submanifold that is $C^1$-close to the submanifold $\Lambda$.
A proof of Theorem \ref {th:stab} on orbital stability in linear approximation is based on some properties (and their ``preservation'' under perturbations) of the linarization of the unperturbed ``succession map'' $A$ at points of the torus $\Lambda$ (cf.\ \cite[\S3.3.2]{kudr:diss} for details).

\begin{scDef} [(generating function)] \label {def:gen:func}
Let $\varepsilon>0$ and $A:M\to M$ be a symplectic self-map of a
symplectic manifold $(M,\oomega)=(M_0\times
M_1,\oomega_0+\varepsilon\oomega_1)$. Denote $v_{ij1}:=\varphi_{ij}$,
$v_{ij2}:=q_{ij}$ (``coordinates''), $u_{ij1}:=I_{ij}$,
$u_{ij2}:=p_{ij}$ (``momenta''). Denote by $\alpha$ the
differential 1-form of the type of $(\u'-\u)d\v+(\v-\v')d\u'$ on $M$.
More precisely, we define $\alpha$ by the formula
 $$
 \left.
\alpha(\pnt):=
 \sum^n_{i=1} \sum_{k=1}^2
 \right( \left( u_{i0k}'-u_{i0k}\right)dv_{i0k}+\left(v_{i0k}-v_{i0k}'\right)du_{i0k}'+
 $$
 $$
 \left. + \varepsilon\sum^{n_i}_{j=1}
       \left( \left(u_{ijk}'-u_{ijk}\right)dv_{ijk}+\left(v_{ijk}-v_{ijk}'\right)du_{ijk}' \right) \right)
 $$
where $A:\pnt=(\v,\u)\mapsto A(\pnt)=:(\v',\u')$. In other words,
 \begin{equation} \label {Psi**}
 \alpha(\pnt)\xi:=\oomega(A(\pnt)-\pnt,dP(\pnt)\xi), \qquad \xi\in T_\pnt M,
 \end{equation}
where $P:\pnt=(\v,\u)\mapsto P(\pnt):=(\v,\u')$, and $\oomega(\xi,\eta)$
denotes the value of the symplectic structure
$\oomega=\sum_{i=1}^n\sum_{k=1}^2(du_{i0k}\wedge
dv_{i0k}+\varepsilon\sum_{j=1}^{n_i}du_{ijk}\wedge dv_{ijk})$ on the
pair of vectors $\xi,\eta\in T_\pnt M$. A function $\Psi=\Psi(\pnt)$
will be called a {\it generating function} of the map $A$ if 
 \begin{equation} \label {Psi*}
 d\Psi(\pnt)=\alpha(\pnt), \qquad \pnt\in M.
 \end{equation}
Let us show that such a function $\Psi$ exists, i.e.\ the form
$\alpha$ is exact. One easily shows that the integral of the form
$\alpha$ along any closed curve equals the integral of the symplectic
structure $\oomega$ along some two-dimensional torus. The latter
integral vanishes, since the symplectic structure under consideration
is exact (being the standard symplectic structure on $M=T^*Q$). Thus,
the integral of the form $\alpha$ along any closed curve vanishes.
This proves that the function $\Psi$ is well-defined up to an
additive constant.
\end{scDef}

Let us derive, via the method of a generating function, Theorem \ref {th1'} from the technical Theorem \ref {th:mult}. For this, we will firts (in Steps 1--3) construct a torus $\widetilde\Lambda\subset M$ and show that the generating function $\Psi$ of the ``perturbed succession map'' 
 \begin{equation} \label {eq:tildeA}
\widetilde A := g_{\widetilde H,\widetilde H_1+\somega^2\widetilde \Phi}^{T} g_{\widetilde I,I_1}^{-\alpha} : M\to M, \qquad \pnt=(\vvarphi,\I,\q,\p)\mapsto
(\widetilde\vvarphi,\widetilde\I,\widetilde\q,\widetilde\p),
 \end{equation}
possesses the property
 \begin{equation} \label {eq:dPsi}
d\Psi(\pnt)= \sum^n_{i=1} 
 \left( \left( \widetilde I_{i0}-I_{i0}\right) d\varphi_{i0} + \varepsilon\sum^{n_i}_{j=1}
        \left( \widetilde I_{ij}-I_{ij}\right) d\varphi_{ij} 
 \right), \qquad \pnt\in\widetilde\Lambda.
 \end{equation}
After that (in Steps 4 and 5), we will justify the main idea of the method, namely: we will show that, for small enough perturbation and $\varepsilon\ne0$, all critical points of the function $\Psi|_{\Sigma\cap\widetilde\Lambda}$ are fixed under the map $\widetilde A$ (i.e., they are intersection points of $(T,\alpha)$-periodic trajectories of the perturbed system with the ``cross section'' $\Sigma$). Here, the standard transversality argument does not work (because the factor $\varepsilon$ in (\ref {eq:dPsi}) is small, cf.\ Remark \ref {rem:transvers} below), however we manage to adapt it.

{\it Step 1.} Let us derive from technical Theorem \ref {th:mult} that the torus $\Lambda$ is nondegenerate in the following sense (\ref {eq:nondegL}). Let us write the ``unperturbed succession map'' $A:=g_{H_0,H_1+\somega^2\Phi}^{T} g_{I_0,I_1}^{-\alpha}:M\to M$ in the following form in the coordinates under consideration:
 $$
A = g_{H_0,H_1+\somega^2\Phi}^{T} g_{I_0,I_1}^{-\alpha} : M\to M, \qquad \pnt=(\vvarphi,\I,\q,\p)\mapsto(\vvarphi',\I',\q',\p'),
 $$ 
where $(\vvarphi',\I',\q',\p')=(\vvarphi'(\vvarphi,\I,\q,\p),\I'(\vvarphi,\I,\q,\p),\q'(\vvarphi,\I,\q,\p),\p'(\vvarphi,\I,\q,\p))$. By nondegeneracy of the torus $\Lambda$, we mean that, at every point $\pnt=(\vvarphi,\I,\q,\p)\in\Lambda$, the following determinant does not vanish:
 \begin{equation} \label{eq:nondegL}
\det\frac{\partial(\vvarphi'-\vvarphi,\q'-\q,\p'-\p)}{\partial(\I,\q,\p)} \ne0, \qquad (\vvarphi,\I,\q,\p)\in\Lambda,
 \end{equation}
i.e., the Jacobi matrix of the map $(\I,\q,\p)\mapsto(\vvarphi'-\vvarphi,\q'-\q,\p'-\p)$ at the point $(\I,\q,\p)$ is nondegenerate.

The decomposition $M=\prod\limits_{i=1}^n\prod\limits_{j=0}^{n_i}M_{ij}$ gives us an isomorphism $T_\pnt M=\bigoplus\limits_{i=1}^n\bigoplus\limits_{j=0}^{n_i}T_{\pnt_{ij}}M_{ij}$, where $p_{ij}:M\to M_{ij}$ is the projection, $\pnt_{ij}:=p_{ij}(\pnt)$. Recall that the vectors $\frac{\partial}{\partial\varphi_{ij}},\frac{\partial}{\partial I_{ij}}$, $\frac{\partial}{\partial q_{ij}},\frac{\partial}{\partial p_{ij}}$ form a basis of the tangent space $T_{\pnt_{ij}}M_{ij}$. Since the unperturbed system ``projects'' to each factor $M_{i0}$ and to each factor $M_{i0}\times M_{ij}$, we conclude that the matrix in (\ref {eq:nondegL}) is blockwise lower-triangular with diagonal blocks $\frac{\partial(\varphi_{ij}'-\varphi_{ij},q'_{ij}-q_{ij},p'_{ij}-p_{ij})}{\partial(I_{ij},q_{ij},p_{ij})}$. Therefore, the nondegeneracy condition (\ref {eq:nondegL}) is equivalent to nondegeneracy of each of these blocks.

Let $j=0$. By thechnical Theorem \ref {th:mult}, the block $\frac{\partial(\varphi_{i0}',q'_{i0},p'_{i0})}{\partial(I_{i0},q_{i0},p_{i0})}$ is conjugated to the corresponding minor of the corresponding block from the same theorem (since if $j=0$ then the bases $e_{i01},e_{i02},e_{i03},e_{i04}$ and $\frac{\partial}{\partial\varphi_{i0}},\frac{\partial}{\partial I_{i0}},\frac{\partial}{\partial q_{i0}},\frac{\partial}{\partial p_{i0}}$ differ from each other just by
changing the last two vectors by their linear combinations).
From the explicit form of this block, we obtain that $\det\frac{\partial(\varphi_{i0}'-\varphi_{i0},q'_{i0}-q_{i0},p'_{i0}-p_{i0})}{\partial(I_{i0},q_{i0},p_{i0})}\ne0$ (since $\alpha\not\in2\pi\bbZ$ due to the nondegeneracy condition (\ref {eq:nondeg}) or to the more delicate condition (\ref {eq:nondeg'})). 

Let now $1\le j\le n_i$. Suppose that $\det\frac{\partial(\varphi_{ij}'-\varphi_{ij},q'_{ij}-q_{ij},p'_{ij}-p_{ij})}{\partial(I_{ij},q_{ij},p_{ij})}=0$. Then (because the subspace $T_{\pnt_{ij}}M_{ij}\subset T_\pnt M$ is invariant under the operator $dA(\pnt)$ for $j>0$) there exists a nontrivial linear combination $\xi_{ij}\in T_{\pnt_{ij}}M_{ij}$ of vectors $\frac{\partial}{\partial I_{ij}},\frac{\partial}{\partial q_{ij}},\frac{\partial}{\partial p_{ij}}$ such that the vector $dA(\pnt)\xi_{ij}-\xi_{ij}$ is proportional to the vector $\frac{\partial}{\partial I_{ij}}$. On the other hand, it follows from the technical Theorem \ref {th:mult} that any vector of the form $dA(\pnt)\eta_{ij}-\eta_{ij}$ (with $j>0$ and $\eta_{ij}\in T_{\pnt_{ij}}M_{ij}$) is a linear combination of vectors $e_{ij1},e_{ij3},e_{ij4}$. But (for small enough $|\fomega|$) such a linear combination can not be proportional to the vector $\frac{\partial}{\partial I_{ij}}$, since, due to the same theorem, the vectors $e_{ij1},\frac{\partial}{\partial I_{ij}},e_{ij3},e_{ij4}$ are linearly independent (and form a basis of the space $T_{\pnt_{ij}}M_{ij}$). Hence $dA(\pnt)\xi_{ij}-\xi_{ij}=0$, i.e.\ the vector $\xi_{ij}$ is fixed under the map $dA(\pnt)$. Now, it follows from the same theorem that any fixed vector of the operator $dA(\pnt)|_{T_{\pnt_{ij}}M_{ij}}$ is proportional to the vector $e_{ij1}$ (since $\alpha_{ij}\not\in2\pi\bbZ$ due to the nondegeneracy condition (\ref {eq:nondeg}) or (\ref {eq:nondeg'})). 
But (for small enough $|\fomega|$) such a vector can not be a linear combination of the vectors $\frac{\partial}{\partial I_{ij}},\frac{\partial}{\partial q_{ij}},\frac{\partial}{\partial p_{ij}}$, since, by the same theorem, the vectors $e_{ij1},\frac{\partial}{\partial I_{ij}},\frac{\partial}{\partial q_{ij}},\frac{\partial}{\partial p_{ij}}$ are linearly independent (and form a basis of the subspace $T_{\pnt_{ij}}M_{ij}$). Therefore $\xi_{ij}=0$, a contradiction. This proves the inequality $\det\frac{\partial(\varphi_{ij}'-\varphi_{ij},q'_{ij}-q_{ij},p'_{ij}-p_{ij})}{\partial(I_{ij},q_{ij},p_{ij})}\ne0$, and hence (\ref {eq:nondegL}).

{\it Step 2.} Now let us describe a construction of the torus $\widetilde\Lambda\subset M$. Let the ``perturbed succession map'' $\widetilde A:M\to M$ has the form (\ref {eq:tildeA}) in the coordinates under consideration, where $(\widetilde\vvarphi,\widetilde\I,\widetilde\q,\widetilde\p)=(\widetilde\vvarphi(\vvarphi,\I,\q,\p),\widetilde\I(\vvarphi,\I,\q,\p),\widetilde\q(\vvarphi,\I,\q,\p),\widetilde\p(\vvarphi,\I,\q,\p))$. Define a set $\widetilde\Lambda\subset M$ as the following set of points near $\Lambda$:
 \begin{equation} \label {eq:tildeL}
 \widetilde\Lambda := \{(\vvarphi,\I,\q,\p) \mid (\widetilde\vvarphi(\vvarphi,\I,\q,\p),\widetilde\q(\vvarphi,\I,\q,\p),\widetilde\p(\vvarphi,\I,\q,\p)) = (\vvarphi,\q,\p) \}.
 \end{equation}

It follows from the Implicit Functions Theorem and from the nondegeneracy condition (\ref {eq:nondegL}) that, for a small enough size of perturbation $\delta\ge|\varepsilon|\ge0$ (here $\delta:=|\varepsilon|+|\mu|+|\nu|+|\rho|$ in the case of the $N+1$ body problem), the subset $\widetilde\Lambda\subset M$ has the form of a graph
 \begin{equation} \label {eq:tildeL:graph}
\widetilde\Lambda=\{(\vvarphi,\I(\vvarphi),\q(\vvarphi),\p(\vvarphi))\mid\vvarphi\in(\bbR/2\pi\bbZ)^N\}
 \end{equation}
for some smooth functions $\I(\vvarphi),\q(\vvarphi),\p(\vvarphi)$, i.e., it is a $N$-dimensional torus that depends smoothly on $\varepsilon$ and other small parameters of the system (i.e., on $\varepsilon,\mu,\nu,\rho$ in the case of the $N+1$ body problem with fixed $m_i,m_{ij}$), as well as on the parameters $\somega,T>0$. The torus $\widetilde\Lambda$ is $O(\delta)$-close to the torus $\Lambda$, due to the Implicit Functions Theorem and $O(\delta)$-closeness of the ``perturbed'' map $\widetilde A=g_{\widetilde H,\widetilde H_1+\somega^2\widetilde \Phi}^{T} g_{\widetilde I,I_1}^{-\alpha}$ to the unperturbed one $A=g_{H_0,H_1+\somega^2\Phi}^{T} g_{I_0,I_1}^{-\alpha}$. From now on, we mean by $O(1)$ real numbers whose absolute values do not exceed a positive value depending continuously on $\somega,T>0$ and the integers $k_{ij}\in\bbZ$.

{\it Шаг 3.} Let $\Psi$ be a generating function of the map $\widetilde A$, cf.\ (\ref {eq:tildeA}), (\ref {eq:gen:funct}), (\ref {Psi*}). It follows from the construction of the torus $\widetilde\Lambda$ (cf.\ Step 2) and the Implicite Functions Theorem that the map $\widetilde A$ shifts each point $\pnt\in\widetilde\Lambda$ by a vector having the form
 \begin{equation} \label {eq:deltaA:small}
\widetilde A(\pnt)-\pnt =(0,\widetilde\I-\I,0,0) =O(\delta), \qquad \pnt=(\vvarphi,\I,\q,\p)\in\widetilde\Lambda.
 \end{equation}
Here we wrote that the size of the shift has the same order $O(\delta)$ as the order of perturbation, since if $\delta=0$ then $\widetilde\Lambda=\Lambda$ and $A(\pnt)=\pnt$ for any point $\pnt\in\Lambda$. Therefore $d\Psi(\pnt)$ has the form (\ref {eq:dPsi}) at every point $\pnt\in\widetilde\Lambda$.

Therefore, if $\varepsilon\ne0$ then a point $\pnt\in\widetilde\Lambda$ is a critical point of the generating function $\Psi$ if and only if it is fixed under the map $\widetilde A$ (i.e., this point is contained in the phase orbit of a $(T,\alpha)$-periodic solution of the perturbed system). 

It remains to show that, if $\varepsilon\ne0$ (and the perturbation is small), then any critical point of the function $\Psi|_{\Sigma\cap\widetilde\Lambda}$ is also fixed under the map $\widetilde A$.

{\it Step 4.} Let $\varepsilon\ne0$ and a point $\pnt\in\Sigma\cap\widetilde\Lambda$ is  a critical point of the function $\Psi|_{\Sigma\cap\widetilde\Lambda}$. This means that the covector (\ref {eq:dPsi}) is a linear combination of the covectors $\sum_{i=1}^n\sum_{j=0}^{n_i}d\varphi_{ij}$ and $\sum_{i=1}^n\sum_{j=0}^{n_i}\somega^{\overline{1-j}}\Omega_{ij}d\varphi_{ij} = \sum_{i=1}^n(\somega\Omega_{i0}d\varphi_{i0}+\sum_{j=1}^{n_i}\Omega_{ij}d\varphi_{ij})$ (since the plane $\Sigma$ is a common level set of the functions $\sum_{i=1}^n\sum_{j=0}^{n_i}\varphi_{ij}$ and $\sum_{i=1}^n\sum_{j=0}^{n_i}k_{ij}\varphi_{ij}$). By denoting the coefficients of this linear combination by $\lambda_1,\lambda_2\in\bbR$, we obtain $\widetilde I_{i0}-I_{i0}=\lambda_1+\lambda_2\somega\Omega_{i0}$, 
$\widetilde I_{ij}-I_{ij}=\varepsilon^{-1}(\lambda_1+\lambda_2\Omega_{ij})$ for $1\le i\le n$, $1\le j\le n_i$.

Thus, the map $\widetilde A$ shifts our point $\pnt$ by the vector
 \begin{equation} \label {eq:deltaA}
\widetilde A(\pnt)-\pnt =(0,\widetilde\I-\I,0,0)=\lambda_1\V_1+\lambda_2\V_2.
 \end{equation}
Here we denoted 
$$
\V_1:=\sum_{i=1}^n\left(\frac{\partial}{\partial I_{i0}}+\varepsilon^{-1}\sum_{j=1}^{n_i}\frac{\partial}{\partial I_{ij}}\right), \quad 
\V_2:=\sum_{i=1}^n\left(\somega\Omega_{i0}\frac{\partial}{\partial I_{i0}}+\varepsilon^{-1}\sum_{j=1}^{n_i}\Omega_{ij}\frac{\partial}{\partial I_{ij}}\right).
$$ 
We have to show that $\lambda_1=\lambda_2=0$ in (\ref {eq:deltaA}).

{\it Step 5.} Since the Hamilton functions  
$\widetilde H=\somega\widetilde H_0+\varepsilon(\widetilde H_1+\somega^2\widetilde\Phi)$ and $\widetilde I=I_0+\varepsilon I_1$ of our perturbed system and of the $S^1$-action are preserved by the map $\widetilde A$, we have $\widetilde I(\widetilde A(\pnt))-\widetilde I(\pnt)=0$ and $\widetilde H(\widetilde A(\pnt))-\widetilde H(\pnt)=0$. 
From this conservation laws, taking into account the inclusion $\pnt\in\widetilde\Lambda$, the relation (\ref {eq:deltaA:small}) and $O(\delta)$-closeness of the torus $\widetilde\Lambda$ to $\Lambda$ (Step 2), we conclude that $\sum_{i=1}^n(\widetilde I_{i0}-I_{i0})+\varepsilon\sum_{i=1}^n\sum_{j=1}^{n_i}(\widetilde I_{ij}-I_{ij})=0$ and $\sum_{i=1}^n(\somega\Omega_{i0}+O(\delta))(\widetilde I_{i0}-I_{i0})+\varepsilon\sum_{i=1}^n\sum_{j=1}^{n_i}(\Omega_{ij}+O(\delta))(\widetilde I_{ij}-I_{ij})=0$. Due to (\ref {eq:deltaA}), we conclude that the pair of real numbers $\lambda_1,\lambda_2\in\bbR$ satisfies the following system of two linear relations that do not involve $\varepsilon$:
 $$
\lambda_1N+\lambda_2\sum_{i,j}\somega^{\overline{1-j}}\Omega_{ij}=0,\ 
\lambda_1\sum_{i,j}\somega^{\overline{1-j}}\Omega_{ij} +\lambda_2\sum_{i,j}(\somega^{\overline{1-j}}\Omega_{ij})^2=O((|\lambda_1|+|\lambda_2|)\delta),
 $$
where $\sum_{i,j}:=\sum_{i=1}^n\sum_{j=0}^{n_i}$. We can rewrite this system in the form
 $$
\lambda_1 \langle \V\o_1,\V\o_1 \rangle +\lambda_2 \langle \V\o_1,\V\o_2 \rangle=0, \qquad 
\lambda_1 \langle \V\o_1,\V\o_2 \rangle +\lambda_2 \langle \V\o_2,\V\o_2 \rangle=:B=O((|\lambda_1|+|\lambda_2|)\delta).
 $$
Here $\langle \cdot , \cdot \rangle$ denotes the standard scalar product in $\bbR^N$, and the vectors
$$
\V\o_1:=\sum_{i=1}^n\left(\frac{\partial}{\partial\varphi_{i0}}+\sum_{j=1}^{n_i}\frac{\partial}{\partial\varphi_{ij}}\right), \quad 
\V\o_2:=\sum_{i=1}^n\left(\somega\Omega_{i0}\frac{\partial}{\partial\varphi_{i0}}+\sum_{j=1}^{n_i}\Omega_{ij}\frac{\partial}{\partial\varphi_{ij}}\right)
$$
do not depend on $\varepsilon$. Since the vectors $\V\o_1,\V\o_2$ are noncollinear (due to (\ref {maxrez'}) and $N\ge2$), the matrix of the latter linear system is nondegenerate, and its solution is unique:
$$
\lambda_1 = \frac{ - \langle \V\o_1,\V\o_2 \rangle B }
 {\langle \V\o_1,\V\o_1 \rangle\langle \V\o_2,\V\o_2 \rangle - \langle \V\o_1,\V\o_2 \rangle^2}, \qquad 
\lambda_2 = \frac{\langle \V\o_1,\V\o_1 \rangle B }
 {\langle \V\o_1,\V\o_1 \rangle\langle \V\o_2,\V\o_2 \rangle - \langle \V\o_1,\V\o_2 \rangle^2}.
$$
Therefore $|\lambda_1|+|\lambda_2|=O(B)=O((|\lambda_1|+|\lambda_2|)\delta)$ when $\delta\to0$. Hence, if the perturbation is small enough (i.e.\ $\delta\to0$) then $\lambda_1=\lambda_2=0$ in (\ref {eq:deltaA}), i.e.\ $\widetilde A(\pnt)=\pnt$ as required. Theorem \ref {th1'} is completely proved. \qed

\begin{scRem} \label {rem:transvers}
Let us explain the following: (a) a geometric meaning of the vectors $\V_1\o,\V_2\o$ and the equality (\ref {eq:deltaA}); (b) how one could use these vectors for proving in a geometric manner (by means of a ``transversality argument'')
the key idea of the generating function method (cf.\ Steps 4--5 above), if a small factor $\varepsilon$ would not enter the ``fast'' components of the co-vector (\ref {eq:dPsi}); (c) why is the standard ``transversality argument'' not working for systems with ``slow'' and ``fast'' variables.

(a) The vectors $\V_1\o,\V_2\o$ are nothing else than the values at the point $\pnt\o\in\Lambda\o$ of the 0-Hamiltonian vector fields corresponding to our $S^1$-action and model system. These vectors are non-collinear, since, by assumption of Theorem \ref {th1'}, one has $N\ge2$ and the integers $k_{ij}\in\bbZ$ are not simultaneously $0$ (due to (\ref {maxrez'})). By construction, the cross-section $\Sigma$ is the orthogonal complement to the plane $\Span(\V_1\o,\V_2\o)$ spanned by the vectors $\V_1\o,\V_2\o$. 
Therefore, the fact that a point $\pnt\in\Sigma\cap\widetilde\Lambda$ is a critical point of the function $\Psi|_{\Sigma\cap\widetilde\Lambda}$ means that the gradient of the function $\Psi$ at the point $\pnt$ belongs to the plane $\Span(\V_1\o,\V_2\o)$. But, due to (\ref {eq:dPsi}), this gradient equals $\sum_{i=1}^n(\widetilde I_{i0}-I_{i0})\frac{\partial}{\partial\varphi_{i0}}+\varepsilon\sum_{i=1}^n\sum_{j=1}^{n_i}(\widetilde I_{ij}-I_{ij})\frac{\partial}{\partial\varphi_{ij}}$, i.e., it has the same coordinates as the vector of shifting (\ref {eq:deltaA:small}), up to a factor $\varepsilon$ and a transposition of the basis vectors $\frac{\partial}{\partial\varphi_{ij}}$ and $\frac{\partial}{\partial I_{ij}}$. Hence, the shift vector belongs to the plane $\Span(\V_1,\V_2)$, i.e.\ (\ref {eq:deltaA}) holds. On the other hand, due to the conservation laws (cf.\ Step 5 above), the endpoints of the shift vector (i.e., the points $\pnt$ and $\widetilde A(\pnt)$) belong to the same common level set $\widetilde H^{-1}(a)\cap\widetilde I^{-1}(b)$ of the functions $\widetilde H,\widetilde I$. We want to manage to conclude that such a shift has to be trivial, i.e.\ $\widetilde A(\pnt)=\pnt$.

(b) The standard transversality argument is as follows. Suppose that, firstly, $N>n$ (i.e., the ``fast'' variables and a small factor $\varepsilon$ are present), secondly, the gradients of the ``unperturbed'' functions $H_0,I_0$ at the point $\pnt\o$ are non-collinear, thirdly, the ``unperturbed'' vectors $\lim_{\varepsilon\to0}(\varepsilon\V_1),\lim_{\varepsilon\to0}(\varepsilon\V_2)$ are also non-collinear, and fourthly, the ``unperturbed'' plane $\Span(\lim_{\varepsilon\to0}(\varepsilon\V_1),\lim_{\varepsilon\to0}(\varepsilon\V_2))$ and the common level set $H_0^{-1}(a\o)\cap I_0^{-1}(b\o)$ are transversal (in fact, the first three assumptions can be fulfilled simultaneously under the hypotheses of Theorem \ref {th1'}, but the first and fourth assumptions can not be fulfilled simultaneously, see below). Then, for small enough perturbation, the ``perturbed'' plane $\Span(\V_1,\V_2)$ and the common level set $\widetilde H^{-1}(a)\cap\widetilde I^{-1}(b)$ are also transversal. Hence, due to (a), the shift vector is zero, i.e.\ $\widetilde A(\pnt)=\pnt$ as required. However our first and fourth assumptions can not be fulfilled simultaneously, since the ``unperturbed'' common level set contains the indicated ``unperturbed'' plane, and hence they can not be transversal. 

(c) Let us summarize. The standard transversality argument can not be applied in our case (when ``fast'' variables are present). The reason is that the small factor $\varepsilon$ enters 
the co-vectors $d\widetilde H,d\widetilde I$ and the vectors $\varepsilon\V_1,\varepsilon\V_2$ in an inconsistent manner. Namely, it enters the ``fast'' components of the co-vectors $d\widetilde H,d\widetilde I$, but the ``slow'' components of the vectors $\varepsilon\V_1,\varepsilon\V_2$.
\end{scRem}

{\it Stage two} of the proof of Theorem \ref {th:sym} (on symmetric $(T,\alpha)$-periodic solutions) can be performed either by the method of Tkhay \cite {Thay} (based on the Implicit Functions Theorem and a simple refinement of Lemma \ref {lem:mult} and Theorem \ref {th:mult} for reversible systems, without using a generating function), or by the following method by Krassinsky \cite {11} (based on Theorem \ref {th1'} and the explicit construction (\ref {eq:tildeL}) of the torus $\widetilde\Lambda$).

Let us prove first that the function $\Psi|_{\widetilde\Lambda}$ is an even function in the variables $\vvarphi|_{\widetilde\Lambda}$. Since the involution $\tau$ is anti-canonical, it maps phase trajectories to phase trajectories with reversing the time. In particular, it maps the segment of the phase trajectory from any point $\pnt$ to the point $\widetilde A(\pnt)$ into the segment of a phase trajectory from the point $\widetilde\pnt:=\tau(\widetilde A(\pnt))$ to the point $\tau(\pnt)$ with reversing the time, therefore $\tau(\pnt)=\widetilde A(\widetilde\pnt)$.
On the other hand, due to (\ref {eq:tildeL}), any point $\pnt=(\vvarphi,\I,\q,\p)\in\widetilde\Lambda$ is sent by the map $\widetilde A$ to a point of the form $\widetilde A(\pnt)=(\vvarphi,\widetilde\I,\q,\p)$. Hence, we conclude from Definition of the involution $\tau$ that $\widetilde\pnt=\tau(\widetilde A(\pnt))=(-\vvarphi,\widetilde\I,\q,-\p)$ and $\widetilde A(\widetilde\pnt)=\tau(\pnt)=(-\vvarphi,\I,\q,-\p)$, therefore $\widetilde\pnt\in\widetilde\Lambda$ (см.\ (\ref {eq:tildeL})). Now, we obtain from (\ref {eq:dPsi}) the equality of co-vectors $d(\Psi|_{\widetilde\Lambda})(\pnt)=-d(\Psi|_{\widetilde\Lambda})(\widetilde\pnt)$ with respect to the variables $\vvarphi|_{\widetilde\Lambda}$ on the torus $\widetilde\Lambda$.
Since (in these variables) the points $\pnt$ and $\widetilde\pnt$ have coordinates $\vvarphi$ and $-\vvarphi$, the differential of the function $\Psi|_{\widetilde\Lambda}=\Psi|_{\widetilde\Lambda}(\vvarphi)$ has the form $d(\Psi|_{\widetilde\Lambda})(\vvarphi)=-d(\Psi|_{\widetilde\Lambda})(-\vvarphi)=d(\Psi|_{\widetilde\Lambda}(-\vvarphi))$. Hence, the function $\Psi|_{\widetilde\Lambda}(\vvarphi)-\Psi|_{\widetilde\Lambda}(-\vvarphi)$ is constant on the torus $\widetilde\Lambda$. Since this function equals $0$ at the point $\vvarphi=0$, it equals $0$ everywhere, i.e.\ the function $\Psi|_{\widetilde\Lambda}(\vvarphi)$ is even.

Since the function $\Psi|_{\widetilde\Lambda}(\vvarphi)$ is even, we conclude that $d(\Psi|_{\widetilde\Lambda})(0)=0$. Hence, by Theorem \ref {th1'}, the solution of the ``perturbed'' system with initial condition of the form $\pnt^0=(0,\I^0,\q^0,\p^0)\in\widetilde\Lambda$ is $(T,\alpha)$-periodic, i.e.\ $\widetilde A(\pnt^0)=\pnt^0$. Let us show that this solution is symmetric. It sufficies to show that $\tau(\pnt^0)=\pnt^0$ (см.\ \S \ref {par:sym}). As we proved above, $\tau(\pnt^0)=\tau(\widetilde A(\pnt^0))=(0,\widetilde\I^0,\q^0,-\p^0)\in\widetilde\Lambda$. Thus, both points $\tau(\pnt^0)$ and $\pnt^0$ belong to $\widetilde\Lambda$ and have the same coordinaet $\vvarphi=0$. Hence these points coincide (since $\widetilde\Lambda$ has the form (\ref {eq:tildeL:graph})), which proves that the solution under consideration is symmetric. In a similar way, one proves that any solution of the ``perturbed'' system with initial condition of the form $\pnt^\beta=(\beta,\I^\beta,\q^\beta,\p^\beta)\in\widetilde\Lambda$ is also relatively-periodic and symmetric, where $\beta\in\{0,\pi\}^N$.
Theorem \ref {th:sym} is completely proved.

{\it Stage two} of the proof of Theorem \ref {th:nondeg} (on ``gaps'' in the families of relatively-periodic solutions) is based on the averaging method on a submanifold \cite[theorem 11.1]{9}, \cite[theorems 2, 4]{29}. 
We completely prove (by this method) Theorems \ref {th:nondeg} and \ref {th:degen:sym} in the important special case: for planetary systems without satellites (cf.\ \Sec\ref {subsec:3body:}, proof of Corollary \ref {cor:3bodies}($\nexists$), and \Sec\ref {subsec:any}).

\section {Relatively-periodic solutions, symmetricity, stability and nondegeneracy} \label {par3:1}

Consider the problem about the motion of a system of $N+1$ particles
attracting by Newton's law on a Euclidean plane $E^2$, $N\ge2$. 
The equations of the motion have the form (\ref {O}),
(\ref {pot}).

\subsection {Relatively-periodic solutions} \label {par3:1:1}

Let us state the problem on finding $(T,\alpha)$-periodic solutions of the system, cf.\ Definition \ref {def:per}.

The motions in a rotating coordinate system with angular velocity $\fomega_1$ are described by the Hamiltonian system with the Hamilton fumction $H-\fomega_1I$, where
 \begin{equation} \label {M}
I:=\sum^N_{i=0}[\M_i,\p_i]
 \end{equation}
is the ``area integral'', also called the {\it kinetic moment}~\cite {12}. Here $[\M_i,\p_i]$ denotes the oriented area of the parallelogram spanned by the vectors $\M_i$ and $\p_i$. Therefore our problem is equivalent to finding $T$-periodic solutions of the Hamiltonian system with the Hamilton function $H-\fomega_1I$, where $\fomega_1=\frac{\alpha-2\pi
s}{T}$, $s$ is any integer. Due to (\ref {maxrez}) and (\ref {parametry}), we can define $\fomega_1$ as the angular frequency of any planet or satellite, e.g.\ the angular frequency of the first planet. 

The invariant $4N$-dimensional submanifold $M^{4N}\approx T^*\widehat Q$ (cf.\ (\ref {eq:4N})) of the $N+1$ body problem is invariant under the phase flow of the system with the Hamilton function (\ref {M}), and a relatively-periodic solution on $M^{4N}$ is defined as above. 

For the model system (\ref {eq:model:H:omega}), (\ref {eq:model:H}), (\ref {eq:model:omega}), a relatively-periodic solution is defined via the Hamiltonian action of the circle with the Hamiltonian $I=\sum_{i=1}([\mx_i,\mxxi_i]+\sum_{j=1}^{n_i}[\my_{ij},\meeta_{ij}])$ on the submanifold $(\bbR^2)^N\times(\bbR^2)^N$ with coordinates $(\mx,\mxxi)$ and the symplecitc structure $d\mxxi\wedge d\mx$. For the unperturbed and perturbed systems (\ref {eq:unpert0}), (\ref {eq:unpert1}) and (\ref {eq:pert:planets}), a relatively-periodic solution is defined via the same $S^1$-action on the phase spaces of these problems, by identifying the phase spaces of these problems with the corresponding domains in the space $(\bbR^2)^N\times(\bbR^2)^N$ via the coordinates $(\ux,\uxxi)$ and $(\px,\pxxi)$ respectively.

\subsection {Symmetric solutions} \label {par:sym}

Let us show that the conditions 1 and 2 in Definition \ref {def:sym} of a
symmetric solution are equivalent. For this we will use the
invariance of the total energy $H$ of the system under the following
two involutions $S_l$ and $S$ in the phase space $T^*\widehat Q$ (cf.\ (\ref {eq:4N})) defined as follows.

Let us fix a line $l$ in the plane of motion passing through the
centre of masses of the system of particles. Define the following
three transformations in the phase space $T^*\widehat Q$ preserving the total
energy $H$ of the system:

1) the canonical involution $S_l:T^*\widehat Q\to T^*\widehat Q$ corresponding to the self-map (axial symmetry) 
of the configuration manifold $\widehat Q$ sending all particles of the system to their images under the symmetry with respect to the line $l$;

2) the anti-canonical involution $S$ (``reversion of time'')
sending each pair $(\q,\p)\in T^*\widehat Q$ to the pair $(\q,-\p)$ where $\q$
and $\p$ are the sets of ``coordinates'' and ``momenta'' of all particles of the system;

3) the anti-canonical involution $\J_l:=S S_l=S_lS$.

Each of these transformations is an involution, i.e.\ coincides with
its inverse. The first involution is canonical, i.e.\ preserves the
canonical symplectic structure $d\p\wedge d\q$ on $T^*\widehat Q$. The second
and the third involutions are anti-canonical, i.e.\ they affect the
symplectic structure by changing its sign to the opposite. Thus all
three involutions move trajectories of the system to trajectories,
moreover the first involution preserves the time on trajectories,
while the second and the third involutions ``reverse the time''.

A solution satisfies the first (respectively the second) condition of
symmetry if and only if the point of the phase space corresponding to
the time $t_0$ (respectively any time $t\in\bbR$) of this solution is
fixed (respectively is mapped to the point corresponding to the time
$2t_0-t$ of the same solution) under the involution $\J_l=S S_l=S_lS$. This shows the
equivalence of the conditions 1 and 2 in Definition \ref {def:sym} of a symmetric
solution.

A solution $\gamma(t)$ is symmetric and $(T,\alpha)$-periodic if and
only if its points $\gamma(t_0)$ and $\gamma(t_0+T/2)$ are fixed
under the involutions $\J_l$ and $\J_{R_{\alpha/2}(l)}$ respectively
where $R_{\varphi}:\bbR^2\to\bbR^2$ denotes the rotation by the angle
$\varphi$.

\subsection {A stable relatively-periodic solution} \label {par:ust}

Suppose that a Hamiltonian system $(X,\oomega,H)$ is
$S^1$-symmetric with respect to the Hamiltonian action of a circle
$S^1$ on $X$ via the Hamiltonian function $I$. Then the function $I$
is a first integral of the system. Consider the phase flow
$g_{H-\fomega_1I}^t:X\to X$, $t\in\bbR$, of the system with the
Hamilton function $H-\fomega_1I$. The map $A:=g_{H-\fomega_1I}^T$
will be called the {\it succession map}, and its linear part
$dA(\pnt)$ at a fixed point $\pnt$ will be called the {\it monodromy
operator} at this point.

Let us define a ``reduced'' succession map for the two-dimensional
torus $\gamma$ corresponding to a $(T,\fomega_1T)$-periodic solution.
Let $\Sigma\subset X$ be a small surface of codimension 2 that
transversally intersects the torus $\gamma$ at some point
$\pnt_0=\gamma\cap\Sigma$. Consider the restriction of the system to
a regular common level set
 $$
X_{H,I}:=\{H=\const,\ I=\const\}\supset\gamma
 $$
of the first integrals $H$ and $I$. Consider the small surface
$\sigma:=\Sigma\cap X_{H,I}$ of codimension 2 in $X_{H,I}$, which
transversally intersects the torus $\gamma$ at the point
$\pnt_0=\gamma\cap\sigma$. Consider the two-dimensional foliation on
the manifold $X_{H,I}$ whose fibres are invariant under the
(commuting) flows of the systems with Hamilton functions $H$ and
$I$; this condition uniquely determines fibres. Take the self-map
$\bar A$ of the surface $\sigma$ sending any point $\pnt\in\sigma$ to
the ``next intersection point'' of the fibre containing the point $\pnt$ with 
the surface $\sigma$. In more detail, the map $\bar
A:\sigma'\to\sigma$ is defined in a sufficiently small neighbourhood
$\sigma'\subset\sigma$ of the point $\pnt_0$ in $\sigma$, it is
``close'' to the map $A|_{\sigma'}=g_{H-\fomega_1I}^T|_{\sigma'}$ and
has the form $\bar A(\pnt)=g_{H-f_1(\pnt)I}^{f_2(\pnt)}(\pnt)$. Here
$f_1$ and $f_2$ are some smooth functions on $\sigma'$ defined by the
conditions $\bar A(\pnt)\in\sigma$, $\pnt\in\sigma'$,
$f_1(\pnt_0)=\fomega_1$, $f_2(\pnt_0)=T$. The map $\bar A$ will be
called the {\it reduced succession map} (or the {\it Poincar\'e
map}), and its linear part $\bA=d\bar A(\pnt_0)$ at the point
$\pnt_0$ will be called the {\it reduced monodromy operator}
corresponding to the torus $\gamma$.
As is well-known \cite {12}, the transversal surface
$\sigma$ (called a {\em cross section}) is a symplectic submanifold and the self-map $\bar A$ of
this surface is also symplectic. In particular, the reduced monodromy
operator $\bA$ is symplectic too. 

Recall that a linear operator $\bA$ is called {\it (Liapunov) stable} if the norm of the operator $\bA^n$ is bounded as $n\to\infty$. A symplectic operator $\bA$ is
called {\it structurally stable} if it is stable and any symplectic
operator that is close enough to $\bA$ is stable too.

\begin {scDef}  \label {def5}
The two-dimensional torus $\gamma$ and the corresponding
relatively- periodic solution will be called {\it orbitally
structurally stable in linear approximation \(OSSL\)} (respectively
{\it orbitally stable in linear approximation on a common level
surface of the first integrals of energy and angular momentum
\(OSLI\)}) if the reduced monodromy operator $\bA=d\bar A(\pnt_0)$
corresponding to the torus $\gamma$ is structurally stable
(respectively stable). The torus $\gamma$ will be called {\it
isoenergetically nondegenerate \(IN\)} if 1 does not belong to the
spectrum of the reduced monodromy operator $\bA$, i.e.\ $1\notin{\rm
spec\,}\bA$. The torus $\gamma$ is called {\it orbitally stable in
linear approximation \(OSL\)} if the linear operator $\bB=d\bar
B(\pnt_0)$ is stable, where $\bar B:\Sigma'\to\Sigma$ is the map
defined similarly to the Poincar\'e map $\bar A:\sigma'\to\sigma$.
 \end {scDef}

\begin {scDef}  \label {deff4}
An eigenvalue $\lambda\in\bbC$ of a symplectic operator $\bA$ is
called {\it elliptic} \cite {12} if it satisfies one of the following
equivalent conditions:

1) the quadratic form $Q(\xi)=\oomega(\bA\xi,\xi)$ is (positive or
negative) definite on the maximal invariant subspace where the
operator $\bA$ has no eigenvalues apart from $\lambda$ and
$\bar{\lambda}$;

2) the Hermitian quadratic form $\frac{1}{2i}\oomega(\xi,\bar{\xi})$
is (positive or negative) definite on the complex eigensubspace with
eigenvalue $\lambda$ of the complexified space.

The quadratic form $Q$ is called {\it the generating function} of the
symplectic operator $\bA$ (see also definition \ref {def:gen:func}).
 \end {scDef}

\begin {Pro}   [{\rm(см.\ \cite {Valigno})}] \label {stAB}
{\rm(A)} A symplectic operator is {\em stable} if and only if it is
diagonalizable over $\bbC$ and all its eigenvalues belong to the unit
circle in $\bbC$.

{\rm(B)} A symplectic operator is {\em structurally stable} if and
only if all its complex eigenvalues are {\em elliptic}.
 \qed
 \end {Pro}

Let us mention some connections between the stability properties
introduced above of an invariant two-dimensional torus $\gamma$:

1) the following implications hold: IN $\Leftarrow$ OSSL
$\Rightarrow$ OSL $\Rightarrow$ OSLI. (The second implication is an
important property of tori having the OSSL property; it follows from
property 3 below. The first implication follows from proposition \ref
{stAB}(B). The third implication is obvious. The inverse implications
are in general false);

2) if all eigenvalues of a symplectic operator $\bA$ are pairwise
different and lie on the unit circle in $\bbC$ then it is
structurally stable, thus the torus $\gamma$ is OSSL;

3) if the torus $\gamma$ is isoenergetically nondegenerate (IN) then
it is included into a smooth two-parameter family of isoenergetically
nondegenerate two-dimensional tori $\gamma_{H,I}$ where parameters of
the family are values of the first integrals 
$H$ and
$I$. If the torus $\gamma$ is OSSL (and, hence, IN)
then all the invariant tori of this family are also OSSL. Hence OSSL
$\Rightarrow$ OSL (and not only OSLI).

We stress that, if the torus $\gamma$ is OSLI and IN, then the other
tori of the family do not need to be OSLI, thus the torus $\gamma$
does not need to be OSL.

Thus, among the stability properties introduced in Definition \ref
{def5} for a relatively-periodic solution, the strongest one is the OSSL
property, which is studied in the present paper (see theorems \ref
{th:ust}(B) and \ref {th:stab}).

\subsection {Nondegeneracy condition} \label {subsec:period}

(A) Let us give {\em sufficient} conditions for nondegeneracy (\ref {alpha}) of a tuple of angular frequencies $\fomega_i,\Omega_{ij}$. Due to (\ref {poryadki:chastot'}), for small enough $0<\somega\ll1$, the nondegeneracy (\ref {alpha}) implies the delicate nondegeneracy (\ref {eq:fine}). 
In the presence of satellites ($N>n$), any of the nondegeneracy condition (\ref {alpha}) and the delicate nondegeneracy condition (\ref {eq:fine}) imply, due to the inequality $|\alpha|\le\pi$, that the relative period $T$ of the solutions under investigation is ``not too big'': 
 \begin{equation} \label {alpha''}
T<\frac\pi{\somega^2} \qquad \mbox{or} \qquad T<\frac\pi{C\somega^3} \qquad\quad\mbox{(for } N>n),
 \end{equation}
respectively. In the case $n>1$ (at least two planets) the relative period $T$ is always ``big'' for $0<\somega\ll1$. In more detail: $T=\frac{2\pi|k_2|}{|\fomega_2-\fomega_1|}>\frac{\pi|k_2|}{\somega}\ge\frac\pi\somega$ due to (\ref {maxrez}), (\ref {poryadki:chastot'}), (\ref {poryadki:chastot''}). Therefore, if $N>n>1$ then any of the relations (\ref {alpha''}) implies
 $$
\frac{\pi}{\somega}<T<\frac{\pi}{\somega^2} \qquad \mbox{or} \qquad 
\frac{\pi}{\somega}<T<\frac{\pi}{C\somega^3} \qquad\quad\mbox{(for } N>n>1),
 $$ 
respectively. However, if $N>n=1$ (only one planet, as in the case by V.N.\ Tkhay \cite {Thay}), then, due to (\ref {poryadki:chastot'}), the minimal positive relative period $T_{\min}\in\left[\frac{2\pi\max_j|K_{1j}|}{1+\somega},\frac{2\pi\min_j|K_{1j}|}{c-\somega}\right]$ is of order $1$ when $\somega\ll1$ and $c,K_{1j}$ are fixed.

(B) Let us give {\em sufficient} conditions for nondegeneracy (\ref {alpha}), which we will call {\it rough nondegeneracy conditions}.

(B1) Let us show that the nondegeneracy condition (\ref {alpha}) is always realizable for any $0<\somega\ll1$ and any relative period $T$ of the form
 \begin{equation} \label {eq:rough0} 
 T=\frac{2\pi a}{\somega}, \quad a\ge a_0(k_2,\dots,k_n):=\max\left\{7|k_i|,\ \sqrt{7(N-n+1)},\ \frac{21}2\right\}, 
 \end{equation}
for suitable collection of angular frequencies $\fomega_i,\Omega_{ij}$ with integers $k_2,\dots,k_n,K_{ij}$ of orders $1,\frac1\somega$. This will imply that the value $\somega^2T$ can be made arbitrary small, hence the nondegeneracy condition (\ref {alpha}) (and, hence, (\ref {eq:fine})) can be always fulfilled.
Indeed: let us fix integers $k_2,\dots,k_n\in\bbZ\setminus\{0\}$ having different absolute values, and real numbers $a\ge a_0(k_2,\dots,k_n)$ and $c\in(0,c_0(a)]$, where $c_0(a):=\min\{\frac1a,\frac1{14(N-n+1)}\}$. For any $b\in[\frac27,\frac37]$ and $0<\somega\le\min\{\somega_0,\somega_0(a)\}$, where $\somega_0(a):=\frac1{4a}$, put
 \begin{equation} \label {eq:rough1}
\fomega_1:=\somega b, \qquad \fomega_i:=\fomega_1+k_i\frac\somega a,
\qquad \Omega_{ij}:=\fomega_1+K_{ij}\frac\somega a
 \end{equation}
for any integers $K_{ij}\in\bbZ$ satisfying the following conditions for any $i$ and $j\ne j'$:
 \begin{equation} \label {eq:Kij}
|K_{ij}|\in\left[\frac{5a}{7\somega},\frac{6a}{7\somega}\right], \ ||K_{ij}|-|K_{ij'}||\ge\ell, \ \mbox{ where } \ell\in\bbN\cap\left[\frac{ca}\somega,\frac{a}{7(N-n+1)\somega}\right].
 \end{equation}
Such integers $\ell$ and $K_{ij}$ exist for $0<\somega\le\somega_0(a)=\frac1{4a}$, since the interval
$[\frac{ca}\somega,\frac{a}{7(N-n+1)\somega}]$ (respectively
$[\frac{5a}{7\somega},\frac{6a}{7\somega}]$)  is of the length
$\frac{a}{\somega}(\frac1{7(N-n+1)}-c)\ge\frac{2a^2}{7(N-n+1)}\ge2$ (respectively
$\frac{a}{7\somega}\ge(N-n+1)\ell$), and, hence, it
always contains a positive integer (respectively $N-n$ different
positive integers with pairwise distances $\ge\ell$). 
Since
$\alpha=\fomega_1T=2\pi ab(\mod2\pi)$ and $\somega^2T=2\pi
a\somega\le\frac\pi2$, the required nondegeneracy condition (\ref {alpha}) has the form 
$[a(b-\somega),a(b+\somega)]\cap\bbZ=\varnothing$. Hence it holds if
$[ab-\frac14,ab+\frac14]\cap\bbZ=\varnothing$, i.e.\ 
$b\in\frac1{2a}+(-\frac1{4a},\frac1{4a})+\frac1a\bbZ$. It follows
from $\frac17\ge\frac1{a}$ that, for any $a\ge a_0(k_2,\dots,k_n)$,
there exists $b\in[\frac27,\frac37]\cap((\frac1{4a},\frac3{4a})+\frac1a\bbZ)$. 
Moreover, for any 
 \begin{equation} \label {eq:rough2}
 a\ge a_0(k_2,\dots,k_n), \qquad  
 b\in \left[\frac27,\frac37\right]\cap\left(\left(\frac1{4a},\frac3{4a}\right)+\frac1a\bbZ\right)
 \end{equation}
the condition $b\in[\frac27,\frac37]$ and the nondegeneracy condition (\ref {alpha})
hold automatically for the period (\ref {eq:rough0}) and the collection of angular frequencies of the form (\ref {eq:rough1}), (\ref {eq:Kij}). 
The constructed frequencies $\fomega_i,\Omega_{ij}$ satisfy the inequalities 
$\frac17\le\frac{|\fomega_i|}\somega\le\frac47<|\Omega_{ij}|<1$,
$\frac{\left| |\fomega_i|-|\fomega_{i'}|\right|}\somega\ge\frac1a\ge c$,
$\left| |\Omega_{ij}|-|\Omega_{ij'}|\right|\ge\frac{\somega}a\ell\ge c$,
and, hence, the inequalities (\ref {poryadki:chastot'}) and (\ref {poryadki:chastot''}).
This proves the realizability of any period $T$ of the form (\ref {eq:rough0}). 
Thus, the system of relations (\ref {eq:rough0})--(\ref {eq:rough2}),
which we will call the {\it rough nondegeneracy condition}, gives a two-parameter family of nondegenerate tuples of frequencies $\fomega_i,\Omega_{ij}$ with parameters $T=\frac{\const}\somega$ and $\fomega_1\sim\somega$ (for integers $k_2,\dots k_n,K_{ij}$ of orders $1,\frac1\somega$), where $\somega\ll1$. 
In particular, the rough nondegeneracy holds if 
 $
 a\ge a_0(k_2,\dots,k_n)$, $b\in\left(\frac1a\left[\frac{3a}7+\frac14\right]-\frac3{4a},\ \frac1a\left[\frac{3a}7+\frac14\right]-\frac1{4a}\right)$, $0<\somega\le\min\{\somega_0,\somega_0(a)\}$.


We notice that Theorems~\ref {th1}--\ref {th:degen:sym} do not assume that the relative period $T>0$ in (\ref {parametry}) is minimal. They assume only that it satisfies either the nondegeneracy condition (\ref {alpha}), of the more delicate condition (\ref {eq:fine}), or the more rough condition (\ref {eq:rough0}), (\ref {eq:rough1}), (\ref {eq:Kij}), (\ref
{eq:rough2}), or the $T_{\min}$-rough condition (cf.\ (B3)). 
For example, we can assume that $T=\frac{\const}{\somega}$ (see above).

(B2) Let $N=n>1$ (there are no satellites, as in the cases by G.M.\ Krasinskii \cite {11} and, in particular, H.\ Poincar\'e \cite {1}). As in the general case (B1), one constructs a two-parameter family of tuples of frequencies $\fomega_i$ with parameters $T=\frac{2\pi a}\somega$ and $\fomega_1\sim\somega$ (for fixed integers $k_2,\dots k_n$, real numbers $a=\const\ge a_0$ and $0<\somega\le\min\{\somega_0,\somega_0(a)\}$), satisfying the {\it rough nondegeneracy condition} 
(\ref {eq:rough0}), (\ref {eq:rough1}), (\ref {eq:rough2}) with $N=n$, where the numbers $\Omega_{ij},K_{ij}$ in (\ref {eq:rough1}) are ignored. As above, the rough nondegeneracy implies the nondegeneracy (\ref {alpha}).

(B3) Let $N>n=1$ (only one planet, as in the cases by V.N.\ Tkhay \cite {Thay} and, in particular, G.W.\ Hill \cite {H}), i.e.\ we study a system Sun--planet--satellites. If $N=2$ (the system Sun--Earth--Moon, as in the case by G.W.\ Hill \cite {H}), the minimal positive relative period is $T_{\min}=\frac{2\pi}{|\Omega_{11}-\fomega_1|}\in[\frac{2\pi}{1+\somega},\frac{2\pi}{c-\somega}]$ and the angle of turning is $\alpha_{\min}=\fomega_1T_{\min}\in(-\pi,\pi)$ (so, they have orders $1$ and $\somega$, respectively) if $0<\somega<\min\{\somega_0,\frac{c}{3}\}$.
Therefore the conditions of a relative resonance (\ref {maxrez}) and nondegeneracy (\ref {alpha}) hold automatically for $T=T_{\min}$, for any $\fomega_1,\Omega_{11}$ of the form (\ref {chast}), (\ref {poryadki:chastot'}), (\ref {poryadki:chastot''}) and $0<\somega\le\min\{\somega_0,\frac{c}{3}\}$. More generally: if $N>n=1$ then, due to boundedness of $T_{\min}$ (cf.\ (A)), the angle of turning is $\alpha_{\min}=\fomega_1T_{\min}\in(-\pi,\pi)$, and the nondegeneracy condition (\ref {alpha}) holds for $T=T_{\min}$ if the tuple of frequencies $\fomega_1,\Omega_{1j}$ has the form (\ref {chast}), (\ref {poryadki:chastot'}), (\ref {poryadki:chastot''}), (\ref {maxrez}) and 
$0<\somega<\min\{\somega_0,\frac{c}{1+2\min_j|K_{1j}|}\}$.
On the other hand, if $n=1$ then the conditions (\ref {chast}), (\ref {poryadki:chastot'}) and (\ref {poryadki:chastot''}) hold automatically, provided that the tuple of frequencies has the form (\ref {maxrez}) for some relatively-prime tuple of nonvanishing integers $K_{1j}$ with pairwise different absolute values, with the properties $\frac{\min_j|K_{1j}|}{\max_j|K_{1j}|}=:r\in(c,1]$ and $\frac{\min_{j<j'}||K_{1j}|-|K_{1j'}||}{\max_j|K_{1j}|}=:r'\in(c,1]$, и каких-либо чисел $\somega\in(0,\min\{\somega_0,\frac{c}{1+2\min_j|K_{1j}|},1-\frac{1+c}{1+r},1-\frac{2+c}{2+r'}\})$,
 $\fomega_1:=\somega c$ and
 \\ $T\in\left[2\pi\frac{\max_j|K_{1j}|}{1-\somega},2\pi\min\{\frac{\min_j|K_{1j}|}{c+\somega},\frac{\min_{j<j'}||K_{1j}|-|K_{1j'}||}{c+2\somega}\}\right]$ (the latter segment has a nonempty interiour, since $0<\somega<\min\{1-\frac{1+c}{1+r},1-\frac{2+c}{2+r'}\}$ and $r,r'>c$).
The latter system of relations on $K_{1j},\somega,\fomega_1,T$ gives a two-parameter family of nondegenerate (i.e.\ satisfying (\ref {alpha})) tupls of frequencies $\fomega_1,\Omega_{1j}$ with parameters $T=T_{\min}\sim1$ and $0<\somega\ll1$ (and fixed $K_{1j}$) if $n=1$. 
Such tuples of frequencies will be called {\it $T_{\min}$-roughly nondegenerate} for $n=1$.


\section {Constructing relatively-periodic solutions of the unperturbed problem (\ref {eq:unpert0}), (\ref {eq:unpert1})}
 \label {par3:1:5}

Let us fix real numbers $m_i,m_{ij},\somega>0$ and an arbitrary tuple of ``angular frequencies'' $\fomega_i,\Omega_{ij}\in\bbR\setminus\{0\}$ of the form (\ref {chast}), (\ref {poryadki:chastot'}). Let $(\mx^0(t),\mxxi^0(t))=(\mx_*^0(t),\my_{**}^0(t),\mxxi_*^0(t),\meeta_{**}^0(t))$ be the main generating solution (\ref {eq:gen:sol}) of the model system (\ref {eq:model:H:omega}) corresponding to this tuple of frequencies. 

Let us study the unperturbed system (\ref {eq:unpert0}), (\ref {eq:unpert1}) near the curve $(\mx^0(t),\mxxi^0(t))$. We will describe a construction of its solution $(\ux^0(t),\uxxi^0(t))=(\ux_*^0(t),\uy_{**}^0(t),\uxxi_*^0(t),\ueeta_{**}^0(t))$ that is $O(\somega)$--close to the solution $(\mx^0(t),\mxxi^0(t))$ and satisfies that following conditions: $(\ux_*^0(t),\uxxi_*^0(t))=(\mx_*^0(t),\mxxi_*^0(t))$ and the vectors $\uy_{ij}^0(0)$ are collinear with the abscissa axis (moreover, the solution is $(T,\alpha)$-periodic if (\ref {maxrez}), (\ref {parametry})).

If there are no satellites then the unperturbed system coincides with the model one, and the solution is already constructed: $(\ux^0(t),\uxxi^0(t))=(\mx^0(t),\mxxi^0(t))$. 

Suppose that satellites are present. 
We have to find a solution $(\uy_{**}^0(t),\ueeta_{**}^0(t))$ of the non-autonomous Hamiltonian system (\ref {eq:unpert1}), where $\ux_*(t):=\mx_*^0(t)$.
This system splits into the direct product of systems (corresponding to ``satellites'') with the Hamilton functions and the symplectic structures 
 \begin {equation} \label {eq:unpert1ij}
  S^{(\ell)}_j(\uy_{\ell j},\ueeta_{\ell j};m_\ell,m_{\ell j}) + \somega^2 m_{\ell j} F(\mx_\ell^0(t),\uy_{\ell j}), \qquad 
 d\ueeta_{\ell j}\wedge d\uy_{\ell j} 
 \end{equation}
corresponding to the $j$-th satellite of the $\ell$-th planet ($1\le\ell\le n$, $1\le j\le n_\ell$),
 \begin{equation} \label {eq:S:F}
 S^{(\ell)}_j(\uy_{\ell j},\ueeta_{\ell j};m_\ell,m_{\ell j}) := \frac{\ueeta^2_{\ell j}}{2m_{\ell j}} -\frac{m_\ell m_{\ell j}}{|{\uy}_{\ell j}|}, \quad
 F(\hx,\hy):=\frac{\hx^2\hy^2-3\langle\hx,\hy\rangle^2}{2|\hx|^5} . 
 \end{equation}
So, it suffices to construct, for each pair $(\ell,j)$, a solution $(\uy_{\ell j}(t),\ueeta_{\ell j}(t))$ of the system (\ref {eq:unpert1ij}) that is $O(\somega)$--close to $(\mx^0_{\ell j}(t),\mxxi^0_{\ell j}(t))$ and satisfies the following conditions: it is relatively-periodic with the same relative period $(\frac{2\pi}{|\Omega_{\ell j}-\fomega_\ell|},\frac{2\pi\fomega_\ell}{|\Omega_{\ell j}-\fomega_\ell|})$ as for the solution $(\mx^0_{\ell j}(t),\mxxi^0_{\ell j}(t))$, and the radius-vector $\uy_{\ell j}^0(0)$ is collinear with the abscissa axis. In fact: if (\ref {maxrez}), (\ref {parametry}) then such a solution will be automatically $(T,\alpha)$-periodic, since $(K_{\ell j}-k_\ell)(\frac{2\pi}{\Omega_{\ell j}-\fomega_\ell}, \frac{2\pi\fomega_\ell}{\Omega_{\ell j}-\fomega_\ell})=(T,\alpha + 2\pi(k_\ell-k))$. 

Let us apply the transformation 
$(\uy_{\ell j},\ueeta_{\ell j})=:e^{i\fomega_\ell t}(\hy,\heeta)$, which is equivalent to passing to a {\it synodic} coordinate system (i.e.\ a coordinate system rotating with the ``scaled planet'' $\mx_\ell^0(t)={e^{i\fomega_\ell t}}\hx\o$, where $\hx\o:=({\Omega_{\ell 0}^{-2/3}},0)$).
This transformation reduces the system (\ref {eq:unpert1ij}) to the autonomous Hamiltonian system with the Hamilton function and the symplectic structure
 \begin {equation} \label {eq:Hill}
S^{(\ell)}_j(\hy,\heeta;m_\ell,m_{\ell j}) - \fomega_\ell [\hy,\heeta]
 + \fomega_\ell^2 {m_{\ell j}} F(\hx,\hy)|_{\hx=(1,0)}, 
 \quad d\heeta\wedge d\hy .
 \end {equation}
The system (\ref {eq:Hill}) with the parameter values $\fomega_\ell=m_\ell=m_{\ell j}=1$ describes the known {\it Hill problem}, and any system (\ref {eq:Hill}) with any nonzero parameter values reduces to it by some power transformaion (cf.\ \S \ref {subsec:Hill:0}).
The summand $F(\hx,\hy)|_{\hx=(1,0)}=-\hyi_1^2+\hyi_2^2/2$ in the Hamilton function of the problem (\ref {eq:Hill}) (and, hence, of the Hill problem), as well as the function $F(\hx,\hy)$, will be called the {\it Hill potential} (or the ``limiting potential of the action of the Sun to the satellite'').

Our purpose is the following. For small enough $|\fomega_\ell|$, we have to construct a $\frac{2\pi}{|\Omega_{\ell j}-\fomega_\ell|}$-periodic solution $(\hy^0(t),\heeta^0(t))$ of the system (\ref {eq:Hill}) that is $O(\fomega_\ell)$--close to the $\frac{2\pi}{|\Omega_{\ell j}-\fomega_\ell|}$-periodic solution $e^{-i\fomega_\ell t}(\my_{\ell j}^0(t),\meeta_{\ell j}^0(t))$ of the similar system with $F\equiv0$ and has the following property: its initial point $\hy^0(0)$ belons to the abscissa axis. 
For doing this, we will first check (in \S \ref {par3:1:5'}) that the system (\ref {eq:Hill}) satisfies the hypotheses of the technical Lemma \ref {lem:mult}(A), which means that the Hamilton function $S^{(\ell)}_j(\hy,\heeta;m_\ell,m_{\ell j})$ of the planar Kepler problem satisfies the conditions (\ref {eq:isoen}) in some canonical coordinates $\varphi_{\ell j}\mod2\pi,I_{\ell j},q_{\ell j},p_{\ell j}$.

\subsection {Normalizing the Kepler problem near circular orbits} \label {par3:1:5'}

Let us show that the Hamilton function of any planar
Kepler's problem has the form $H=H(I,q,p)$
and satisfies the conditions (\ref {eq:isoen}) of Theorem \ref {th1'},
with respect to some canonical coordinates $\varphi,I,q,p$ in some neighbourhood of the union of phase trajectories of circular solutions.

The {\it planar Kepler problem} is given by the Hamiltonian system
 \begin{equation} \label {HamKepler}
 \left(M=T^*(\bbR^2\setminus\{0\}),\ \oomega=d\p\wedge d\q,\ H=\frac{\p^2}{2m}-\frac{km}{|\q|}=\Kin+U\right).
 \end{equation}
Here $\Kin=\frac{\p^2}{2m}$ and $U=-\frac{km}{|\q|}$ are the kinetic
and potential energies of the system, $\q\in\bbR^2\setminus\{0\}$ and
$\p\in\bbR^2$ are the radius vector and the momentum of the
particle, $m,k>0$ are parameters. The solutions of the Kepler problem
having negative energy levels $H$ are periodic. Since the Kepler
problem is invariant under all rotations of the plane, it has the
first integral of angular momentum $I=[\q,\p]$.

Note simple properties of circular motions in the Kepler problem:

1) For any $r>0$, there is a unique (up to changing the direction of
rotation) circular motion of the particle satisfying the system (\ref
{HamKepler}) along a circle of radius $r$. The angular velocity of
this motion equals $\Omega=\pm\sqrt{\frac{k}{r^3}}$, while the energy
$H$ and the angular momentum $I$ equal $H=-\frac{km}{2r}$ and $I=m\Omega
r^2=\pm m\sqrt{kr}$ respectively. In particular, the values $H$ and
$I$ depend monotonically on the value $\Omega$ (for $\Omega>0$ or
$\Omega<0$) and take all values in the domains $H<0$ and $I\ne 0$
respectively. We will assume that the parameters $r$, $\Omega$, $H$,
$I$ of a circular motion are related by the formulae from above.

2) Circular motions correspond to the equilibrium (i.e.\ stationary)
positions of the particle with respect to a rotating coordinate
system with angular velocity $\Omega$. Therefore, for any $\alpha\ne
0$, the solution of the Kepler problem corresponding to a circular
motion is $(T,\alpha)$-periodic with $T=\frac{\alpha}{\Omega}$. In
other words, such a solution is $\frac{\alpha}{\Omega}$-periodic with
respect to a rotating coordinate system with angular velocity
$\Omega$ (as well as any angular velocity of the form
$(1+\frac{2\pi\ell}{\alpha})\Omega$ where $\ell$ is an integer).

The Kepler problem of our interest (\ref {eq:unpert1ij}) with $\somega=0$ (or (\ref {eq:Hill}) with $\fomega_\ell=0$) 
has the form (\ref {HamKepler}), where $H=S^{(\ell)}_{j}$, $I=I_{\ell j}$,
 \begin{equation} \label {eq:Kepler:param}
 (\M,\p) =(\hy,\heeta), \quad 
 m= m_{\ell j}, \quad k=m_\ell,
 \quad r=r_{\ell j}, \quad \Omega=\Omega_{\ell j}.
 \end{equation}


\begin {Lem} [(normalizing the Kepler problem)] \label {lem3:2:}
Let $H$ and $I$ be the Hamilton function and the 
angular momentum first integral of the planar Kepler problem {\rm(\ref {HamKepler})}.
In the domain $\{r>0,\ p_\psi\ne0\}$ of the phase space of the
problem, consider the coordinates $\varphi\mod2\pi$, $I$, $q$, $p$ of
the form
 \begin{equation} \label {nov:koord}
 \varphi=\psi-\frac{2rp_r}{p_\psi}, \quad  I=p_\psi, \qquad q=\ln \frac{km^2r}{p^2_\psi}, \quad p=rp_r.
 \end{equation}
Here $r$, $\psi$ are the polar coordinates in the plane of motion,
$p_r$, $p_\psi=I$ are the corresponding momenta. Then:

{\rm(a)} the coordinates {\rm(\ref {nov:koord})} are canonical, i.e.\
$\oomega=d I\wedge d\varphi+dp\wedge dq$;

{\rm(b)} in the coordinates {\rm(\ref {nov:koord})}, the Hamilton
function $H$ of the Kepler problem does not depend on the angular
coordinate $\varphi\mod 2\pi$ and has the form
 $$
 \frac{H}{k^2m^3}=\frac{I^2+p^2}{2I^4e^{2q}}-\frac{1}{I^2e^q}
 =-\frac1{2I^2} + \frac{q^2}{2I^2} + \frac{p^2}{2I^4}
 + o(q^2+p^2)
 $$
as $q^2+p^2\to0$. Furthermore the involutions $S_l$, $S$ in
{\rm\Sec\ref {par:sym}} with $l=\bbR\times\{0\}\subset\bbR^2$ have the form
 \begin{equation} \label {0}
S_l(\varphi,I,q,p)=(-\varphi,-I,q,p), \quad S(\varphi,I,q,p)=(\varphi,-I,q,-p).
 \end{equation}
 \end{Lem}

One proves Lemma \ref {lem3:2:} in a direct way. \qed

The canonical coordinates $\varphi\mod 2\pi,I,q,p$ in Lemma~\ref
{lem3:2:}, as follows from their construction, are quite ``similar''
to the canonical coordinates $\psi\mod 2\pi,p_\psi=I,r,p_r$
corresponding to the polar coordinates $\psi\mod 2\pi,r$ in the
configuration space of the planar Kepler problem. For example, the
involutions $S_l$, $S$ have the form (\ref {0}) in the coordinates
$\psi,p_\psi,r,p_r$ too:
 $$
S_l(\psi,p_\psi,r,p_r)=(-\psi,-p_\psi,r,p_r), \quad
S(\psi,p_\psi,r,p_r)=(\psi,-p_\psi,r,-p_r).
 $$

For any $\Omega\neq0$, denote by $\gamma_\Omega$ the phase trajectory
of the Kepler problem corresponding to the circular motion with angular
velocity $\Omega$. The invariant two-dimensional surface
$\bigcup\limits_{\Omega\ne0}\gamma_\Omega$ in the phase space formed by all
these trajectories will be called {\it the surface of circular
motions}. This surface is smooth and consists of two connected
components, each of which is diffeomorphic to a punctured plane and
bijectively projects onto the configuration manifold under the
canonical projection.

We recall that the system describing the motion with respect to a
rotating coordinate system with angular velocity $\Omega$ is a
Hamiltonian system with the Hamiltonian function $H-\Omega I$. Thus
the circular motions correspond to the stationary points of such
systems.

\begin{Cor} \label {corlem3:2:}
The surface of circular motions of the planar Kepler problem is a region in the symplectic
coordinate cylinder with respecto to the coordinates $\varphi\mod2\pi$, $I$, $q$, $p$ from Lemma {\rm\ref {lem3:2:}}:
 $$
\bigcup_{\Omega\in\bbR\setminus\{0\}}\gamma_\Omega=\{(\varphi\mod2\pi,I,0,0)\mid I\ne0\} =S^1\times(\bbR\setminus\{0\})\times\{(0,0)\}.
 $$
For any $\Omega\in\bbR\setminus\{0\}$, at any point of the circle $\gamma_\Omega$, the differential of the function $H-\Omega I$  equals $0$, while the quadratic part \(i.e.\ the
quadratic form whose matrix is formed by the second partial
derivatives\) of this function 
has a diagonal form with respect to the coordinates {\rm(\ref {nov:koord})}:
 \begin{equation} \label {d2H}
\delta^2(H-\Omega I)|_{\gamma_\Omega}=-\frac{3}{mr^2}\delta I^2+
 {\Omega}\left(I\delta q^2+\frac{\delta p^2}{I}\right).
 \end{equation}
 \end{Cor}

\subsection {Constructing a family of periodic solutions of the Hill problem (the Moon theory)}
\label {subsec:Hill:0}

Рассмотрим частный случай плоской задачи трёх тел --- систему типа Солнце--Земля--Луна ($N=2$, $n=1$). Изучим в этом случае невозмущённую систему и построим её относительно-периодические решения. 

Пусть, как выше, ``масштабированная планета'' совершает круговое движение $(\ux_1^0(t),\uxxi_1^0(t))=(\mx_1^0(t),\mxxi_1^0(t))=
 e^{i\fomega_1 t} (\Omega_{10}^{-2/3},i\Omega_{10}^{1/3} m_1)$ с угловой скоростью $\fomega_1$. 
Тогда, согласно (\ref {eq:Hill}), невозмущённое движение ``спутника'' в синодической (т.е.\ вращающейся с угловой скоростью $\fomega_1$) системе координат описывается гамильтоновой системой с функцией Гамильтона и симплектической структурой 
 \begin{equation} \label {eq:unpert:Hill}
H_1
= \frac{\heeta^2}{2m} -\frac{km}{|\hy|} - \fomega_1 [\hy,\heeta] + \fomega_1^2 m \widehat F(\hy), 
 \qquad \oomega=d\heeta\wedge d\hy, 
 \end{equation}
с параметрами $m,k>0$ и $\fomega_1\ne0$, где $m:=m_{11}$, $k:=m_1$ (как в (\ref {eq:Kepler:param})), 
 \begin{equation} \label {eq:F}
\widehat F(\hy):=-\hyi_1^2+\hyi_2^2/2=F(\hx,\hy)|_{\hx=(1,0)}
 \end{equation}
--- потенциал Хилла (``предельный потенциал действия Солнца на спутник'').

Задача (\ref {eq:unpert:Hill}) при единичных значениях параметров $m=k=\fomega_1=1$ является {\it задачей Хилла}, с функцией Гамильтона и симплектической структурой
 \begin{equation} \label {eq:unpert:Hill'}
 \widetilde H_1 
 = \frac{\widetilde\heeta^2}{2}-\frac{1}{|\widetilde\hy|} - [\widetilde\hy,\widetilde\heeta] 
 + \widehat F(\widetilde\hy), 
 \qquad
 \widetilde\oomega=d\widetilde\heeta\wedge d\widetilde\hy.
 \end{equation}
Покажем, что задача (\ref {eq:unpert:Hill}) приводится к задаче Хилла (\ref {eq:unpert:Hill'}) некоторой заменой переменных. Дробно-степенное преобразование 
 $\widehat\hy=k^{-1/3}\hy$, 
 $\widehat\heeta=m^{-1}k^{-1/3}\heeta$, 
 $\widehat H_1=m^{-1}k^{-2/3}H_1$,
 $\widehat\oomega=m^{-1}k^{-2/3}\oomega$
приводит гамильтонову систему (\ref {eq:unpert:Hill}) к системе с функцией Гамильтона и симплектической структурой 
 \begin{equation} \label {eq:my:unpert:Hill}
 \widehat H_1=
 \frac{\widehat\heeta^2}{2}-\frac{1}{|\widehat\hy|} - \fomega_1 [\widehat\hy,\widehat\heeta] 
 + \fomega_1^2 \widehat F(\widehat\hy) ,
 \qquad \widehat\oomega = d\widehat\heeta\wedge d\widehat\hy
 \end{equation}
с параметром $\fomega_1>0$. Преобразование 
 $\widetilde\hy=\fomega_1^{2/3}\widehat\hy$, 
 $\widetilde\heeta=\fomega_1^{-1/3}\widehat\heeta$, 
 $\widetilde H_1=\fomega_1^{-2/3}\widehat H_1$,
 $\tilde t=\fomega_1t$,
 $\widetilde\oomega=\fomega_1^{1/3}\widehat\oomega$
приводит систему (\ref {eq:my:unpert:Hill}) к требуемой задаче Хилла (\ref {eq:unpert:Hill'}).

Рассмотрим ``стандартную'' (без параметров) ``синодическую задачу Кеплера'' с функцией Гамильтона и симплектической структурой
 \begin{equation} \label {eq:sin:Kepler:1}
 \frac{\widetilde\p^2}{2}-\frac{1}{|\widetilde\M|} - [\widetilde\M,\widetilde\p] , \qquad d\widetilde\p\wedge d\widetilde\q.
 \end{equation}
Её функция Гамильтона получена из функции Гамильтона $\widetilde H_1$ задачи Хилла (\ref {eq:unpert:Hill'}) откидыванием последнего слагаемого (т.е.\ при $\widehat F\equiv0$).
Рассмотрим 1-параметрическое семейство ``круговых'' решений
 \begin{equation} \label {eq:circ:sol:1}
 (\widetilde\q_{\widetilde\Omega}(\tilde t),\widetilde\p_{\widetilde\Omega}(\tilde t))
 = e^{i(\widetilde\Omega-1)\tilde t}(\widetilde\Omega^{-2/3},i\widetilde\Omega^{1/3})
 \end{equation}
задачи (\ref {eq:sin:Kepler:1}), где $\tilde t$ --- время, $\widetilde\Omega:=\Omega/\fomega_1$ --- отношение угловой скорости сидерического кругового вращения ``спутника'' к угловой скорости кругового вращения ``масштабированной планеты'' (которая здесь считается равной 1).

\begin {Exa} [(см.\ пример 1 в \cite{K:vest13})] \label {ex:Henon}
Многие семейства периодических решений задачи Хилла (\ref {eq:unpert:Hill'}) хорошо известны и изучены \cite {Bruno07}, причём некоторые из этих семейств ``достаточно близки'' к семейству круговых решений (\ref {eq:circ:sol:1}) синодической задачи Кеплера (\ref {eq:sin:Kepler:1}). Например, известны два 1-параметрических семейства $f$ и $g_+$ периодических решений задачи Хилла, параметром на каждом из которых служит гамильтониан $\widetilde H_1\in(-\infty,\infty)$, или минимальный положительный период
$\widetilde T\in(0,T_f)$ и $\widetilde T\in(0,T_{g_+})$ соответственно, где $T_f\approx2\pi$ и $T_{g_+}\approx4\pi$ \cite[tables 3, 4]{Henon69}, или отношение
 $\widetilde\Omega=1\pm\frac{2\pi}{\widetilde T}\ne0$ 
средней угловой частоты (сидерического) вращения ``спутника'' к угловой частоте вращения ``масштабированной планеты''. Оба семейства $f$ и $g_+$ начинаются (при $\widetilde\Omega\to-\infty$ или $\widetilde\Omega\to+\infty$) квази-круговыми орбитами вокруг точки $\widetilde\hy=0$, причём направление движения на первом семействе --- обратное (по часовой стрелке, $\widetilde\Omega<0$), а на втором --- прямое (против часовой стрелки, $\widetilde\Omega>0$). Согласно численному результату М.\ Хенона \cite[tables 11, 12]{Henon69}, семейства $f$ и $g_+$ ``достаточно близки'' к семейству круговых решений (\ref {eq:circ:sol:1}) синодической задачи Кеплера (\ref {eq:sin:Kepler:1}) на следующих начальных участках этих семейств: когда параметр $\Gamma:=-2\widetilde H_1$ семейства удовлетворяет оценке $\Gamma\ge1$ и $\Gamma\ge5$ соответственно (т.е.\ период решения удовлетворяет оценке $\widetilde T\le1.51822$ и $\widetilde T\le0.92296$ соответственно \cite[tables 3, 4]{Henon69}). Наш подход даёт построение начальных участков семейств $f$ и $g_+$ (см.\ (\ref {eq:sol:Hill}) ниже) и их ``достаточную близость'' к семейству круговых решений (\ref {eq:circ:sol:1}) при следующей оценке на период решения:
$\widetilde T\le2\pi\somega_0$, где $\somega_0\in(0,1)$ --- некоторая (достаточно малая) константа. Видимо, наша верхняя грань $2\pi\somega_0$ для периода меньше обоих значений $1.51822$ и $0.92296$, указанных Хеноном.
Движение Луны в системе Солнце-Земля-Луна приближённо соответствует решению из семейства $g_+$ с периодом $\widetilde T=2\pi\cdot\frac{28}{365}\approx\frac{2\pi}{13}\approx0.48$.
\end{Exa}

Рассмотрим задачу Хилла (\ref {eq:unpert:Hill'}) (к которой, как мы показали, сводится задача (\ref {eq:unpert:Hill}), описывающая невозмущённое движение ``спутника''). Мы хотим найти два 1-параметрических семейства её периодических решений
$(\widetilde\hy_{\widetilde\Omega}(\tilde t),\widetilde\heeta_{\widetilde\Omega}(\tilde t))$ при достаточно большом отношении средних частот $|\widetilde\Omega|\gg1$.

Наша основная идея состоит в следующем. Вместо задачи Хилла (\ref {eq:unpert:Hill'}) мы рассмотрим (эквивалентную ей) ``параметрическую задачу Хилла'' (\ref {eq:unpert:Hill}) или (\ref {eq:my:unpert:Hill}) с малым параметром $\fomega_1$ и построим (в (\ref {eq:sol:sin}) и (\ref {eq:my:sol'}) ниже) её периодические решения методами теории возмущений (см.\ лемму \ref {lem:mult}(A)).

Проведём построение искомого семейства $\frac{2\pi}{|\Omega-\fomega_1|}$-периодических решений (см.\ (\ref {eq:my:sol'}) ниже) задачи (\ref {eq:my:unpert:Hill}) 
с помощью леммы \ref {lem:mult}(A).

{\it Шаг 1.} Опишем сначала ``порождающие'' круговые решения. При откидывании последнего слагаемого (порядка $\fomega_1^2$) из функции Гамильтона $\widehat H_1$ задачи (\ref {eq:my:unpert:Hill}), т.е.\ при $\widehat F\equiv0$, получаем ``1-параметрическую синодическую задачу Кеплера'' (для ``спутника'') с гамильтонианом и симплектической структурой
$$
 \frac{\widehat\p^2}{2}-\frac{1}{|\widehat\M|} - \fomega_1 [\widehat\M,\widehat\p] , \qquad d\widehat\p\wedge d\widehat\q,
$$
с параметром $\fomega_1\in\bbR$. Рассмотрим семейство ``круговых'' решений этой задачи:
 \begin{equation} \label {eq:circ:sol}
 (\widehat\q_{\Omega,\fomega_1}(t),\widehat\p_{\Omega,\fomega_1}(t))
  = e^{i(\Omega-\fomega_1)t}(\Omega^{-2/3},i\Omega^{1/3}), \qquad \Omega,\fomega_1\in\bbR,\ \Omega\ne0,
 \end{equation}
где $\fomega_1$ и $\Omega$ --- угловые скорости сидерического кругового вращения ``масштабированной планеты'' и ``спутника'' соответственно. Круговое решение (\ref {eq:circ:sol}) является $\frac{2\pi}{|\Omega-\fomega_1|}$-периодическим, его фазовая траектория $\gamma_\Omega$ не зависит от $\fomega_1$.

{\it Шаг 2.} Рассмотрим ``1-параметрическую задачу Хилла'' (\ref {eq:my:unpert:Hill}) c малым параметром $\fomega_1$. Её фазовое пространство есть $M=(\bbR^2\setminus\{0\})\times\bbR^2$ с координатами $\widehat\hy,\widehat\heeta$. При $\fomega_1=0$ получаем задачу Кеплера $(M,\widehat\oomega,\widehat H:=\widehat H_1|_{\fomega_1=0})$.

Фиксируем число $\Omega\ne0$ (без ограничения общности можно положить $\Omega:=1$). 
Изоэнергетическая поверхность $\Pi\o=\Pi\o_{|\Omega|}:=\widehat H^{-1}(-\Omega^{2/3}/2)$ задачи Кеплера заполнена $\frac{2\pi}{|\Omega|}$-периодическими траекториями задачи Кеплера.
Рассмотрим окружность $\gamma_\Omega\subset\Pi\o$, т.е.\ фазовую траекторию задачи Кеплера, отвечающую круговому решению $(\widehat\q_{\Omega,0}(t),\widehat\p_{\Omega,0}(t))$ из (\ref {eq:circ:sol}) с угловой скоростью $\Omega$ (см.\ \S \ref {par3:1:5'}).
Согласно лемме \ref {lem3:2:} и ее следствию \ref {corlem3:2:}, 1-параметрическая задача Хилла (\ref {eq:my:unpert:Hill}) удовлетворяет условиям (\ref {eq:isoen}). Поэтому, согласно лемме \ref {lem:mult}(A),
существует столь малое число $\somega_0>0$ и гладкое 1-параметрическое семейство $\frac{2\pi}{|\Omega|}$-пе\-ри\-о\-ди\-чес\-ких траекторий (``возмущённой'') системы (\ref {eq:my:unpert:Hill}) с параметром $\fomega_1\in(-\somega_0,\somega_0)$, такое, что нулевому значению параметра $\fomega_1=0$ отвечает траектория кругового решения (т.е.\ окружность $\gamma_\Omega$), а при $\fomega_1\in(-\somega_0,0)\cup(0,\somega_0)$ траектория семейства является единственной $\frac{2\pi}{|\Omega|}$-пе\-ри\-о\-ди\-чес\-кой траекторией системы (\ref {eq:my:unpert:Hill}),
$O(\somega_0)$--близкой к окружности $\gamma_\Omega$. На каждой траектории этого семейства выберем параметризацию так, чтобы она задавала $\frac{2\pi}{|\Omega|}$-пе\-ри\-о\-ди\-чес\-кое решение системы (\ref {eq:my:unpert:Hill}), положение которого в начальный момент времени принадлежит оси абсцисс. Полученное 1-параметрическое семейство $\frac{2\pi}{|\Omega|}$-пе\-ри\-о\-ди\-чес\-ких решений системы (\ref {eq:my:unpert:Hill}) (при фиксированном $\Omega\ne0$) обозначим через 
 \begin {equation} \label {eq:my:sol}
(\widehat\hy_{\fomega_1}(t),\widehat\heeta_{\fomega_1}(t)), \qquad \fomega_1\in(-\somega_0,\somega_0).
 \end{equation}
Имеем $(\widehat\hy_{0}(t),\widehat\heeta_{0}(t))=(\widehat\q_{\Omega,0}(t),\widehat\p_{\Omega,0}(t))$. 

{\it Шаг 3.} Напомним (шаг 2),
что семейство (\ref {eq:my:sol})  $\frac{2\pi}{|\Omega|}$-пе\-ри\-о\-ди\-чес\-ких решений системы (\ref {eq:my:unpert:Hill}) построено для фиксированного числа $\Omega\ne0$. 
Пусть для определённости это число $\Omega:=1$. Для остальных $\Omega\ne0$ решения строятся так.

Из семейства (\ref {eq:my:sol}) при $\Omega:=1$ мы получаем два (обнаруженных ещё Хиллом \cite {H,W}) 1-пара\-мет\-ри\-чес\-ких семейства $\frac{2\pi}{|\widetilde\Omega-1|}$-пе\-ри\-о\-ди\-чес\-ких решений 
 \begin {equation} \label {eq:sol:Hill}
 \begin{array}{c} 
 \left(\widetilde\hy_{\widetilde\Omega}(\tilde t),\widetilde\heeta_{\widetilde\Omega}(\tilde t)\right) := \\ :=\left((\widetilde\Omega-1)^{-2/3}\widehat\hy_{(\widetilde\Omega-1)^{-1}}((\widetilde\Omega-1)\tilde t),\ (\widetilde\Omega-1)^{1/3}\widehat\heeta_{(\widetilde\Omega-1)^{-1}}((\widetilde\Omega-1)\tilde t)\right), \\ \widetilde\Omega\in(-\infty,1-\somega_0^{-1})\cup(1+\somega_0^{-1},+\infty),
 \end{array}
 \end{equation}
задачи Хилла (\ref {eq:unpert:Hill'}). 
Параметр первого семейства отрицателен, а параметр второго --- положителен. (Для получения формул (\ref {eq:sol:Hill}) надо положить $\fomega_1:=(\widetilde\Omega-1)^{-1}$ в (\ref {eq:my:sol}) и использовать преобразование $\widetilde\hy:=\fomega_1^{2/3}\widehat\hy$,  $\widetilde\heeta:=\fomega_1^{-1/3}\widehat\heeta$, $\tilde t:=\fomega_1t$, приводящее задачу (\ref {eq:my:unpert:Hill}) к задаче Хилла.)
Семейства решений (\ref {eq:sol:Hill}) --- это (обнаруженные Хиллом \cite {H,W}) начальные участки хорошо известных и изученных семейств $f$ и $g_+$ периодических решений задачи Хилла \cite {Bruno07}. В каждом из семейств $f$ и $g_+$ интервал изменения периода содержит указанный нами интервал $(0,2\pi\somega_0)$ (см.\ также 
пример \ref {ex:Henon}). Оба эти семейства порождают семейства периодических решений круговой ограниченной задачи трёх тел (см.\ \cite {brown1892,32}, \cite[\S 17--19]{33}, \cite {Perko83,Bruno90,Bruno07}) и относительно-периодических решений задачи трёх тел (см.~\cite {32} и \cite[\S 18--19]{33}).

Из двух семейств (\ref {eq:sol:Hill}) решений задачи Хилла получаем искомое 2-пара\-мет\-ри\-чес\-кое семейство $\frac{2\pi}{|\Omega-\fomega_1|}$-пе\-ри\-о\-ди\-чес\-ких решений системы (\ref {eq:my:unpert:Hill}):
 \begin{equation} \label {eq:my:sol'}
 \begin{array}{c} \widehat\gamma_{\Omega,\fomega_1}(t)=\left(\widehat\hy_{\Omega,\fomega_1}(t),\widehat\heeta_{\Omega,\fomega_1}(t)\right):= \\ :=\left((\Omega-\fomega_1)^{-2/3}\widehat\hy_{\frac{\fomega_1}{\Omega-\fomega_1}}((\Omega-\fomega_1)t),\  (\Omega-\fomega_1)^{1/3}\widehat\heeta_{\frac{\fomega_1}{\Omega-\fomega_1}}((\Omega-\fomega_1)t)\right), \\
 \Omega\ne0, \quad \frac{\fomega_1}\Omega \in
 (-\frac{\somega_0}{1-\somega_0},\frac{\somega_0}{1+\somega_0}).
 \end{array}
 \end{equation}
(Для этого надо положить $\widetilde\Omega:=\Omega/\fomega_1$ в (\ref {eq:sol:Hill}) и использовать обратное преобразование переменных.) Отметим, что решение (\ref {eq:my:sol}) равно $\widehat\gamma_{1+\fomega_1,\fomega_1}(t)$.

\begin {scRem} \label {rem:Hill}
Построенное (с помощью технической леммы \ref {lem:mult}(A)) 2-параметрическое семейство решений (\ref {eq:my:sol'}) ``1-параметрической задачи Хилла'' (\ref {eq:my:unpert:Hill}) гладко зависит от параметров $\Omega,\fomega_1$ и при $\fomega_1=0$ совпадает с семейством круговых решений (\ref {eq:circ:sol}). 
В частности, эти два семейства $\frac{2\pi}{|\Omega-\fomega_1|}$-пе\-ри\-о\-ди\-чес\-ких решений, (\ref {eq:my:sol'}) и (\ref {eq:circ:sol}), являются $O(\fomega_1)$--близкими. Покажем, что из технической леммы \ref {lem:mult}(B) следует, что эти семейства даже $O(\fomega_1^2)$--близки. Рассмотрим функцию $\langle\widehat F\o\rangle$ на изоэнергетической поверхности $\Pi\o$ задачи Кеплера (см.\ шаг 2),
полученную усреднением функции $\widehat F\o:=\widehat F|_{\Pi\o}$ по $\frac{2\pi}{\Omega}$-периодическим решениям задачи Кеплера $(M,\widehat\oomega,\widehat H)$. Несложным вычислением, с учётом (\ref {eq:F}) и следствия \ref {corlem3:2:}, проверяется, что окружность $\gamma_\Omega\subset\Pi\o$ является критическим множеством функции $\langle\widehat F\o\rangle$. Поэтому из леммы \ref {lem:mult}(B) следует требуемая оценка $\widehat\gamma_{\Omega,\fomega_1}(t)=(\widehat\q_{\Omega,\fomega_1}(t),\widehat\p_{\Omega,\fomega_1}(t))+O(\fomega_1^2)$.
\end {scRem}

Таким образом, мы построили следующее невозмущённое синодическое (т.е.\ описываемое задачей (\ref {eq:unpert:Hill})) $\frac{2\pi}{|\Omega-\fomega_1|}$-периодическое движение ``спутника'':
 \begin{equation} \label {eq:sol:sin}
(\hy^0(t),\heeta^0(t)) := k^{1/3}(\widehat\hy_{\Omega,\fomega_1}(t),m\widehat\heeta_{\Omega,\fomega_1}(t)),
 \quad  
 \frac{\fomega_1}\Omega \in
 \left(-\frac{\somega_0}{1-\somega_0},\frac{\somega_0}{1+\somega_0}\right),
 \end{equation}
см.\ (\ref {eq:my:sol'}), (\ref {eq:my:sol}). Это решение отвечает круговому движению ``масштабированной планеты'' с угловой скоростью $\fomega_1$ и сидерическому движению ``спутника'' со средней уговой частотой $\Omega$. Поэтому искомое невозмущённое (сидерическое) движение ``спутника'' для задачи Солнце-Земля-Луна можно положить равным
 $$ 
(\uy_{11}^0(t),\ueeta_{11}^0(t)):=e^{i\fomega_1t}
m_1^{1/3}\left(\widehat\hy_{\Omega_{11},\fomega_1}(t),m_{11}\widehat\heeta_{\Omega_{11},\fomega_1}(t)\right),  \
 \frac{\fomega_1}{\Omega_{11}-\fomega_1} \in (-\somega_0,\somega_0).
 $$

\subsection {Constructing families of relatively-periodic solutions of the unperturbed system}
\label {subsec:Hill}

Рассмотрим невозмущённую систему (\ref {eq:unpert0}), (\ref {eq:unpert1}) для задачи $N+1$ тел с любым числом планет и спутников, $N\ge n\ge1$. Положим 
 $$
(\uy_{\ell j}^0(t),\ueeta_{\ell j}^0(t)):=e^{i\fomega_\ell t}
m_\ell^{1/3}(\widehat\hy_{\Omega_{\ell j},\fomega_\ell}(t),m_{\ell j}\widehat\heeta_{\Omega_{\ell j},\fomega_\ell}(t)), 
 \quad 
 \frac{\fomega_\ell}{\Omega_{\ell j}-\fomega_\ell} \in (-\somega_0,\somega_0),
 $$
при $1\le j\le n_\ell$, 
см.\ (\ref {eq:my:sol'}), (\ref {eq:my:sol}). Построенное нами решение $(\ux^0(t),\uxxi^0(t))=(\ux_{*}^0(t),\uy_{**}^0(t),\uxxi_{*}^0(t),\ueeta_{**}^0(t))$ является искомым решением невозмущённой системы (\ref {eq:unpert0}), (\ref {eq:unpert1}), $O(\somega)$--близ\-ким к порождающему решению $(\mx^0(t),\mxxi^0(t))$. Более того, как выведено в замечании \ref {rem:Hill} из леммы \ref {lem:mult}(B), решение $(\ux^0(t),\uxxi^0(t))$ даже $O(\somega^2)$--близко к порождающему решению $(\mx^0(t),\mxxi^0(t))$.

\section {Constructing relatively-periodic solutions of the perturbed system
(\ref {eq:pert:planets})} \label {sec:pert}

В этом параграфе мы начнём вывод теоремы \ref {th:ust}(A) из леммы \ref {lem:mult}.

Как и в \S\ref {par3:1:5}, фиксируем числа $m_i,m_{ij}>0$, $\somega\in(0,\somega_0]$ и набор ``угловых частот'' $\fomega_i,\Omega_{ij}\in\bbR\setminus\{0\}$ вида (\ref {chast}), (\ref {poryadki:chastot'}). Предположим, что набор частот является ``относительно резонансным'', т.е.\ имеет вид (\ref {maxrez}), (\ref {parametry}) (это условие автоматически выполнено при $N=2$, т.е.\ в задаче 3 тел).

Пусть $(\mx^0(t),\mxxi^0(t))=(\mx_*^0(t),\my_{**}^0(t),\mxxi_*^0(t),\meeta_{**}^0(t))$ --- основное порождающее решение (\ref {eq:gen:sol}) модельной системы (\ref {eq:model:H:omega}), отвечающее этому набору частот. Пусть $(\ux^0(t),\uxxi^0(t))=(\ux_*^0(t),\uy_{**}^0(t),\uxxi_*^0(t),\ueeta_{**}^0(t))$ --- построенное в \S \ref {par3:1:5} $(T,\alpha)$-пери\-о\-ди\-ческое решение невозмущённой системы (\ref {eq:unpert0}), (\ref {eq:unpert1}), $O(\somega)$--близкое к решению $(\mx^0(t),\mxxi^0(t))$ и такое, что $(\ux_*^0(t),\uxxi_*^0(t))=(\mx_*^0(t),\mxxi_*^0(t))$ и радиус-векторы $\uy_{ij}^0(0)$ сонаправлены с осью абсцис. Рассмотрим тор $\Lambda$, образованный фазовыми траекториями $(T,\alpha)$-периодических решений $(\ux^\beta(t),\uxxi^\beta(t))$ невозмущённой системы, $\beta=(\beta_*,\beta_{**})\in(\bbR/2\pi\bbZ)^N$ (см.\ (\ref {eq:unpert:sol})).
В \S \ref {par3:1:5} мы доказали выполнение всех условий леммы \ref {lem:mult}
для данной невозмущённой системы. Поэтому применима лемма \ref {lem:mult},
и из неё легко следует, что тор $\Lambda$ обладает свойствами, указанными в теореме \ref {th:mult}.

Предположим, что набор частот удовлетворяет условию невырожденности (\ref {alpha}) или более тонкому условию (\ref {eq:fine}) из теоремы \ref {th1}, где в более тонком условии (\ref {eq:fine}) определим константу $C>0$ равной константе $C_2>0$, отвечающей тору $\Lambda$ согласно теореме \ref {th:mult}. 
Если выполнено условие невырожденности (\ref {alpha}), то при $0<\somega\le3c/(4C)$ выполнено (ввиду (\ref {poryadki:chastot'})) более тонкое условие (\ref {eq:fine}). Поэтому можно считать, что выполнено более тонкое условие (\ref {eq:fine}).

Изучим возмущённую систему (\ref {eq:pert:planets}) вблизи тора $\Lambda$. При фиксированных $m_i,m_{ij},\somega,\fomega_i,\Omega_{ij}$ система зависит от $4$ параметров $\varepsilon,\mu,\nu,\rho\in\bbR$, а тор $\Lambda$ фиксирован. Мы хотим вывести существование (при фиксированных $m_i,m_{ij},\somega,\fomega_i,\Omega_{ij}$) числа $\mu_0>0$ и гладкого $4$-параметрического семейства торов $\widetilde\Lambda$ с параметрами $\varepsilon,\mu,\nu,\rho$, $|\varepsilon|+|\mu|+|\nu|+|\rho|\le\mu_0$, обладающего нужными нам свойствами.

Проведём вывод аналогично выводу теоремы \ref {th1'} из технической теоремы \ref {th:mult}, методом производящей функции (см.\ \S\ref {subsec:ideas}, второй этап, шаги 1--5).

{\it Шаг 1.} Согласно методу производящей функции, нужно проверить, что тор $\Lambda$ обладает свойством невырожденности (\ref {eq:nondegL}). Это свойство мы вывели (в \S\ref {subsec:ideas}, шаг 1) из теоремы \ref {th:mult} и более тонкого условия (\ref {eq:nondeg'}) из теоремы \ref {th1'}.

Проверим условие (\ref {eq:nondeg'}). Его первая часть равносильна первой части (\ref {eq:fine}). Для проверки второй части (\ref {eq:nondeg'}) найдём число $\Delta_{ij}$ в (\ref {eq:Delta:ij}). В силу (\ref {f0}) имеем $F_{ij}=F_{ij}(\px_i,\py_{ij})=m_{ij}F(\px_i,\py_{ij})$. 
По построению $\widehat F_{ij}\o=(F_{ij}(\px_i\o,\cdot))|_{H_{ij}^{-1}(H_{ij}(I_{ij}\o,0,0))}$,
где $\px_i\o=\const$, $|\px_i\o|=R_i/R=|\Omega_{i0}|^{-2/3}$ (см.\ теорему
\ref {th1'}). Пусть $\langle\widehat F_{ij}\o\rangle=\langle\widehat F_{ij}\o\rangle(q_{ij},p_{ij})$ --- функция, полученная усреднением потенциала Хилла $\widehat F_{ij}\o=\widehat F_{ij}\o(q_{ij},p_{ij})$ по $\frac{2\pi}{\Omega_{ij}}$-периодическим решениям задачи Кеплера $(M_{ij}=S^1\times(\bbR\setminus\{0\})
\times\bbR^2,\oomega_{ij}=dI_{ij}\wedge d\varphi_{ij}+dp_{ij}\wedge dq_{ij},H_{ij}=S^{(i)}_j)$ для ``спутника'', где $\varphi_{ij},I_{ij},q_{ij},p_{ij}$ --- ``нормализующие'' канонические координаты для рассматриваемой задачи Кеплера (см.\ лемму \ref {lem3:2:}). Согласно замечанию \ref {rem:Hill}, $d\langle\widehat F_{ij}\o\rangle(0,0)=0$. Не\-слож\-ным вычислением
находим также матрицу Гесса функции $\langle\widehat F_{ij}\o\rangle$ в нуле:
 $$
   \frac{\partial^2\langle\widehat F_{ij}\o\rangle(0,0)}{\partial(q_{ij},p_{ij})^2}
   = \Omega_{i0}^2 \frac{I_{ij}}{\Omega_{ij}}
 \left( \begin{array}{cc} -29/8 & 0 \\
                              0 & 25/(8I_{ij}^2) \end{array}\right).
 $$
Отсюда с учетом (\ref {d2H}) имеем
 $$
 \Delta_{ij}
 = \frac{\Omega_{ij}}2 \Tr\left(\left
 (\frac{\partial^2H_{ij}(I_{ij}\o,0,0)}{\partial(q_{ij},p_{ij})^2}\right)^{-1}
 \frac{\partial^2\langle\widehat F_{ij}\o\rangle(0,0)}{\partial(q_{ij},p_{ij})^2}\right)
 =
 $$
 \begin{equation} \label {eq:Delta}
 = \frac{\Omega_{i0}^2}{2\Omega_{ij}}
 \Tr\left(
 \left( \begin{array}{cc} 1 & 0 \\
                          0 & I_{ij}^2 \end{array}\right)
 \left( \begin{array}{cc} -29/8 & 0 \\
                              0 & 25/(8I_{ij}^2) \end{array}\right)
 \right)
 = - \frac{\Omega_{i0}^2}{4\Omega_{ij}}.
 \end{equation}
Поэтому требуемая вторая часть условия (\ref {eq:nondeg'}) имеет вид
 $$
 \alpha+\Delta_{ij}\somega^2T = \alpha-\frac{\fomega_i^2}{4\Omega_{ij}}T \not\in\left[-C_2\somega^3T,C_2\somega^3T\right]+2\pi\bbZ,
 $$
т.е.\ равносильна второй части условия (\ref {eq:fine}) с константой $C:=C_2$.

Итак, выполнено более тонкое условие невырожденности (\ref {eq:nondeg'}) из теоремы \ref {th1'}. Отсюда следует (см.\ выше) свойство невырожденности (\ref {eq:nondegL}) тора $\Lambda$.

{\it Шаги 2--5.} Из явного построения функций $\widetilde H_0,\widetilde H_1,\widetilde\Phi$ (см.\ \S\ref {par3:1:4}) видно, что они являются $S^1$-инвариантными. Возмущающий потенциал $\widetilde\Phi$ является аналитической функцией в окрестности любого тора $\Lambda\o$, на котором набор угловых частот $\Omega_{ij}$ удовлетворяет условиям (\ref {poryadki:chastot'}) и (\ref {poryadki:chastot''}) ``отсутствия столкновений''. Главная часть $\Phi:=\widetilde\Phi|_{\mu=\nu=\rho=0}$ возмущающего потенциала имеет вид (\ref {eq:Phi}). 

Согласно методу производящей функции (шаги 2--5), ввиду невырожденности (\ref {eq:nondegL}) тора $\Lambda$, искомое семейство торов $\widetilde\Lambda$ можно определить формулой (\ref {eq:tildeL}). Поэтому из леммы \ref {lem:mult}
и результатов следующего \S\ref{par3:1:4} вытекает теорема \ref {th:ust}(A) о существовании семейства торов $\widetilde\Lambda$ с нужными свойствами.
 \qed

Попутно мы показали, что теорему \ref {th:ust}(A) можно вывести также из технической теоремы \ref {th:mult} (или из теоремы \ref {th1'}) и результатов следующего \S\ref{par3:1:4}.

Отметим, что в частном случае $N=2$, $n=1$ (система Солнце-Земля-Луна) наш результат показывает, что начальные участки (\ref {eq:sol:Hill}) семейств $f$ и $g_+$ периодических решений задачи Хилла порождают начальные участки семейств периодических решений ограниченной задач трёх тел (случай $\nu=0$) и семейств относительно-периодических решений задачи трёх тел. Аналогичные участки семейств решений ограниченной задачи трёх тел (даже без ограничений на параметр $\widetilde\mu=\frac{\mu m_1}{1+\mu m_1}\in(0,1)$) были найдены Брауном \cite {brown1892} методом Хилла \cite {H,W} разложения решений в ряд (см.\ также \cite{32}, \cite[\S 17--19]{33}, \cite {Perko83}). Аналогичные участки семейств решений задачи трёх тел (даже без ограничений на параметры $\widetilde\mu=\frac{\mu(m_1+\nu m_{11})}{1+\mu(m_1+\nu m_{11})},\theta=\frac{\nu m_{11}}{m_1+\nu m_{11}}\in[0,1]$) были найдены Мультоном \cite {32} методом малого параметра Пуанкаре (см.\ также \cite[\S 18--19]{33}).

\section{Reducing the $N+1$ body problem to the perturbed system (\ref {eq:pert:planets})
} \label {par3:1:4}

Рассмотрим на $(\bbR^2\setminus\{0\})\times\bbR^2$ 
гладкую функцию $F=F(\hx,\hy)$, определяемую формулой (\ref {F:xy}), см.\ (\ref {eq:unpert1}).
Напомним (см.\ (\ref {eq:Hill})), что функцию $F(\hx,\hy)|_{\hx=(1,0)}=-\hyi_1^2+\hyi_2^2/2$, а также функцию $F(\hx,\hy)$, мы называем {\it потенциалом Хилла} (или ``предельным потенциалом действия Солнца на спутник'').


\begin{scRem} (A) The Hill potential $F(\x,\y)$ is in fact the third
coefficient of the power series of the function $-\frac1{|\x+\rho\y|}$
in the variable $\rho$ at zero:
\begin{equation} \label {rho:to0}
 \frac{1}{|\x+\rho\y|}=\frac1{\sqrt{\x^2+2\rho\langle\x,\y\rangle+\rho^2\y^2}}=:
 \frac{1}{|\x|}-\rho\frac{\langle\x,\y\rangle}{|\x|^3} -
 \rho^2 F_{0,\rho}(\x,\y),
\end{equation}
 $$
F_{0,\rho}(\x,\y) = F(\x,\y)-
 \rho\langle\x,\y\rangle\frac{3\x^2\y^2-5\langle\x,\y\rangle^2}{2|\x|^7}-
 $$
 \begin{equation} \label {tildeF0} -\rho^2\frac{3\x^4\y^4-30\x^2\langle\x,\y\rangle^2\y^2+35\langle\x,\y\rangle^4}{8|\x|^9}+
 \dots, \quad \rho\to 0.
 \end{equation}
A more general analytic potential
 \begin{equation} \label {eq:F:theta:rho}
F_{\theta,\rho}(\x,\y)=\theta F_{0,-\theta\rho}(\x,\y)+(1-\theta)F_{0,(1-\theta)\rho}(\x,\y)
 \end{equation}
appears in the three-body problem (with $0<\theta<1$,
$\mu,\omega,\rho>0$ and (\ref{nov:*})) and in the {\it restricted
three-body problem} (with $\theta=0$, $\mu,\omega,\rho>0$ and (\ref
{nov:*})), see~\Sec\ref {par3:1:4''} and (\ref {tildeF}). We remark
that
 $\frac{\partial F}{\partial\y}(\x,\y)=\frac{\x^2\y-3\langle\x,\y\rangle\x}{|\x|^5}$.

(B) The unperturbed system (\ref {3:7}) shows that the variables
$\x_i$ and $\y_{ij}$ are automatically slow and fast variables
respectively, provided that $\omega$ is small.
 \end{scRem}

Let $\px_i$, $\py_{ij}$, $\pxxi_i$, $\peeta_{ij}$ be the coordinates (\ref {eq:zamena}) in the phase space $T^*\widehat Q$, $1\le i\le n$, $1\le j\le n_i$. Denote
\begin{equation} \label {m}
 \bar m_i=m_i+\nu\sum^{n_i}_{j=1}m_{ij}, \quad
 \widetilde m_i=\frac{\bar m_i}{1+\mu\bar m_i}, \quad
 \bar m_{ij}=\frac{m_{ij}}{m_i}, \quad
 \widetilde m_{ij}=\frac{m_{ij}m_i}{m_i+\nu m_{ij}}
\end{equation}
the total mass of the $i$th satellite system,
and the ``reduced'' masses of planets and satellites. Introduce the
following functions on $T^*\widehat Q$: the Hamilton functions
 \begin{equation} \label {KS}
 \widetilde K_i=\frac{\xxi^2_i}{2\widetilde m_i}-\frac{\bar m_i}{|\x_i|}, \qquad
 \widetilde S^{(i)}_j=
 \frac{\eeta^2_{ij}}{2\widetilde m_{ij}}-\frac{m_im_{ij}}{|{\y}_{ij}|}
 \end{equation}
of the Kepler problems, the angular momenta
 \begin{equation} \label{eq:I}
I_{i0}=[\x_i,\xxi_i], \qquad I_{ij}=[\y_{ij},\eeta_{ij}],
 \qquad 1\le i\le n, \ \ 1\le j\le n_i
 \end{equation}
of ``scaled planets'' and ``satellites'', and the ``perturbation functions''
 \begin{equation} \label {tilde:KS}
 K_{ii'}=\langle\xxi_i,\xxi_{i'}\rangle-
 \frac{\bar m_i\bar m_{i'}}{|\x_i-\x_{i'}|},
 \qquad
 S^{(i)}_{jj'}=\frac{\langle\eeta_{ij},\eeta_{ij'}\rangle}{m_i}-
 \frac{m_{ij}m_{ij'}}{|\y_{ij}-\y_{ij'}|},
 \end{equation}
$1\le i<i'\le n$, $1\le j<j'\le n_i$, of the planetary system and
the satellite systems, respectively (corresponding to pair-wise
interactions of planets, respectively satellites of the same planet).

As a ``perturbation potential'', let us consider the function
 \begin{equation} \label {G}
\widetilde\Phi=\widetilde\Phi(\px_*,\py_{**},\bar m_*,\bar m_{**},\mu,\nu,\rho):=\sum_{i=1}^n\bar m_i\Phi_i+
 \mu\sum_{1\le i<i'\le n}
 \bar m_i\bar m_{i'}\Phi_{ii'}
 \end{equation}
in the configuration variables $\px_i,\py_{ij}\in\bbR^2$ and the parameters
$\bar m_i,\bar m_{ij},\mu,\nu,\rho\in\bbR$. Here the functions
$\Phi_i=\Phi_i(\px_i,\py_{i*},\bar m_{i*},\nu,\rho)$ and 
$\Phi_{ii'}=\Phi_{ii'}(\px_i-\px_{i'},\py_{i*},\py_{i'*},
 \bar m_{i*},\bar m_{i'*},\nu,\rho)$ are defined by the formulae
 \begin{equation} \label {tildeFi}
 \Phi_i
 :=
 \frac{1}{\nu\rho^2}\left(\frac{1}{|\px_i|} - \frac{m_i/\bar m_i}{|\px_i-\nu\rho\ddelta_i|}
 - \nu\sum_{j=1}^{n_i} \frac{m_{ij}/\bar m_i}{|\px_i+\rho\py_{ij}-\rho\nu\ddelta_i|}\right),
 \end{equation}
 $$
\Phi_{ii'}
 :=
 \frac{1}{\nu\rho^2\bar m_i\bar m_{i'}} 
 \left(
 \frac{\bar m_i\bar m_{i'}}{|\px_i-\px_{i'}|}
 - \sum_{j=1}^{n_i}\frac{\nu m_{i'}m_{ij}}{
            |\px_i-\px_{i'}+\rho\py_{ij}-\rho\nu(\ddelta_i-\ddelta_{i'})|}
 -
 \right.
 $$
 $$
 - \frac{m_im_{i'}}{|\px_i-\px_{i'}-\nu\rho(\ddelta_i-\ddelta_{i'})|}
 - \sum_{j'=1}^{n_{i'}}\frac{\nu m_im_{i'j'}}{
         |-\px_i+\px_{i'}+\rho\py_{i'j'}+\rho\nu(\ddelta_i-\ddelta_{i'})|}
 -
 $$
 \begin{equation} \label {tildeFii}
 \left.
 - \sum_{j=1}^{n_i}\sum_{j'=1}^{n_{i'}}\frac{\nu^2m_{ij}m_{i'j'}}{
  |\px_i-\px_{i'}+\rho(\py_{ij}-\py_{i'j'})-\rho\nu(\ddelta_i-\ddelta_{i'})|}
 \right).
 \end{equation}
Here 
 \begin{equation} \label {delta:i}
\ddelta_i:=\sum_{j=1}^{n_i}\frac{m_{ij}}{\bar m_i}\py_{ij} =
\sum_{j=1}^{n_i}\bar m_{ij}\py_{ij}/(1+\nu\sum_{j=1}^{n_i}\bar m_{ij})
 \end{equation}
is the radius vector drawn from a planet to the centre of masses of the system of
its satellites, multiplied by $(\sum_{j=1}^{n_i}m_{ij})/\bar
m_i$. One easily shows (see (\ref {rho:to0})) that the function
$\Phi_i$ is analytic in all its variables in the region
\begin{equation} \label {oblast}
\left\{ 1+\nu\sum_{j'=1}^{n_i}\bar m_{ij'}\ne0,
 \quad |\rho|\left(1+|\nu|\sum_{j'=1}^{n_i}|\bar m_{ij'}|\right)|\y_{ij}|<|\x_i|
\right\}_{j=1}^{n_i},
\end{equation}
while the function $\Phi_{ii'}$ is analytic in all its variables in the domain
\begin{equation} \label {oblast}
\left\{ 1+\nu\sum_{j'=1}^{n_i}\bar m_{ij'}\ne0,
 \quad |\rho|\left(1+|\nu|\sum_{j'=1}^{n_i}|\bar m_{ij'}|\right)|\py_{ij}|<|\px_i|
\right\}_{j=1}^{n_i},
\end{equation}
while the function $\Phi_{ii'}$ is analytic in all its variables in the domain
\begin{equation} \label {oblast'}
\left\{
 \begin{array}{c}
  1+\nu\sum_{j=1}^{n_i}\bar m_{ij}\ne0, \quad
  1+\nu\sum_{j'=1}^{n_{i'}}\bar m_{i'j'}\ne0, \quad
     \\
 |\rho|(1+|\nu|\sum_{j''=1}^{n_i}|\bar m_{ij''}|)|\py_{ij}|<\frac{|\px_i|}{2},
     \quad 1\le j\le n_i, \\
 |\rho|(1+|\nu|\sum_{j''=1}^{n_{i'}}|\bar m_{i'j''}|)|\py_{i'j'}|<\frac{|\px_{i'}|}{2},
  \quad 1\le j'\le n_{i'}
 \end{array}
\right\}.
\end{equation}
The functions $\Phi_{ii'}$ are expressed in terms of $\Phi_1,\dots,\Phi_n$ as
follows:
 $$
\Phi_{ii'}=\frac{m_{i'}}{\bar m_{i'}}\Phi_i(\px_i-\px_{i'}+\nu\rho\ddelta_{i'},\py_{i*},\bar m_{i*},\nu,\rho)
  +\Phi_{i'}(\px_{i'}-\px_i,\py_{i'*},\bar m_{i'*},\nu,\rho) +
 $$
 \begin{equation} \label {eq:Gii} 
 +\nu\sum_{j'=1}^{n_{i'}}\frac{m_{i'j'}}{\bar m_{i'}}\Phi_i(\px_i-\px_{i'}-\rho\py_{i'j'}+\nu\rho\ddelta_{i'},\py_{i*},\bar m_{i*},\nu,\rho).
 \end{equation}
We set $\Phi_i:=0$ if $n_i=0$ (i.e.\ the $i$th planet has no
satellites), and we set $\Phi_{ii'}:=0$ if $n_i=n_{i'}=0$. If $n_i=1$
(i.e.\ the $i$th planet is a double planet, $\theta_i:=\nu
m_{i1}/(m_i+\nu m_{i1})$) then we have
$\Phi_i=\frac{\theta_i(1-\theta_i)}\nu F_{\theta_i,\rho}(\px_i,\py_{i1})$ and
 $$
 \Phi_{ii'}
 =
 \Phi_{i'}(\px_{i'}-\px_i,\py_{i'*},\bar m_{i'*},\nu,\rho)
 + \frac{\theta_i(1-\theta_i)}{\nu{\bar m_{i'}}/{m_{i'}}} F_{\theta_i,\rho}(\px_i-\px_{i'}+\nu\rho\ddelta_{i'},\py_{i1} ) +
 $$
 $$
 +{\theta_i(1-\theta_i)}\sum_{j'=1}^{n_{i'}}\frac{m_{i'j'}}{\bar m_{i'}} F_{\theta_i,\rho}(\px_i-\px_{i'}-\rho\py_{i'j'}+\nu\rho\ddelta_{i'},\py_{i1} ).
$$

 \begin{Lem}[(equivalence of the $N+1$ body problem to an $\varepsilon$-Hamiltonian system)] \label {lem3:1}
Let $\widehat Q$ be the $2N$-dimensional vector space formed by all
configurations of $N+1$ particles with masses {\rm(\ref {A})} and the
centre of masses at the origin in a Euclidean plane. Define linear
coordinates on $\widehat Q$ to be the collection of radius vectors
$\x_i,\y_{ij}:\widehat Q\to\bbR^2$, $1\le i\le n$, $1\le j\le n_i$ {\rm(see
(\ref {a}), (\ref {b}))}. There exists a collection of linear
functions $\xxi_i,\eeta_{ij}:\widehat Q^*\to\bbR^2$, $1\le i\le n$, $1\le j\le
n_i$ \(momenta\) having the following properties. In the
coordinates $\x_i,\y_{ij},\xxi_i,\eeta_{ij}$ on $T^*\widehat Q\cong \widehat Q\times
\widehat Q^*$, the canonical symplectic structure $\oomega=d\p\wedge d\q$, the
Hamiltonian function $H$ of the $N+1$ body problem and the first
integral $I$ of angular momentum {\rm(see (\ref {H}) and (\ref {M}))}
have the form
 \begin{equation}\label {eq:HM}
 \oomega=\frac\rho{\omega}\widetilde\oomega, \qquad
 H=\frac\rho{\omega}\widetilde H, \qquad
 I=\frac\rho{\omega}\widetilde I
 \end{equation}
where
 \begin{equation} \label {vozmHM}
 \widetilde\oomega=\oomega_0+\varepsilon\oomega_1, \qquad
 \widetilde H=\omega\widetilde H_0+\varepsilon\widetilde H_1+\omega^2\varepsilon\widetilde\Phi, \qquad
 \widetilde I=I_0+\varepsilon I_1,
 \end{equation}
$$
 \oomega_0=d\xxi\wedge d\x=\sum_{i=1}^nd\xxi_i\wedge d\x_i, \quad
 \oomega_1=d\eeta\wedge d\y=\sum_{i=1}^n\sum_{j=1}^{n_i}d\eeta_{ij}\wedge d\y_{ij},
$$
 \begin{equation} \label {H0}
 \widetilde H_0= \sum^n_{i=1}\widetilde K_i+ \mu\sum_{1\le i<i'\le n}K_{ii'}, \quad
 \widetilde H_1= \sum^n_{i=1}\left(\sum^{n_i}_{j=1}\widetilde S^{(i)}_j+ \nu\sum_{1\le j<j'\le n_i}S^{(i)}_{jj'}\right),
 \end{equation}
 \begin{equation} \label {M0}
 I_0=[\x,\xxi]=\sum_{i=1}^nI_{i0}, \qquad
 I_1=[\y,\eeta]=\sum_{i=1}^n\sum_{j=1}^{n_i}I_{ij},
 \end{equation}
{\rm see (\ref {KS}), (\ref {eq:I}), (\ref {tilde:KS})}. Here the
small parameters $0<\omega,\varepsilon,\mu,\nu,\rho\ll1$ are related
by the conditions $\rho=\omega^{2/3}\mu^{1/3}$ and
$\varepsilon=\omega^{1/3}\mu^{2/3}\nu=\nu\rho^2/\omega$,
$\rho=\frac1R$. The ``perturbation potential''
$\widetilde\Phi=\widetilde\Phi(\x_*,\y_{**},\bar m_*,\bar
m_{**},\mu,\nu,\rho)$ has the form {\rm(\ref {G})}, is an analytic
function on the direct product of the regions
 $$
\left\{
 \begin{array}{c}
  1+\nu\sum_{j'=1}^{n_i}\bar m_{ij'}\ne0,
  \\
  |\rho|\left(1+|\nu|\sum_{j'=1}^{n_i}|\bar m_{ij'}|\right)|\y_{ij}|<
  \min\!\left(|\x_i|,\frac{1}{2}\min_{i'\ne i}|\x_i-\x_{i'}|\right)
 \end{array}
\right\}_{j=1}^{n_i},
 $$
$1\le i\le n$, and satisfies the condition
 \begin{equation} \label {f0}
\widetilde\Phi|_{\nu=\rho=0}=\sum^n_{i=1} \sum^{n_i}_{j=1} m_{ij}
 \left(
 F(\x_i,\y_{ij})+\mu\sum_{\substack{i'=1\\i'\ne i}}^n\bar m_{i'}F(\x_i-\x_{i'},\y_{ij})
 \right),
 \end{equation}
{\rm see (\ref {F:xy})}. In particular, $\widetilde
H_1=\widetilde\Phi=0$ in the case $n_i=0$ of a planetary system
without satellites.
  \end {Lem}

\begin{scRem} \label {rem:unpert}
Lemma~\ref {lem3:1} implies equivalences of the following
($\varepsilon$-)Hamiltonian systems for
$\omega,\varepsilon,\mu,\nu,\rho>0$:
$$
(T^*\widehat Q,\oomega,H)\cong(T^*\widehat Q,\widetilde\oomega,\widetilde H)
 \cong(T^*\widehat Q,T^*\widehat Q_0,p;\oomega_0,\oomega_1;\widetilde H,\widetilde H_1+\omega^2\widetilde\Phi)^\varepsilon
$$
where $\widehat Q_0$ is the configuration space of planets, and $p:T^*\widehat Q\to
T^*\widehat Q_0$ is the projection. The third of these systems, called {\it
the unperturbed system}, is not only Hamiltonian, but also
$\varepsilon$-Hamiltonian (see (\ref {eq:lam:Ham})). Hence it
naturally extends to any {\it nonnegative} values
$\omega,\varepsilon,\mu,\nu,\rho\ge0$ of small parameters (despite of
the fact that the symplectic structure degenerates if one of the
parameters vanishes). For the limiting values of the parameters
$\omega>0$ and $\mu=\nu=0$ (and, hence, $\varepsilon=\rho=0$), the
third system becomes a 0-Hamiltonian system
$(T^*\widehat Q,T^*\widehat Q_0,p;\oomega_0,\oomega_1;\omega H_0,H_1+\omega^2\Phi)^0$
called {\em unperturbed} where $H_0:=\widetilde H_0|_{\mu=0}$,
$H_1:=\widetilde H_1|_{\nu=0}$,
$\Phi:=\widetilde\Phi|_{\mu=\nu=\rho=0}$. The system
$(T^*\widehat Q,T^*\widehat Q_0,p;\oomega_0,\oomega_1;\omega H_0,H_1)^0$, called {\it
the model system}, is $\omega^2$-close to the unperturbed system. It
follows from lemma \ref {lem3:1} that the unperturbed system indeed
has the form (\ref {3:7}) in the configuration space $\widehat Q$.
\end{scRem}

Due to (\ref {rho:to0}) and (\ref {eq:Gii}), the functions $\Phi_i$,
$\Phi_{ii'}$ in (\ref {tildeFi}), (\ref {tildeFii}) have the
following form for $|\nu|\le\nu_0$, $|\rho|\le\rho_0$:
 $$
 \Phi_i =\nu\frac{m_i}{\bar m_i}F_{0,\nu\rho}(\x_i,\ddelta_i)+\sum_{j=1}^{n_i}\frac{m_{ij}}{\bar m_i}F_{0,\rho}(\x_i,\y_{ij}-\nu\ddelta_i)=
 $$
 \begin{equation} \label {d}
 =\sum^{n_i}_{j=1}
 \frac{m_{ij}}{\bar m_i}F_{0,\rho}(\x_i,\y_{ij})-\nu F(\x_i,\ddelta_i)+O(\nu\rho),
 \end{equation}
 $$
\Phi_{ii'}
 =\sum^{n_i}_{j=1}\frac{m_{ij}}{\bar m_i}F_{0,\rho}(\x_i-\x_{i'},\y_{ij})
 + \sum^{n_i'}_{j'=1}
 \frac{m_{i'j'}}{\bar m_{i'}}F_{0,\rho}(\x_{i'}-\x_i,\y_{i'j'}) -
 $$
 $$
 -\nu F(\x_i-\x_{i'},\ddelta_i)
 -\nu F(\x_{i'}-\x_i,\ddelta_{i'})
 + O(\nu\rho).
 $$
Hence the ``perturbation potential'' $\widetilde\Phi$ in (\ref {G})
satisfies the condition
 \begin{equation} \label {f}
\widetilde\Phi=\sum^n_{i=1} \sum^{n_i}_{j=1} m_{ij}
 \left(
       F_{0,\rho}(\x_i,\y_{ij})+
       \mu\sum_{\substack{i'=1\\i'\ne i}}^n\bar m_{i'}F_{0,\rho}(\x_i-\x_{i'},\y_{ij})
 \right)-
 \end{equation}
 $$
 -\nu\sum_{i=1}^n
 \left(\bar m_i F(\x_i,\ddelta_i) + \mu\sum_{\substack{i'=1\\i'\ne i}}^n\bar m_i F(\x_i-\x_{i'},\ddelta_i)\right)
 +O(\nu\rho), \quad \nu,\rho\to0,
 $$
which implies~(\ref {f0}), see~(\ref {tildeF0}).

Let us explain a {\it geometric meaning} of the Hill potential $F(\x,\y)$
when the $i$th planet has $n_i>1$ satellites ($1\le i\le n$).
Consider the function $\Phi_i=\Phi_i(\x_i,\y_{i*},\bar
m_{i*},\nu,\rho)$ defined by the formula (\ref {tildeFi}) and called
the ``potential of interaction of all satellites of the $i$th planet
with the Sun''. Due to (\ref {d}), the function $\Phi_i|_{\rho=0}$
equals a linear combination of the functions $F(\x_i,\y_{ij})$, $1\le
j\le n_i$, and $F(\x_i,\ddelta_i)$.

\subsection {The Poincar\'e transformation in the $n+1$ body problem}

In order to prove lemma \ref {lem3:1}, we will explicitly construct
the variables of momenta $\xxi_i$, $\eeta_{ij}$ and will show that the
function $H$ in (\ref {eq:HM}), (\ref {vozmHM}), (\ref {H0}) equals
the total energy of the system. One easily shows that those summands
in $H$ that do not depend on the momenta give the potential energy
$U$.

Let us compute the kinetic energy $G$. We will explore the fact that
the transition from the coordinates $(\M_0,\M_1,\ldots,\M_N)$ in the
configuration space to the coordinates $\x_i,\y_{ij}$ (see \Sec\ref
{par:param}) can be done by applying twice the following
transformation called the Poincar\'e transformation.

Let us consider the configuration manifold $Q$ of a planetary system
(i.e.\ the system of $n+1$ particles in a Euclidean plane). It
consists of all ordered collections of radius vectors
$\M_0,\M_1,\ldots,\M_n$ with associated masses $c_0=1$, $c_1=\lambda
m_1,\ldots,c_n=\lambda m_n$ where $0<\lambda\ll 1$. Thus the manifold
$Q$ is naturally identified with the vector space $\bbR^{2(n+1)}$
with coordinates $\M_0,\M_1,\ldots,\M_n$. Consider the linear
transformation $L=L_{c_0,c_1,\dots,c_n}$ in $\bbR^{2(n+1)}$ that
corresponds to introducing the new linear coordinates
on the space $Q$ corresponding to
the following collection of radius vectors:
 $$
\widetilde
\M_0=\frac{\M_0+c_1\M_1+\ldots+c_n\M_n}{1+c_1+\ldots+c_n},\quad
\widetilde \M_1=\M_1-\M_0,\ \ldots ,\ \widetilde \M_n=\M_n-\M_0
 $$
where $\widetilde \M_0=\C :=
\frac{\M_0+c_1\M_1+\ldots+c_n\M_n}{1+c_1+\ldots+c_n}$ is the radius
vector of the centre of masses of the system.

\begin{scDef} \label {definit2}
The transformation $L=L_{c_0,c_1,\dots,c_n}$ is called {\it the
Poincar\'e trans\-form\-ation} on the configuration manifold of the
planetary system.
 \end{scDef}

Actually one could consider another transformation, namely {\it the
Jacobi transformation} $\widetilde \M_0=\C$, $\widetilde
\M_1=\M_1-\C,\ldots ,\widetilde \M_n=\M_n-\C$. But this
transformation would lead to more awkward formulae. Moreover it would
bring us to a desired result only in the case of a usual planetary
system, i.e.\ having no satellites.

Consider the dual space $Q^*$, i.e.\ the space of all linear
functions on the space $Q$ (or, equivalently, the cotangent space to
$Q$ at its any point). This space consists of all collections
$\p_0,\p_1,\ldots,\p_n$ whose each item $\p_i$ is a linear function
on the plane, i.e.\ a co-vector. In fact, we can define the value of
the linear function corresponding to such a collection on the
configuration $(\M_0,\M_1,\ldots,\M_n)\in Q$ to be
$\sum^n_{i=0}\langle \p_i,\M_i\rangle$. It is clear that the
nondegenerate transformation $L$ on $\bbR^{2(n+1)}\cong Q$ induces a
linear transformation $L^*$ on $(\bbR^{2(n+1)})^*\cong Q^*$. Denote
the image of the collection $\p_0,\p_1,\ldots,\p_n$ under the
transformation $L^*$ by $\widetilde \p_0,\widetilde
\p_1,\ldots,\widetilde \p_n$.

Consider the real valued function
$\Kin=\sum^n_{i=0}\frac{\p^2_i}{2c_i}$ of kinetic energy on the space
$Q^*$. Besides we consider the function $I=\sum_{i=0}^n[\M_i,\p_i]$
of angular momentum on the space $T^*Q$. Finally consider the
function of the total momentum $\bP=\p_0+\p_1+\ldots +\p_n$ on $Q^*$
whose values belong to the space of co-vectors, i.e.\ of linear
functions on the plane.

\begin {Lem}
Under the Poincar\'e transformation $L=L_{c_0=1,c_1,\dots,c_n}$ on
the configuration space $Q$ of the $n+1$ body problem, the functions
$\Kin$ and $\bP$ on $Q^*$ transform as follows:

{\rm(A)} The kinetic energy $\Kin=\sum^n_{i=0}\frac{\p^2_i}{2c_i}$
has the form
\begin{equation} \label {4:9}
 \Kin=\frac{\widetilde\p^2_0}{2\bar c_0}+\sum^n_{i=1}\frac{\widetilde\p^2_i}{2c_i}+
 \frac{1}{2}(\widetilde\p_1+\ldots+\widetilde\p_n)^2
\end{equation}
where $\bar c_0=1+c_1+\ldots +c_n=1+\lambda(m_1+\ldots +m_n)$. The
expression {\rm(\ref {4:9})} can be rewritten as follows:
\begin{equation} \label {4:10}
 \Kin=\frac{\widetilde\p^2_0}{2\bar c_0}+\sum^n_{i=0}\frac{\widetilde\p^2_i}{2\widetilde c_i}+
 \sum_{1\le i<i'\le n}\langle\widetilde \p_i,\widetilde \p_{i'}\rangle
\end{equation}
where $\widetilde c_i:=c_0c_i/(c_0+c_i)=c_i/(1+c_i)$, $1\le i\le n$.

{\rm(B)} The total momentum $\bP=\p_0+\p_1+\ldots+\p_n$ transforms to
the momentum of the ``heaviest'' particle:
\begin{equation} \label {4:11}
 \bP=\widetilde\p_0.
\end{equation}

The function of angular momentum $I=\sum_{i=0}^n[\M_i,\p_i]$ on the
phase space $X=T^*Q$ is $L$-invariant: $I=\sum_{i=0}^n[\widetilde
\M_i,\widetilde\p_i]$.
 \end {Lem}

\begin{proof} Items A and B directly follow by substituting into the
functions $\Kin$ and $\bP$ the following explicit formulae for the
transformation $L^*$ of momenta:
 $\p_0=\frac{1}{1+c_1+\ldots+c_n}\widetilde\p_0-\widetilde\p_1-\dots-\widetilde\p_n$,
 $\p_i=\widetilde \p_i+\frac{c_i}{1+c_1+\ldots +c_n}\widetilde \p_0$,
 $1\le i\le n$.

The invariance of the angular momentum follows from its invariance
under the transformation on the space $X=T^*Q$ induced by
any linear transformation $\widetilde\M_i=\sum_ja_{ij}\M_j$ on the
space $Q$ with a nondegenerate matrix $\|a_{ij}\|$. The latter holds,
since the transformation of momenta has the form
$\widetilde\p_i=\sum_kb_{ik}\p_k$ where
$\sum_ib_{ik}a_{ij}=\delta_{jk}$, hence
 $$
\sum_{i=0}^n[\widetilde\M_i,\widetilde\p_i]=
 \sum_{i=0}^n\sum_{j=0}^n\sum_{k=0}^n[\M_j,\p_k]a_{ij}b_{ik}=
 \sum_{j=0}^n[\M_j,\p_j]=I.
 \eqno\square
 $$
\end{proof}

\subsection {Proof of the main Lemma \ref {lem3:1}} \label {subsec:lem}

Let us consider the case of a planetary system without satellites.
Observe that the transition from the coordinates $\M$ to the
coordinates $\x$ is a composition of the Poincar\'e transformation
$L$ and the homothety $\widetilde\M_i=R\x_i$, $1\le i\le N$. By
setting $\widetilde\p_i=\sqrt{\frac{\mu}{R}}\xxi_i$, $1\le i\le N$,
$\widetilde\p_0=0$, one obtains from (\ref {4:10}) the desired
expression for the kinetic energy $\Kin$. In fact, $\Kin R$ equals
$\sum^N_{i=1}\frac{\xxi^2_i}{2\widetilde m_i}+
 \mu\sum_{1\le i<j\le N}\langle\xxi_i,\xxi_j\rangle$.
Hence, in the partial case of a planetary system without satellites,
the function $H$ in (\ref {eq:HM}), (\ref {vozmHM}), (\ref {H0})
indeed equals the total energy $\Kin+U$ of the system. The symplectic
structure $d\widetilde\p\wedge d\widetilde\M=\sqrt{\mu R}d\xxi\wedge
d\x$ also has the desired form, since $\sqrt{\mu R}=\frac{1}{\omega
R}$.

In the general case of a planetary system with satellites, we observe
that a transition from the radius vectors $(\M_0,\M_1,\ldots,\M_N)$
to the coordinates $\x_i$, $\y_{ij}$ can be obtained via performing
the following three transformations. At first, one should
perform the Poincar\'e transformation $L$ to the whole system (see
above). At second, one performs the transformation $L$ to each
satellite system $\widetilde\M_{ij}$, $0\le j\le n_i$. Finally one
performs the ``scaling'' homothety
$\widetilde{\widetilde\M}_{i0}=\C_i=R\x_i$, $1\le i\le n$,
$\widetilde{\widetilde\M}_{ij}=\y_{ij}$, $1\le j\le n_i$. For the
sake of simplicity, we will assume that all planets have satellites,
i.e.\ all $n_i$ are positive.

{\it Step 1.} The first performing of the Poincar\'e transformation
$L$ to the initial configuration variables gives, due to (\ref
{4:9}),
 $$
\Kin=\frac{\widetilde\p^2_0}{2\bar c_0}+\sum^N_{i=1}\frac{\widetilde\p^2_i}{2c_i}
 +\frac{1}{2}(\widetilde\p_1+\ldots +\widetilde\p_N)^2=
 \frac{\widetilde\p^2_0}{2\bar c_0}+\sum^n_{i=1}\widetilde \Kin_i+
 \frac{1}{2}(\widetilde \bP_1+\ldots +\widetilde \bP_n)^2.
 $$
Here $\widetilde\Kin_i=\frac{\widetilde\p^2_i}{2c_i}+
 \sum^{n_i}_{j=1}\frac{\widetilde\p^2_{ij}}{2c_{ij}}=
 \frac{\widetilde\p^2_i}{2\mu m_i}+
 \sum^{n_i}_{j=1}\frac{\widetilde\p^2_{ij}}{2\mu\nu m_{ij}}$ is the
kinetic energy of the $i$th satellite system, and
$\widetilde\bP_i=\widetilde\p_i+\sum^{n_i}_{j=1}\widetilde \p_{ij}$
is its total momentum.

{\it Step 2.} Put $\widetilde\p_0=0$ and perform separately the
Poincar\'e transformation $L$ to each satellite system. As a result,
we have from the formula (\ref {4:10})
 $$
\widetilde\Kin_i=\frac{\widetilde{\widetilde\p}^2_i}{2\mu\bar m_i}+
\sum^{n_i}_{j=1}\frac{\widetilde{\widetilde\p}^2_{ij}}{2\mu\nu\widetilde m_{ij}}+\frac{1}{\mu m_i}
 \sum_{1\le j<j'\le n_i}\langle\widetilde{\widetilde\p}_{ij},
 \widetilde{\widetilde\p}_{ij'}\rangle
 $$
where $\bar m_i$, $\widetilde m_{ij}$ are as in (\ref {m}). By the
formula (\ref {4:11}), we have
$\widetilde\bP_i=\widetilde{\widetilde\p}_i$.

By using the previous step, we obtain
 $$
\Kin=\sum^n_{i=1}\left(\frac{\widetilde{\widetilde\p}^2_i}{2\mu\bar m_i}+
 \sum^{n_i}_{j=1}\frac{\widetilde{\widetilde\p}^2_{ij}}{2\mu\nu\widetilde m_{ij}}+
 \sum_{1\le j<j'\le n_i}\frac{\langle\widetilde{\widetilde\p}_{ij},
 \widetilde{\widetilde\p}_{ij'}\rangle}{\mu m_i}\right)+
 \frac{1}{2}\left(\sum^n_{i=1}\widetilde{\widetilde\p}_i\right)^2 =
 $$
 $$
=\sum^n_{i=1}\left(\frac{\widetilde{\widetilde\p}^2_i}{2\mu\widetilde m_i}+
 \sum^{n_i}_{j=1}\frac{\widetilde{\widetilde\p}^2_{ij}}{2\mu\nu\widetilde m_{ij}}+
 \sum_{1\le j<j'\le n_i}
 \frac{\langle\widetilde{\widetilde\p}_{ij},
 \widetilde{\widetilde\p}_{ij'}\rangle}{\mu m_i}\right)+
 \sum_{1\le i<i'\le n}\langle\widetilde{\widetilde\p}_i,\widetilde{\widetilde\p}_{i'}\rangle
 $$
where $\widetilde m_i=\frac{\bar m_i}{1+\mu\bar m_i}$, $1\le i\le n$.

{\it Step 3.} Now perform the following ``scaling'' of coordinates
and momenta:
 \begin{equation} \label {4:12'}
\widetilde{\widetilde\M}_i=R\x_i, \quad
\widetilde{\widetilde\M}_{ij}=\y_{ij}, \qquad
\widetilde{\widetilde\p}_i=\sqrt{\frac{\mu}{R}}\xxi_i, \quad
\widetilde{\widetilde\p}_{ij}=\mu\nu\eeta_{ij}
 \end{equation}
for $1\le i\le n$, $1\le j\le n_i$. This gives:
 $$
 \Kin=
 \sum^n_{i=1}\left(\frac{\xxi_i^2}{2\widetilde m_iR}+
 \mu\nu\sum^{n_i}_{j=1}\frac{\eeta_{ij}^2}{2\widetilde m_{ij}}+
 \frac{\mu\nu^2}{m_i}
 \sum_{1\le j<j'\le n_i}
 \langle\eeta_{ij},\eeta_{ij'}\rangle\right)+\frac{\mu}{R}\sum_{1\le i<i'\le n}
 \langle\xxi_i,\xxi_{i'}\rangle.
 $$
By taking into account that $\mu\nu\omega R=\varepsilon$, we have the
desired formula:
 $$
 \omega R\Kin=
\sum^n_{i=1}\left(\omega\frac{\xxi_i^2}{2\widetilde m_i}+
 \varepsilon\sum^{n_i}_{j=1}\frac{\eeta_{ij}^2}{2\widetilde m_{ij}}+
 \frac{\varepsilon\nu}{m_i}\sum_{1\le j<j'\le n_i}
 \langle\eeta_{ij},\eeta_{ij'}\rangle\right)+
 \omega\mu\sum_{1\le i<i'\le n}
 \langle\xxi_i,\xxi_{i'}\rangle.
 $$

Due to (\ref {4:12'}), the symplectic structure has the form
 $$
\oomega=d\p\wedge d\M=d\widetilde{\widetilde\p}\wedge d\widetilde{\widetilde\M}=
 \sqrt{\mu R}\sum^n_{i=1}\left(d\xxi_i\wedge d\x_i+
 \varepsilon\sum^{n_i}_{j=1}d\eeta_{ij}\wedge d\y_{ij}\right),
 $$
since, recall, $\varepsilon=\nu\sqrt{\frac{\mu}{R}}$. This symplectic
structure has the desired form (\ref {eq:HM}), (\ref {vozmHM}), (\ref
{H0}), since $\sqrt{\mu R}=\frac{1}{\omega R}$. In a similar way, one
proves the formulae for the first integral of angular momentum.

This finishes the proof of the main lemma \ref {lem3:1}. \qed

\section {Deriving theorems \ref {th1}--\ref {th:degen:sym} from theorems \ref {th1'}--\ref {th:nondeg}}
\label {par:proofs}

Due to lemma \ref {lem3:1} and remark \ref {rem:unpert}, the $N+1$
body problem of the type of planetary system with satellites is
equivalent to the $\varepsilon$-Hamiltonian system (\ref {eq:pert})
with small parameters $0<\omega,\varepsilon,\mu,\nu,\rho\ll1$ related
by the conditions $\varepsilon=\omega^{1/3}\mu^{2/3}\nu$ and
$\rho=\omega^{2/3}\mu^{1/3}$. Moreover the functions $\widetilde
H_0,\widetilde H_1,\widetilde\Phi$ are $S^1$-invariant, the function
$\widetilde H_0=H_0+\mu R_0$ ``projects'' to $M_0:=T^*Q_0$, the
function $\widetilde H_1=H_1+\nu R_1$ ``projects'' to $M_1:=T^*Q_1$,
and their ``principal parts'' equal the sums $H_0=\sum_{i=1}^nH_{i0}$
and $H_1=\sum_{i=1}^n\sum_{j=1}^{n_i}H_{ij}$. Furthermore, due to
lemma \ref {lem3:2:}, each summand has the form
$H_{ij}=H_{ij}(I_{ij},q_{ij},p_{ij})$ and satisfies the conditions
(\ref {eq:isoen}). The perturbation potential $\widetilde\Phi$ is an
analytic function in a neighbourhood of any torus $\Lambda\o$,
provided that the collection of angular frequencies $\Omega_{ij}$
satisfies the conditions (\ref {poryadki:chastot'}) and (\ref
{poryadki:chastot''}) of ``lack of collisions''. The principal part
$\Phi:=\widetilde\Phi|_{\mu=\nu=\rho=0}$ of the perturbation
potential has the form (\ref {eq:Phi}).

So, the $N+1$ body problem of the type of planetary system with
satellites considered in theorem \ref {th1} is equivalent to an
$\varepsilon$-Hamiltonian system belonging to the class of
``perturbed'' systems in theorem \ref {th1'}.

\begin{proof}[of theorems \ref {th1} and \ref {th:ust}] {\it Step 1.}
In theorems \ref {th1} and \ref {th1'}, the ``relative resonance''
conditions (\ref {maxrez}) and (\ref {maxrez'}) on the collection of
frequencies are equivalent. The nondegeneracy condition (\ref
{alpha}) from theorem \ref {th1} is equivalent to the nondegeneracy
condition (\ref {eq:nondeg}) from theorem \ref {th1'}.

Let us suppose that the more delicate nondegeneracy condition (\ref
{eq:fine}) from theorem \ref {th1} holds. Let us prove the
nondegeneracy conditions (\ref {eq:nondeg'}) from theorem \ref
{th1'}. The first condition in (\ref {eq:nondeg'}) is equivalent to
the first condition in (\ref {eq:fine}). In order to prove the second
condition in (\ref {eq:nondeg'}), let us evaluate the number
$\Delta_{ij}$ in (\ref {eq:Delta:ij}). Due to (\ref {f0}), we have
$F_{ij}=F_{ij}(\x_i,\y_{ij})=m_{ij}F(\x_i,\y_{ij})$. By construction,
$F_{ij}\o=(F_{ij}(\x_i,\cdot))|_{H_{ij}^{-1}(H_{ij}(I_{ij}\o,0,0))}$
where $\x_i=\const$, $|\x_i|=R_i/R=|\Omega_{i0}|^{-2/3}$ (see theorem
\ref {th1'}). Let $\langle F_{ij}\o\rangle=\langle
F_{ij}\o\rangle(q_{ij},p_{ij})$ be the function obtained by averaging
the Hill potential $F_{ij}\o=F_{ij}\o(q_{ij},p_{ij})$ along the
$\frac{2\pi}{\Omega_{ij}}$-periodic solutions of the Kepler problem
$(M_{ij},\oomega_{ij},H_{ij})$ for the satellite. By an easy
calculation, taking into account (\ref {F:xy}) and corollary \ref
{corlem3:2:}, we find the differential and the Hesse matrix of the
function $\langle F_{ij}\o\rangle$ at the point $(0,0)$:
 $$
   d\langle{F_{ij}\o}\rangle(0,0)=0, \qquad
   \frac{\partial^2\langle F_{ij}\o\rangle(0,0)}{\partial(q_{ij},p_{ij})^2}
   = \Omega_{i0}^2 \frac{I_{ij}}{\Omega_{ij}}
 \left( \begin{array}{cc} -29/8 & 0 \\
                              0 & 25/(8I_{ij}^2) \end{array}\right).
 $$
This and (\ref {d2H}) imply that
 $$
 \Delta_{ij}
 = \frac{\Omega_{ij}}2 \Tr\left(\left
 (\frac{\partial^2H_{ij}(I_{ij}\o,0,0)}{\partial(q_{ij},p_{ij})^2}\right)^{-1}
 \frac{\partial^2\langle F_{ij}\o\rangle(0,0)}{\partial(q_{ij},p_{ij})^2}\right)
 =
 $$
 $$
 = \frac{\Omega_{i0}^2}{2\Omega_{ij}}
 \Tr\left(
 \left( \begin{array}{cc} 1 & 0 \\
                          0 & I_{ij}^2 \end{array}\right)
 \left( \begin{array}{cc} -29/8 & 0 \\
                              0 & 25/(8I_{ij}^2) \end{array}\right)
 \right)
 = - \frac{\Omega_{i0}^2}{4\Omega_{ij}}.
 $$
Therefore the second desired condition in (\ref {eq:nondeg'}) has the
form
 $$
 \alpha+\Delta_{ij}\omega^2T = \alpha-\frac{\fomega_i^2}{4\Omega_{ij}}T \not\in\left[-C_2\omega^3T,C_2\omega^3T\right]+2\pi\bbZ,
 $$
i.e.\ it is equivalent to the second condition in (\ref {eq:fine})
with the constant $C:=C_2$.

Thus, all the conditions of theorem \ref {th1'} are fulfilled. Hence
this theorem implies theorem \ref {th:ust}(A).

{\it Step 2.} Let us check that, for the $N+1$ body problem under
consideration, the model problem, the unperturbed and the perturbed
problems are reversible. The construction of the functions
$\widetilde H$, $\widetilde H_0$, $\widetilde H_1$ shows that they
(and, hence, also $\widetilde\Phi$) are invariant under each of the
involutions $S_l$ and $S$ from \Sec\ref {par:sym}. Hence they are
invariant under the composition $\J=S_lS=SS_l$. Due to (\ref {0}), the
involution $\J=S_lS=SS_l$ acts component-wise in the form
$(\varphi_{ij},I_{ij},q_{ij},p_{ij})\mapsto(-\varphi_{ij},I_{ij},q_{ij},-p_{ij})$.
From here, by taking into account the $\J$-invariance of the functions
$\widetilde H_0$, $\widetilde H_1$, $\widetilde\Phi$, we obtain the
reversibility of the model system, of the unperturbed and the perturbed
systems. Hence theorem \ref {th:sym} implies theorem \ref {th1}.

{\it Step 3.} Let us derive theorem \ref {th:ust}(B) from theorem
\ref {th:stab}. By lemma \ref {lem3:2:} or (\ref {d2H}), all numbers
$\frac{\partial^2H_{ij}}{\partial I_{ij}^2}(I_{ij}\o,0,0)$ are
negative. The sign $\eta_{ij}$ from (\ref {eq:eta:ij}) equals
 $$
 \eta_{ij}
 ={\rm sgn\,}\left(\Omega_{ij}\Tr\frac{\partial^2H_{ij}(I_{ij}\o,0,0)}{\partial(q_{ij},p_{ij})^2}\right)
 = {\rm sgn\,}\Omega_{ij}.
 $$
Hence, by the first property of having fixed sign in theorem \ref
{th:ust}, all the signs $\eta_{i0}={\rm sgn\,}\Omega_{i0}$ are the
same, moreover $\eta_{ij}\Delta_{ij}<0$ for $1\le j\le n_i$. Suppose
that the conditions (\ref {eq:fine'a}) and (\ref {eq:fine'b}) hold
for $C:=C_2$. Then $\alpha\not\in\pi\bbZ$ and
 \begin{equation} \label {eq:fine'a'} 
\frac{\eta_{10}+\eta_{ij}}2\alpha+\frac{\eta_{ij}\Delta_{ij}}2\omega^2T\not\in\left[-\frac{C_2}2\omega^3T,\frac{C_2}2\omega^3T\right]+\pi\bbZ, 
 \end{equation}
 \begin{equation} \label {eq:fine'b'} 
\frac{\eta_{ij}+\eta_{i'j'}}2\alpha+\frac{\eta_{ij}\Delta_{ij}+\eta_{i'j'}\Delta_{i'j'}}2\omega^2T\not\in\left[-C_2\omega^3T,C_2\omega^3T\right]+\pi\bbZ
 \end{equation}
for $1\le j\le n_i$ and $1\le j'\le n_{i'}$. Consider any collection
of real numbers $\alpha_{ij}$, $1\le i\le n$, $0\le j\le n_i$, such
that
 $$
 \alpha_{i0}=\eta_{i0}\alpha, \
 |\alpha_{ij}-\eta_{ij}(\alpha+\Delta_{ij}\omega^2T)|\le C_2\omega^3T,
 \quad 1\le i\le n, \ 1\le j\le n_i.
 $$
Then:

1) the sum $\alpha_{i0}+\alpha_{i'0}=2\eta_{10}\alpha$ does not
belong to $2\pi\bbZ$, since $\alpha\not\in\pi\bbZ$;

2) for $1\le j\le n_i$, the sum
$\alpha_{i'0}+\alpha_{ij}\in(\eta_{10}+\eta_{ij})\alpha+\eta_{ij}\Delta_{ij}\omega^2T+[-C_2\omega^3T,C_2\omega^3T]$
does not belong to $2\pi\bbZ$ because of (\ref {eq:fine'a'}); 

3) for $1\le j\le n_i$ and $1\le j'\le n_{i'}$, the sum
$\alpha_{ij}+\alpha_{i'j'}\in(\eta_{ij}+\eta_{i'j'})\alpha+(\eta_{ij}\Delta_{ij}+\eta_{i'j'}\Delta_{i'j'})\omega^2T+[-2C_2\omega^3T,2C_2\omega^3T]$
does not belong to $2\pi\bbZ$ because of (\ref {eq:fine'b'}). 

Thus, the sum of any two, possibly coinciding, numbers of the set
$\alpha_{ij}$ does not belong to the set $2\pi\bbZ$. Hence, the
hypothesis of theorem \ref {th:stab} holds, and therefore this
theorem implies theorem \ref {th:ust}(B).
 \qed
\end{proof}

\subsection {Existence of ``gaps'' in the families of relatively-periodic solutions}
\label {par3:1:3'}

In this section, we give a more exact definition of the notion
``almost any'' in theorem~\ref {th:degen:sym} (meaning ``any'' for
$N\ge n=2$, see corollary \ref {cor:3bodies}($\nexists$)) and of the
subsets $\cM\subset\cM^{\rm sym}\subset\bbR^n_{>0}$. Besides we
derive theorem \ref {th:degen:sym} and corollary \ref
{cor:3bodies}($\nexists$) from theorem \ref {th:nondeg}.

Consider the sequence of {\it positive} real numbers
 \begin{equation} \label {cs}
c_\kappa:=\sqrt{r_\kappa} \left(\left(\frac12-\frac1{4\kappa}\right) A_\kappa + \left(\frac54r_\kappa+\frac1{4r_\kappa}-\frac32\right) \frac{B_\kappa}\kappa \right), \quad\mbox{где}\quad
 r_\kappa:=\left(\frac{\kappa+1}{\kappa}\right)^{2/3},
 \end{equation}
$$
A_\kappa := \frac{1}{2\pi}\int\limits_{-\pi}^{\pi} \frac{\cos(\kappa t)d t} {\sqrt{r_\kappa+\frac{1}{r_\kappa}-2\cos t}}, \quad 
B_\kappa := \frac{1}{2\pi}\int\limits_{-\pi}^{\pi} \frac{\cos(\kappa t)d t} {(r_\kappa+\frac{1}{r_\kappa}-2\cos t)^{3/2}}, 
$$
$\kappa\in\bbZ\setminus\{-1,0\}$. It will interest us only up to a
nonzero multiplicative factor. 
The following properties of the sequence $c_\kappa$ and its entries 
can be used for approximate computations:
 $$
c_\kappa = \frac{1}{4\pi} A + \frac{1}{3\pi} B + o(1), \quad
A_\kappa = \frac{1}{2\pi} A + o(1), \quad 
B_\kappa = \frac{\kappa^2}{2\pi} B + o(\kappa^2), 
\qquad
|\kappa|\to\infty,
 $$
and $\frac54r_\kappa+\frac1{4r_\kappa}-\frac32 = \frac{2}{3\kappa}-\frac{5}{81\kappa^3}+O(\frac{1}{\kappa^4})$, where
 $$
A := 
\int\limits_{-\infty}^{\infty}\frac{\cos u\,du}{\sqrt{\frac{4}{9}+u^2}}, \qquad
B := 
\int\limits_{-\infty}^{\infty}\frac{\cos u\,du}{(\frac{4}{9}+u^2)^{3/2}}>0.
 $$

Put
\begin{equation} \label {sii'}
\kappa_{ii'}:=\frac{\fomega_i}{\fomega_{i'}-\fomega_i}, \qquad i\ne i', \quad
1\le i,i'\le n.
 \end{equation}
Then
$\frac{\fomega_{i'}}{\fomega_i}=\frac{\kappa_{ii'}+1}{\kappa_{ii'}}$,
$\kappa_{i'i}=-\kappa_{ii'}-1$, thus the numbers $\kappa_{i'i}$ and
$\kappa_{ii'}$ are either both integer or both non-integer, moreover
$\kappa:=\kappa_{ii'}\in\bbR\setminus\{0,-1\}$. The number
$\frac{\kappa+1}{\kappa}$ equals the ratio
$\frac{\fomega_{i'}}{\fomega_i}$ of the angular frequencies of two
planets along the circular orbits. Hence, due to Kepler's second 
law, the number $r_\kappa=(\frac{\kappa+1}{\kappa})^{2/3}$ in (\ref
{cs}) equals the ratio of the radii of these orbits. Let us define
the number $c_\kappa\in\bbR$ for any $\kappa\in\bbR$ as follows:
either by the formula (\ref {cs}) if
$\kappa\in\bbZ\setminus\{-1,0\}$, or by the formula
 \begin{equation} \label {eq:cs0}
c_\kappa:=0 \qquad \mbox{if}\qquad
\kappa\in\bbR\setminus(\bbZ\setminus\{-1,0\}).
 \end{equation}
Consider the following collection of complex-valued functions on the
$n$-dimensional torus $(S^1)^n$ with angular coordinates
$\vvarphi=(\varphi_1,\dots,\varphi_n)$:
 \begin{equation} \label {fl}
f_{ll'}(\vvarphi):=\kappa_{ll'}c_{\kappa_{ll'}}e^{i\kappa_{ll'}(\varphi_{l'}-\varphi_l)},
 \qquad l\ne l', \quad 1\le l,l'\le n,
 \end{equation}
where $i=\sqrt{-1}\in\bbC$ is the imaginary unit and $c_\kappa\ge0$
is defined in (\ref {cs}) and (\ref {eq:cs0}).

The functions~(\ref {fl}) are $2\pi$-periodic in each argument,
moreover they are equivariant:
 $$
f_{ll'}(\varphi_1+t,\dots,\varphi_{n}+t)=f_{ll'}(\varphi_1,\dots,\varphi_{n}),
\qquad t\in\bbR,
 $$
 $$
f_{ll'}(\varphi_1+\fomega_1t,\dots,\varphi_{n}+\fomega_{n}t)=
 e^{i\fomega_lt}f_{ll'}(\varphi_1,\dots,\varphi_{n}), \qquad t\in\bbR.
 $$

\begin{scDef} \label {def:m:adm}
Let us fix the angular frequencies $\fomega_i$ of the planets
satisfying the properties (\ref {maxrez}), (\ref {parametry}), (\ref
{poryadki:chastot'}), (\ref {poryadki:chastot''}). A collection of planets'
masses $\mu(m_1,\dots,m_n)\in\bbR^n_{>0}$ will be called
{\it unclosing for a phase point
$\vvarphi=(\varphi_1,\dots,\varphi_n)\in(S^1)^n$}, or simply {\it
$\vvarphi$-unclosing}\footnote{\rm A collection of masses
$\mu(m_1,\dots,m_n)\in\bbR^n_{>0}$ is {\it $\vvarphi$-closing} if and
only if, for any index $l=1,\dots,n$, the planar polygonal line
$A_{l1}(\vvarphi)\ldots A_{ln}(\vvarphi)\subset\bbC$ is closed,
provided that the segments of this polygonal line have the form
$\overrightarrow{A_{ll'}(\vvarphi)A_{l,l'+1}(\vvarphi)}=m_{l'}f_{ll'}(\vvarphi)$
for $1\le l'\le l-1$ and the form
$\overrightarrow{A_{l,l'-1}(\vvarphi)A_{ll'}(\vvarphi)}=m_{l'}f_{ll'}(\vvarphi)$
for $l+1\le l'\le n$.}, if at least one of the following numbers does
not vanish:
 \begin{equation} \label{eq:fl}
 f_l(\vvarphi;m_1,\dots,m_n):=\sum_{\substack{l'=1\\l'\ne l}}^nm_{l'}f_{ll'}(\vvarphi)\in\bbC, \qquad 1\le l\le n,
 \end{equation}
see (\ref {cs})--(\ref {fl}). A phase point $\vvarphi\in(S^1)^n$ will be called
{\it symmetric} if it is fixed under the involution
$(S^1)^n\to(S^1)^n$, $\vvarphi\mapsto-\vvarphi$, of the
$n$-dimensional torus, i.e.\ its coordinates have the form
$\varphi_l\in\{0,\pi\}\mod2\pi$, $1\le l\le n$. Denote by $\cM$
(respectively $\cM^{\rm sym}$) the set of all collections of planets' masses
$\mu(m_1,\dots,m_n)\in\bbR^n_{>0}$ that are
$\vvarphi$-unclosing for any (respectively for any symmetric) phase
point $\vvarphi\in(S^1)^n$.
\end{scDef}

\begin{scRem} \label {rem:sym}
The number of symmetric phase points equals $2^n$. For any phase
point $\vvarphi\in(S^1)^n$, the set of $\vvarphi$-closing
collections of planets' masses is the intersection of a linear
subspace of $\bbR^n$ with $\bbR_{>0}^n$. Hence $\cM^{\rm sym}$ is
open in $\bbR^n_{>0}$. Moreover it is dense in $\bbR^n_{>0}$ whenever
it is nonempty. The subset $\cM\subset\bbR_{>0}^n$ is open in
$\bbR_{>0}^n$, since the torus $(S^1)^n$ is compact and the functions
$f_l=f_l(\vvarphi;m_1,\dots,m_n)$ are continuous, see (\ref {eq:fl}).
Suppose that $\kappa_{ii'}\in\bbZ$ for some $i\ne i'$, $1\le i,i'\le
n$. Then the set of $\vvarphi$-closing collections of planets' masses 
is contained in a plane of codimension $\ge2$ (since the
system of functions (\ref {eq:fl}) is linear in $(m_1,\dots,m_n)$ and
its rank is at least 2). Moreover any collection of masses with
$|\kappa_{ii'}|c_{\kappa_{ii'}}m_{i'}>\sum_{l\ne
i,i'}|\kappa_{il}|c_{\kappa_{il}}m_{l}$ is $\vvarphi$-unclosing for
any phase point $\vvarphi\in(S^1)^n$ (i.e.\ belongs to $\cM$). Hence
the open sets $\cM\subset\cM^{\rm sym}$ are nonempty, thus ``almost
any'' collection of planets' masses (see the paragraph before
theorem \ref {th:degen:sym}) belongs to $\cM^{\rm sym}$.
\end{scRem}

Consider the natural angular coordinates on the torus $\Lambda\o$:
 $$
 \pnt \mapsto \{\varphi_l,\ \varphi_{lj},\ 1\le l\le n,\ 1\le j\le n_l\}, \qquad \pnt\in\Lambda\o.
 $$
The following statement generalizes theorem \ref {th:degen:sym}.

\begin {Pro} \label {pro:degen:planets}
Consider the $N+1$ body problem of the type of planetary system with
\(or without\) satellites, $N\ge n\ge2$. Under the hypothesis of
theorem {\rm\ref {th1}}, fix the angular frequencies
$\fomega_1,\dots,\fomega_n$ of planets having the form {\rm(\ref
{maxrez}), (\ref {parametry}), (\ref {poryadki:chastot'}), (\ref
{poryadki:chastot''})}. Suppose that there exists at least one pair
of planets with indices $i\ne i'$, whose frequencies are in a special
resonance {\rm(\ref {degen:3bod'})}. In this case, one automatically
has $\alpha=0$, $\kappa_{ii'}\in\bbZ\setminus\{0,-1\}$ and
$c_{\kappa_{ii'}}>0$. Fix the two-dimensional torus
$\gamma\subset\Lambda\o$ corresponding to a $T$-periodic solution of
the model system. Let us suppose that the collection of planets' masses 
$\mu(m_1,\dots,m_n)\in\bbR_{>0}^n$ is $\vvarphi$-unclosing
for some \(and, hence, any\) point
$x=\{\varphi_l,\varphi_{lj}\}\in\gamma$, {\rm see (\ref {eq:fl})}.

Then, for any real number $D>0$, there exist numbers $\mu_0,\nu_0>0$
and a neighbourhood $U_0$ of the projection of the two-dimensional
torus $\gamma$ to the phase space of planets such that the following
property holds. For any values $\mu,\nu$ such that
$0<(\frac{\nu}{\nu_0})^3\le\mu\le\mu_0$ \(respectively
$0<\mu\le\mu_0$ if there are no satellites\), the direct product $U$
of the neighbourhood $U_0$ and the phase space of satellites does not
contain any $(\widetilde T,\widetilde\alpha)$-periodic trajectory of
the $N+1$ body problem under consideration, provided that the
parameters $\widetilde T$, $\widetilde\alpha$ have the form
 $$
|\widetilde T-T|+|\widetilde\alpha|\le D\mu.
 $$
In particular, the assertions of theorem {\rm\ref {th:degen:sym}}
hold.
 \end {Pro}

\subsubsection {Система Солнце--две планеты} \label {subsec:3body:}

Let us show that theorem \ref {th:nondeg} (or the averaging method on a submanifold \cite[theorem 11.1]{9}, \cite[theorem 4]{29}) implies corollary \ref
{cor:3bodies}($\nexists$) on non-existence of periodic solutions of a planetary system with two planets ($N=n=2$). 
The Hamilton function and the symplectic structure of the perturbed problem are
 $$
 \somega\widetilde H_0=
 \somega\left(\widetilde K_1+\widetilde K_2+\mu K_{12}\right), \qquad
 \oomega_0=d\xxi_1\wedge d\x_1+d\xxi_2\wedge d\x_2.
 $$
Here $\widetilde K_i=\frac{\xxi_i^2}{2\widetilde m_i}-\frac{\bar
m_i}{|\x_i|}$ is the Hamiltonian function of the Kepler problem
corresponding to the $i$th planet, $\widetilde m_i=\frac{\bar
m_i}{1+\mu\bar m_i}$, $i=1,2$,
$K_{12}=\langle\xxi_1,\xxi_2\rangle-\frac{\bar m_1\bar
m_2}{|\x_1-\x_2|}$.

Suppose that the ratio of frequencies $\fomega_1$, $\fomega_2$ is
rational, i.e.\ they have the form
 $$
\fomega_1=\frac{2\pi k_1}{T}, \quad \fomega_2=\frac{2\pi k_2}{T}
 $$
where $k_1$, $k_2$ are nonzero integers, $k_1\ne\pm k_2$, $T>0$. Then
the solutions of the unperturbed problem ($\mu=0$) corresponding to
the independent circular motions of planets with angular frequencies
$\fomega_1$, $\fomega_2$ are $T$-periodic. Conversely, if the solution is
$T$-periodic then the pair of angular frequencies (\ref {chast}) is
proportional to a pair of integers, with coefficient $\frac{2\pi}{T}$.

\begin{proof}[of corollary \ref {cor:3bodies}($\nexists$)]
{\it Step 1.} Recall that the unperturbed problem (corresponding to
$\mu=0$) splits into two independent planar Kepler's problem. Hence
its $T$-periodic phase trajectories form the six-dimensional
submanifold
 $$
\Theta=\{K_1=\const,\ K_2=\const\}
 $$
in the eight-dimensional phase space. In fact, due to periodicity of
solutions of the Kepler problem with negative energy levels, the period
of any its closed trajectory is a smooth (and strictly monotone, see
\S\ref {par3:1:5}) function in the value of energy.

Let us find the averaged perturbation $\langle R_0\o\rangle$, i.e.\
the function obtained by averaging the perturbation function
$R_0\o=R_0|_\Theta=(K_{12}+\frac12\xxi^2)|_\Theta$ along the periodic
solutions of the unperturbed problem. In more detail, let us find the
differential of the function $\langle R_0\o\rangle$ at any point of
the torus $\Lambda\o\subset\Theta$.

{\it Step 2.} With respect to polar coordinates $\psi$, $r$ on the
plane of motion, we have
 $$
\widetilde K_i=\frac{p_{r_i}^2+p_{\psi_i}^2/r_i^2}{2\widetilde m_i}-
  \frac{\bar m_i}{r_i}, \qquad i=1,2,
 $$
 $$
K_{12}=
 \left(p_{r_1}p_{r_2}+\frac{p_{\psi_1}p_{\psi_2}}{r_1r_2}\right)\cos(\psi_1-\psi_2)+
 $$
 $$
+ \left(p_{r_1}\frac{p_{\psi_2}}{r_2}-\frac{p_{\psi_1}}{r_1}p_{r_2}\right)\sin(\psi_1-\psi_2)-
 \frac{\bar m_1\bar m_2}{\sqrt{r_1^2+r_2^2-2r_1r_2\cos(\psi_1-\psi_2)}}.
 $$

{\it Step 3.} Let us transfer to the coordinates
$\varphi_i,I_i,q_i,p_i$, $i=1,2$, from lemma~\ref {lem3:2:}. By this
lemma,
 $$
\Lambda\o=\{\,p_1=p_2=q_1=q_2=0,\ I_1=\const,\ I_2=\const\,\},
 $$
moreover the restriction of the linearized unperturbed system to the
tangent bundle $T_{\Lambda\o}\Theta=\{\xi\mid
dI_1(\xi)=dI_2(\xi)=0\}$ to $\Theta$ has the following form (with
respect to the coordinates $\varphi_i$, $d\varphi_i$, $dq_i$, $dp_i$,
$i=1,2$, on this bundle):
 \begin{equation} \label {sist}
 \frac{d\varphi_i}{dt}=\fomega_i, \quad
 \frac{d(d\varphi_i)}{dt}=0, \qquad
 \frac{d(dq_i)}{dt}=\fomega_i\frac{dp_i}{I_i}, \quad
 \frac1{I_i}\frac{d({dp_i})}{dt}=-\fomega_idq_i,
 \end{equation}
$i=1,2$. Hence, at each point $(\varphi_1,\varphi_2)$ of the torus
$\Lambda\o$, the differential of the function $K_{12}|_\Theta$ has
the following form:
 $$
 d (K_{12} |_\Theta)(\varphi_1,\varphi_2)=
 \frac{\bar m_1\bar m_2r_1r_2}{r_{12}^3}
 \left(
 \sin\varphi_{12}\left(d\varphi_{12}+2\frac{dp_1}{I_1}-2\frac{dp_2}{I_2}\right)+
 \right.
 $$
 $$
 + \left.\left(\frac{r_1}{r_2}-\cos\varphi_{12}\right)dq_1+
 \left(\frac{r_2}{r_1}-\cos\varphi_{12}\right)dq_2 \right) -
 $$
 \begin{equation} \label {vozm:syst}
 - \frac{I_1I_2}{r_1r_2}
 \left(
 \cos\varphi_{12}(dq_1+dq_2)+
 \sin\varphi_{12}\left(d\varphi_{12}+\frac{dp_1}{I_1}-\frac{dp_2}{I_2}\right)
 \right)
 \end{equation}
where $\varphi_{12}:=\varphi_1-\varphi_2$,
$r_{12}:=\sqrt{r_1^2+r_2^2-2r_1r_2\cos\varphi_{12}}$. The
perturbation function has the form $R_0\o=
R_0|_\Theta=(K_{12}+\frac12\xxi^2 )|_\Theta$. One easily shows that
the contribution of the summand
$\frac12\xxi^2|_\Theta=\frac{1}{2}\sum_{i=1}^2(p_{r_i}^2+p_{\psi_i}^2/r_i^2))|_\Theta$
to the averaged perturbation $\langle R_0\o\rangle$ has a trivial
differential at any point $(\varphi_1,\varphi_2)\in\Lambda\o$, i.e.\
$d\langle R_0\o\rangle(\varphi_1,\varphi_2)=d\langle
K_{12}|_\Theta\rangle(\varphi_1,\varphi_2)$.

{\it Step 4.} Consider the rational number
$\kappa=\kappa_{12}=\frac{k_1}{k_2-k_1}=\frac{\fomega_1}{\fomega_2-\fomega_1}$,
see (\ref {sii'}). If $\kappa\in\bbZ$ then define the number
$c_\kappa$ by the formula (\ref {cs}).

By integrating the values of the co-vector $d(K_{12}|_\Theta)$ (see
(\ref {vozm:syst})) on the solutions of the linearized system (\ref
{sist}), one obtains the following:

1) if $\kappa\in\bbQ\setminus\bbZ$ then the differential of the
function $\langle R_0\o\rangle$ vanishes at any point of the
two-dimensional torus $\Lambda\o$;

2) if $\kappa\in\bbZ$ then this differential has the following form
at any point $(\varphi_1,\varphi_2)\in\Lambda\o$:
 $$
d \langle R_0\o\rangle (\varphi_1,\varphi_2)=
 {\bar m_1\bar m_2}
 \left(
 \frac{\kappa c_\kappa}{r_1}
 \left( \cos(\kappa\varphi_{12})dq_1+\sin(\kappa\varphi_{12})\frac{dp_1}{I_1}
 \right) - \right.
 $$
 \begin{equation} \label {Husr}
 \left.
 - \frac{(\kappa+1)c_{-\kappa-1}}{r_2}
 \left( \cos((\kappa+1)\varphi_{12})dq_2-\sin((\kappa+1)\varphi_{12})\frac{dp_2}{I_2}
 \right) \right)
 \end{equation}
and, hence, it does not vanish (since $c_\kappa\ne0$, $c_{-\kappa-1}\ne0$, see~(\ref {cs})).

Since any point of the torus $\Lambda\o$ is noncritical for the
function $\langle R_0\o\rangle$, Theorem \ref {th:nondeg} (of the averaging method on a submanifold \cite[theorem 11.1]{9}, \cite[theorem 4]{29}) implies
Corollary \ref {cor:3bodies}($\nexists$).
 \qed
\end{proof}

\subsubsection {The general case of a planetary system with satellites} \label {subsec:any}

Let us derive Proposition \ref {pro:degen:planets} from Theorem \ref {th:nondeg}. 
By lemma \ref {lem3:1}, the summands $\widetilde H_0,\oomega_0$ in
the decom\-po\-si\-tions (\ref {vozmHM}) of the Hamiltonian function and
the symplectic structure of the $N+1$ body problem have the form
 $$
 \widetilde H_0=
 \sum_{i=1}^{n}\widetilde K_i+\mu\sum_{1\le i<i'\le n}K_{ii'}, \qquad
 \oomega_0=\sum_{i=1}^{n}d\xxi_i\wedge d\x_i.
 $$
Here $\widetilde K_i$ is the Hamiltonian function of the Kepler problem
of the $i$th planet that is similar to the Hamiltonian function
$\widetilde K_1$, and $K_{ii'}$ is the function similar to $K_{12}$.
Put $\kappa_{ii'}=\frac{\fomega_i}{\fomega_{i'}-\fomega_i}$, $1\le
i,i'\le n$, $i\ne i'$, see (\ref {sii'}).

From the equality (\ref {Husr}) in the case $N=n=2$, we immediately
obtain the analogous formula in general case $N\ge n\ge2$:
 $$
d \langle R_0\o\rangle|_{\Lambda\o}=d \langle R_0\o\rangle(\varphi_1,\dots,\varphi_{n})=
 \sum_{1\le i<i'\le n}\xi^*_{ii'}(\varphi_i,\varphi_{i'}).
 $$
Here $\xi^*_{ii'}(\varphi_i,\varphi_{i'})$ is the co-vector 
analogous to the co-vector (\ref {Husr}), which is denoted by
$\xi^*_{12}(\varphi_1,\varphi_2)$. In more detail, we have
 $$
 d \langle R_0\o\rangle|_{\Lambda\o}=
 \sum_{i=1}^{n}
 \frac{\bar m_i}{r_i}
 \sum_{\substack{i'=1\\i'\ne i}}^n
 \bar m_{i'}\kappa_{ii'} c_{\kappa_{ii'}}
 \left(
 \cos(\kappa_{ii'}\varphi_{ii'})dq_i-
 \sin(\kappa_{ii'}\varphi_{ii'})\frac{dp_i}{I_i}
 \right)
 $$
where $\varphi_{ii'}:=\varphi_i-\varphi_{i'}$. This co-vector
vanishes at those points of the torus $\Lambda\o$ where the functions
$f_l$, $1\le l\le n$, simultaneously vanish, see (\ref {eq:fl}).
Hence theorem \ref {th:nondeg} implies the absence of $(\widetilde
T,\widetilde\alpha)$-periodic solutions in some neighbourhood of the
torus $\gamma$, provided that the parameters
$\omega,\mu,\varepsilon>0$ are small enough and are related by the
inequalities $\omega\varepsilon/\mu_0\le\mu\le\mu_0$. Due to the
relation $\varepsilon=\omega^{1/3}\mu^{2/3}\nu$, these inequalities
have the form $\omega^{4/3}\mu^{2/3}\nu/\mu_0\le\mu\le\mu_0$, i.e.\
the form $\omega^{4}(\nu/\mu_0)^3\le\mu\le\mu_0$. Therefore theorem
\ref {th:nondeg} indeed implies proposition \ref {pro:degen:planets}.
 \qed

Theorem \ref {th:degen:sym} obviously follows from Proposition \ref
{pro:degen:planets}, Definition \ref {def:m:adm} of the subsets
$\cM\subset{\cM}^{\rm sym}\subset\bbR^n_{>0}$, and Remark \ref
{rem:sym}.

\end{fulltext}

\renewcommand {\refname}{References}

\end{document}